\documentclass[12pt,reqno]{amsart}
\usepackage[a4paper,
  left=25mm,
  right=25mm,
  top=25mm,
  bottom=25mm
]{geometry}

\usepackage{amsmath,amsfonts,amssymb, amsthm}
\usepackage{bm}
\usepackage{graphicx}
\usepackage{ascmac}
\usepackage[legacycolonsymbols]{mathtools}
\usepackage{nccmath} % needed!
\usepackage{mathrsfs}
\usepackage{mleftright}

\usepackage[italicdiff]{physics}

\usepackage{thmtools}
\usepackage{thm-restate} % use this?
\usepackage{needspace}
\usepackage{etoolbox}
\usepackage{xspace}
\usepackage{placeins}
\usepackage{hyperref}
\usepackage{multirow}

\usepackage{enumitem}
\setlist[enumerate]{label=(\arabic*), ref=(\arabic*), leftmargin=2.63em, labelsep=0.35em}
\setlist[itemize]{leftmargin=2em, labelsep=0.4em}
\mleftright
\allowdisplaybreaks
\mathtoolsset{showonlyrefs=true}
\def\thefootnote{\arabic{footnote}}

\DeclareMathAlphabet{\mathpzc}{OT1}{pzc}{m}{it}

\makeatletter

\newif\if@noindentafterheading
\let\orig@afterheading\@afterheading
\def\@afterheading{%
  \if@noindentafterheading\global\@afterindentfalse\fi
  \orig@afterheading
}

\renewcommand\section{\@startsection{section}{1}{\z@}%
  {-3.5ex \@plus -1ex \@minus -.2ex}% before skip
  {2.6ex \@plus .2ex}% after skip
  {\normalfont\large\bfseries}%
}

\renewcommand\subsection{\@startsection{subsection}{2}{\z@}%
  {-2.5ex \@plus -.3ex \@minus -.2ex}%
  {1.5ex \@plus .2ex}%
  {\normalfont\bfseries}%
}

\let\orig@section\section
\renewcommand\section{%
  \@ifstar{\section@star}{\section@nostar}%
}
\newcommand\section@nostar{%
  \@ifnextchar[{\section@opt}{\section@noopt}%
}
\newcommand\section@opt[2][]{%
  \global\@noindentafterheadingtrue
  \orig@section[#1]{#2}%
  \global\@noindentafterheadingfalse
}
\newcommand\section@noopt[1]{%
  \global\@noindentafterheadingtrue
  \orig@section{#1}%
  \global\@noindentafterheadingfalse
}
\newcommand\section@star[1]{%
  \global\@noindentafterheadingtrue
  \orig@section*{#1}%
  \global\@noindentafterheadingfalse
}

\let\orig@subsection\subsection
\renewcommand\subsection{%
  \@ifstar{\subsection@star}{\subsection@nostar}%
}
\newcommand\subsection@nostar{%
  \@ifnextchar[{\subsection@opt}{\subsection@noopt}%
}
\newcommand\subsection@opt[2][]{%
  \global\@noindentafterheadingtrue
  \orig@subsection[#1]{#2}%
  \global\@noindentafterheadingfalse
}
\newcommand\subsection@noopt[1]{%
  \global\@noindentafterheadingtrue
  \orig@subsection{#1}%
  \global\@noindentafterheadingfalse
}
\newcommand\subsection@star[1]{%
  \global\@noindentafterheadingtrue
  \orig@subsection*{#1}%
  \global\@noindentafterheadingfalse
}

\renewcommand\subsubsection{\@startsection{subsubsection}{3}{\z@}%
  {-2.0ex \@plus -.3ex \@minus -.2ex}% before skip
  {1.0ex \@plus .2ex}% after skip
  {\normalfont\bfseries}% font
}

\let\orig@subsubsection\subsubsection
\renewcommand\subsubsection{%
  \@ifstar{\subsubsection@star}{\subsubsection@nostar}%
}
\newcommand\subsubsection@nostar{%
  \@ifnextchar[{\subsubsection@opt}{\subsubsection@noopt}%
}
\newcommand\subsubsection@opt[2][]{%
  \global\@noindentafterheadingtrue
  \orig@subsubsection[#1]{#2}%
  \global\@noindentafterheadingfalse
}
\newcommand\subsubsection@noopt[1]{%
  \global\@noindentafterheadingtrue
  \orig@subsubsection{#1}%
  \global\@noindentafterheadingfalse
}
\newcommand\subsubsection@star[1]{%
  \global\@noindentafterheadingtrue
  \orig@subsubsection*{#1}%
  \global\@noindentafterheadingfalse
}

\makeatother
\newcommand{\mbE}{\mathbb{E}}
\newcommand{\mbG}{\mathbb{G}}
\newcommand{\mbP}{\mathbb{P}}

\newcommand{\msD}{\mathscr{D}}

\newcommand{\Z}{\mathbb{Z}}								% integer
\newcommand{\Zz}{\mathbb{Z}_{\geq 0}}						% non-negative integer
\newcommand{\Zp}{\mathbb{Z}_{> 0}}						% positive integer
\newcommand{\Q}{\mathbb{Q}}								% rational number
\newcommand{\Qb}{\overline{\mathbb{Q}}}					% algebraic number
\newcommand{\C}{\mathbb{C}}								% complex number
\newcommand{\Ch}{\widehat{\mathbb{C}}}					% Riemann sphere
\newcommand{\Pj}{\mathbb{P}}								% projective space

\newcommand{\A}{\alpha}									% 
\newcommand{\B}{\beta}									% 
\newcommand{\la}{\lambda}								% 
\newcommand{\vp}{\varphi}								% 
\newcommand{\OL}{\vspace{5mm}}							% 
\newcommand{\HL}{\vspace{2mm}}							% 
\newcommand{\q}{\quad}									% 1em

\newcommand{\f}[2]{\frac{#1}{#2}}							% 
\newcommand{\npmod}[1]{\!\!\!\! \pmod{#1}}					%
\newcommand{\ceq}{\coloneqq} 							% :=
\newcommand{\Ra}{\Rightarrow} 							% 
\newcommand{\Lra}{\Leftrightarrow} 							% 
\newcommand{\relmiddle}[1]{\mathrel{}\middle#1\mathrel{}}		% 
\newcommand{\ie}{i.e.\ }									% i.e.
\newcommand{\eg}{e.g.\ }									% e.g.

\newcommand{\id}{\mathop{\mathrm{id}}\nolimits}				% id
\newcommand{\Aut}{\mathop{\mathrm{Aut}}\nolimits}			% Aut
\newcommand{\Gal}{\mathop{\mathrm{Gal}}\nolimits}			% Gal
\newcommand{\ord}{\mathop{\mathrm{ord}}\nolimits}				% ord
\newcommand{\Sym}{\mathop{\mathrm{Sym}}\nolimits}			% Sym

\pretocmd{\section}{\needspace{4\baselineskip}}{}{}
\pretocmd{\subsection}{\needspace{3\baselineskip}}{}{}
\pretocmd{\subsubsection}{\needspace{2\baselineskip}}{}{}

\newtheoremstyle{mythmstyle}  % style name
  {12pt}   % space above
  {12pt}   % space below: increasing this value automatically adds vertical spacing
  {\itshape} % body font
  {}      % indentation
  {\bfseries} % theorem heading font
  {.}     % punctuation after the heading
  { }     % space after the heading
  {}      % theorem heading specification

\newtheoremstyle{mydefstyle}  % style name
  {12pt}   % space above
  {12pt}   % space below: increasing this value automatically adds vertical spacing
  {}       % body font
  {}       % indentation
  {\bfseries} % heading font
  {.}      % punctuation after the heading
  { }      % space after the heading
  {}       % theorem heading specification

\newtheoremstyle{myremarkstyle} % style name
  {12pt}   % space above
  {12pt}   % space below
  {}       % body font (roman)
  {}       % indentation
  {\normalfont} % heading font
  {.}      % punctuation after heading
  { }      % space after heading
  {}       % heading specification

\theoremstyle{mythmstyle}
\newtheorem{lemma}{Lemma}[section]
\newtheorem{prop}[lemma]{Proposition}
\newtheorem{thm}[lemma]{Theorem}
\newtheorem{cor}[lemma]{Corollary}

\theoremstyle{mydefstyle}
\newtheorem{definition}[lemma]{Definition}

\newtheorem{exa}[lemma]{Example}

\theoremstyle{myremarkstyle}
\newtheorem{rem}[lemma]{Remark}

\AtBeginEnvironment{thm}{\needspace{3\baselineskip}}
\AtBeginEnvironment{prop}{\needspace{3\baselineskip}}
\AtBeginEnvironment{lemma}{\needspace{3\baselineskip}}
\AtBeginEnvironment{cor}{\needspace{3\baselineskip}}
\AtBeginEnvironment{definition}{\needspace{3\baselineskip}}
\AtBeginEnvironment{defi}{\needspace{3\baselineskip}}
\AtBeginEnvironment{exa}{\needspace{3\baselineskip}}

\newcommand{\thmref}[1]{Theorem~\ref{#1}}
\newcommand{\propref}[1]{Proposition~\ref{#1}}
\newcommand{\lemref}[1]{Lemma~\ref{#1}}

\newcommand{\defref}[1]{Definition~\ref{#1}}
\newcommand{\exref}[1]{Example~\ref{#1}}
\newcommand{\figref}[1]{Figure~\ref{#1}}
\newcommand{\tabref}[1]{Table~\ref{#1}}

\newcommand{\belyi}{Bely\u{\i}\xspace}
\newcommand{\dde}{dessin d'enfant\xspace}
\newcommand{\ddes}{dessins d'enfants\xspace}
\newcommand{\Dde}{Dessin d'enfant\xspace}

\newcommand{\DdEs}{Dessins d'Enfants\xspace}
\newcommand{\otherwise}{\text{otherwise}}
\newcommand{\AD}{\Aut{\msD}}

\newcommand{\sigman}{(1\ 2\ \ldots\ n)}
\numberwithin{equation}{section}
\newcommand{\thmNT}{
Let $n \ge 3$ be an integer such that $n = bq$, with $b \ge 2$, and assume that $n \equiv q \pmod{2}$.
Fix $x = \sigman \in S_{n}$. Then

\begin{align}
\f{N(b, q)}{T(b, q)} \ge \f{2}{n+2}.
\end{align}

In particular, the inequality becomes an equality when $b = 2$.
}

\newcommand{\thmITa}{
Fix $x = \sigman \in S_{n}$. Among the elements $y \in S_{n}$ of cycle type $(b^{q})$, the number of those for which
$G = \langle x, y \rangle$ becomes an imprimitive group with respect to residue classes modulo
$m$ $(m \mid n$, $2 \le m < n)$ is given by

\begin{align}
I_{m}(b, q) = \f{m!}{b^{q}}\left(\left(\f{n}{m}\right)!\right)^{m}\sum_{\{(d_{i}, t_{i})\}}
\prod_{i}\f{1}{d_{i}^{t_{i}}t_{i}!}\left(\f{d_{i}^{\f{d_{i}q}{m}}}{\left(\f{d_{i}q}{m}\right)!}\right)^{t_{i}},
\end{align}
where the summation on the right-hand side is taken over all partitions $\{(d_{i}, t_{i})\}$ of $m$ satisfying
}

\newcommand{\thmITb}{
\sum_{i} d_{i} t_{i} = m, \quad d_{i}\text{ are all distinct}, \quad d_{i} \mid b, \ m \mid d_{i} q\text{ for all }i.
}
\newcommand{\thmITbit}{
\sum_{i} d_{i} t_{i} = m, \quad d_{i}\textit{ are all distinct}, \quad d_{i} \mid b, \ m \mid d_{i} q\textit{ for all }i.
}

\newcommand{\thmtrees}{For a uniform passport $[a^{p}, b^{q}, n] \ (p, q \ge 1)$, there exists a regular \dde with this passport if and only if $\gcd(p, q) = 1$.

The same statement holds for the passports $[a^{p}, n, b^{q}]$ and $[n, a^{p}, b^{q}]$.}

\newcommand{\thmggetwo}{If a uniform passport $[n, b^{q}, n]$ $(q \ge 1)$ has genus at least~$2$, then it admits a \dde with a trivial automorphism group.

The same statement holds for the passports $[b^{q}, n, n]$ and $[n, n, b^{q}]$.}

\begin{document}

\title[Regularity and Automorphism Groups of Dessins with Uniform Passports]{Regularity and Automorphism Groups of\\Dessins d'Enfants with Uniform Passports}

\author{Tatsuya Ohnishi}
%\date{\today}
\markboth{}{}

\begin{abstract}
For a smooth algebraic curve defined over a number field, one can associate a bipartite graph known as a \emph{\dde}.

In this paper, we investigate the regularity and automorphism groups of \ddes with \emph{uniform passports},
that is, those for which the valencies of black vertices, white vertices, and faces are constant,
and study how these properties depend on the genus.
Although uniformity imposes a high degree of symmetry, such dessins are not necessarily regular.

Our main results are as follows:
(1) A passport of the form $[a^{p}, b^{q}, n]$ (the tree case) admits a regular dessin if and only if $\gcd(p,q)=1$.
(2) Every passport of the form $[n, b^{q}, n]$ of genus at least~$2$ admits a dessin with a trivial automorphism group.

In addition, we obtain several results on uniform passports of genus~$0$ and~$1$.
We also establish two theorems on the enumeration of elements in symmetric groups,
which are useful for the study of automorphism groups of dessins.
\end{abstract}

\maketitle

\vspace{-1\baselineskip}

%\renewcommand{\thefootnote}{\fnsymbol{footnote}}
% Footnotes for arXiv version
\def\thefootnote{\fnsymbol{footnote}}
\makeatletter
\renewcommand\@makefntext[1]{%
  \noindent\hspace{0.5em}#1%
}
\makeatother
\footnotetext{Tatsuya~Ohnishi~(\raisebox{-0.18ex}{\includegraphics[height=2.05ex]{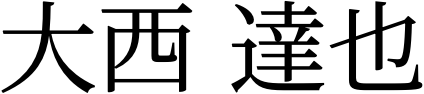}})}
\footnotetext{Graduate School of Information Science and Technology, The University of Osaka, Japan}
\footnotetext{e-mail: {\tt ohnishi-t@ist.osaka-u.ac.jp}}
\makeatletter
\renewcommand\@makefntext[1]{%
  \noindent\@makefnmark\ #1%
}
\makeatother

\def\thefootnote{\arabic{footnote}}

%%%%%%%%%%%%%%%%
{\small
\tableofcontents 
}
%%%%%%%%%%%%%%%%

%\vspace{-1.5\baselineskip}

%\OL

%%%%%%%%%%%%%%%%%%%%%%%%%%%%%%%%%%%%%%%%%%%%%%%%
\section{Introduction}

%%%%%%%%%%%%%%%%%%%%%%%%%%%%%%%%%%%%
\subsection{\belyi's Theorem and \DdEs}

\label{sec:belyi}
\begin{thm}[\belyi's Theorem]\cite{Belyi79}\cite{Belyi02}\cite[Theorem~1.3]{Jones16}
Let $X$ be a compact Riemann surface, that is, a smooth projective algebraic curve in the complex projective space
$\Pj_{\C}^{N}$ for some $N$.
Then $X$ can be defined over the field of algebraic numbers $\Qb$ if and only if there exists a non-constant meromorphic
function $\B\colon X \to \Ch\ (\ceq \C \cup \{ \infty \})$ ramified over at most three points.
\end{thm}

Everything begins with this striking and deep theorem. As a consequence of this result, for any compact Riemann surface
defined over a number field, one can choose a suitable function $\B$ (called a \emph{\belyi function}) such that all its
critical values lie within three points. 

By applying an appropriate M\"{o}bius transformation, these three points
can be taken to be $0$, $1$, and $\infty$.
For such a pair $(X, \B)$ (called a \emph{\belyi pair}), one can draw a bipartite graph called
a \emph{\dde} (child's drawing), or simply a \emph{dessin} (see \defref{def:dessin}).

The dessin for $(X, \B)$ is drawn on an orientable surface of the same genus as $X$, where the black vertices ($\bullet$)
 and white vertices ($\circ$) represent $\B^{-1}(0)$ and $\B^{-1}(1)$, respectively\footnote{Some references
(such as \cite{Jones16}) represent $\B^{-1}(0)$ by white vertices and $\B^{-1}(1)$ by black vertices.
In this paper, we adopt the convention used in many classical references.}.
The faces --- that is, the connected components bounded by edges --- correspond to $\B^{-1}(\infty)$, and
the edges correspond to $\B^{-1}([0,1])$. Each face is homeomorphic to a disk.

By studying the properties of \ddes, one can combine insights from algebraic geometry and
combinatorics, including graph theory, to enrich both areas of analysis. This perspective also opens up
a range of possibilities for further applications.

Two examples of \ddes are shown in \figref{fig:dessins-ex}.
In the left example, $X = \Ch$ (of genus~$0$) and $\B = -z^{4}(3z-5)/(5z-3)$, hence the dessin is drawn on a sphere.
The preimage $\B^{-1}(0)$ consists of the points $0$, $5/3$ with valencies $4$ and $1$, respectively;
$\B^{-1}(1)$ consists of the points $1$, $(-2 \pm \sqrt{5} i)/3$ with valencies $3$, $1$, and $1$;
and $\B^{-1}(\infty)$ consists of the points $3/5$, $\infty$ with valencies $1$ and $4$.

In the right example, $X$ is the curve defined by $x^{3}+y^{3}=z^{3}$
(the Fermat curve of degree~$3$ and genus~$1$) in $\Pj^{2}_{\C}$, and $\B = x^{3}/z^{3}$.
Let $\zeta_{3}$ denote a primitive cube root of unity.
This dessin has three faces corresponding to the points
$[1,-1,0]$, $[1,-\zeta_{3},0]$, and $[1,-\zeta_{3}^{2},0]$.

\begin{figure}[htbp]
\centering
\includegraphics[width=150mm]{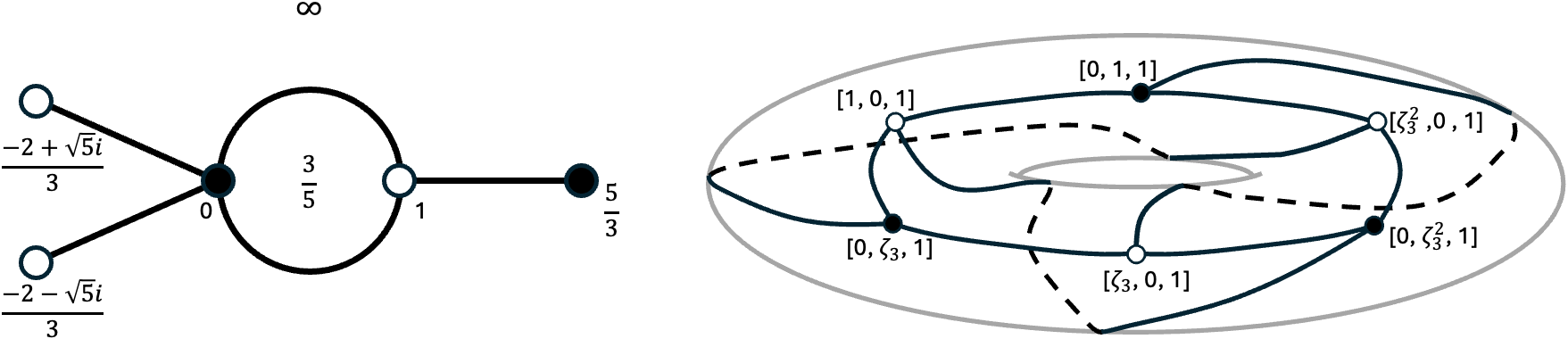} 
\caption{Examples of \ddes}
\label{fig:dessins-ex}
\end{figure}

The number of edges of a dessin coincides with the degree of the \belyi function.
Note that the degree of a \belyi function is defined as the number of points in the preimage of
a non-critical value; this does not necessarily coincide with the degree of $\B$
when it is regarded as a rational function.
In fact, in the right example above, the degree of the \belyi function is $9$, whereas the degree of $\B$
as a rational function is $3$.

Two dessins $\msD$ and $\msD'$, corresponding respectively to the \belyi pairs
$(X, \B)$ and $(X', \B')$, are said to be \emph{isomorphic} if there exists
a homeomorphism $\varphi \colon X \to X'$ such that $\B' \circ \varphi = \B$.
Equivalently, $\msD$ and $\msD'$ are isomorphic if there exists an orientation-preserving homeomorphism
between $X$ and $X'$ that induces an isomorphism of the embedded bipartite graphs, preserving vertex colors and
the cyclic order of incident edges at each vertex.

For a \dde, one can define an action of the absolute Galois group
$\mbG = \Gal(\Qb/\Q)$ \cite[4.2.1]{Jones16}.
Given a dessin $\msD$ associated with a \belyi pair $(X, \B)$ and an element
$\sigma \in \mbG$, we denote by $(X^{\sigma}, \B^{\sigma})$ the pair obtained by
the action of $\sigma$ on $X$ and $\B$.
We then denote by $\msD^{\sigma}$ the dessin corresponding to the \belyi pair
$(X^{\sigma}, \B^{\sigma})$.

Although $\msD^{\sigma}$ has the same number of edges and the same \emph{passport}
(\ie the same list of ramification indices) as $\msD$, the incidence relations between vertices may change.
Consequently, $\msD^{\sigma}$ need not be isomorphic to $\msD$ as a dessin.

The action of $\mbG$ on the set of all dessins is faithful; that is,
for any two distinct elements of $\mbG$, there exists a dessin
whose images under these elements lie in distinct isomorphism classes.
Therefore, the study of Galois orbits of dessins provides a powerful tool for investigating the structure of the absolute Galois group.

%%%%%%%%%%%%%%%%%%%%%%%%%%%%%%%%%%%%
\subsection{Regularity and Automorphism Groups}

\label{sec:regaut}
Two fundamental invariants (up to isomorphism) of a \dde with respect to the Galois action are its regularity and its automorphism group.
A dessin is said to be \emph{regular} if its \emph{monodromy group} (see \defref{def:monog}) acts
regularly (that is, freely and transitively) on the set of edges.
From a geometric point of view, the \emph{automorphism group} of a dessin is the group of deck transformations of the
associated \belyi covering.
Equivalently, it can be identified with the centralizer of the monodromy group in the symmetric group acting on the edges.

By studying regularity and automorphism groups, one can gain insight into the symmetry properties of a \dde.
The order of the automorphism group always divides the number of edges of the dessin, and the equality of these two numbers is equivalent to the dessin being regular.
Thus, the structure of the automorphism group provides a precise measure of the symmetry exhibited by the dessin.

\figref{fig:regular} shows two dessins with the same passport; that is, they have the same valency list:
each has $8$ edges, two black vertices of valency $4$, four white vertices of valency $2$, and two faces of valency $4$.
Although both dessins may appear highly symmetric at first glance, the left dessin is regular, whereas the right one is not.
The automorphism group of the left dessin has order~$8$, while that of the right dessin has order~$4$.

\begin{figure}[htbp]
\centering
\includegraphics[width=135mm]{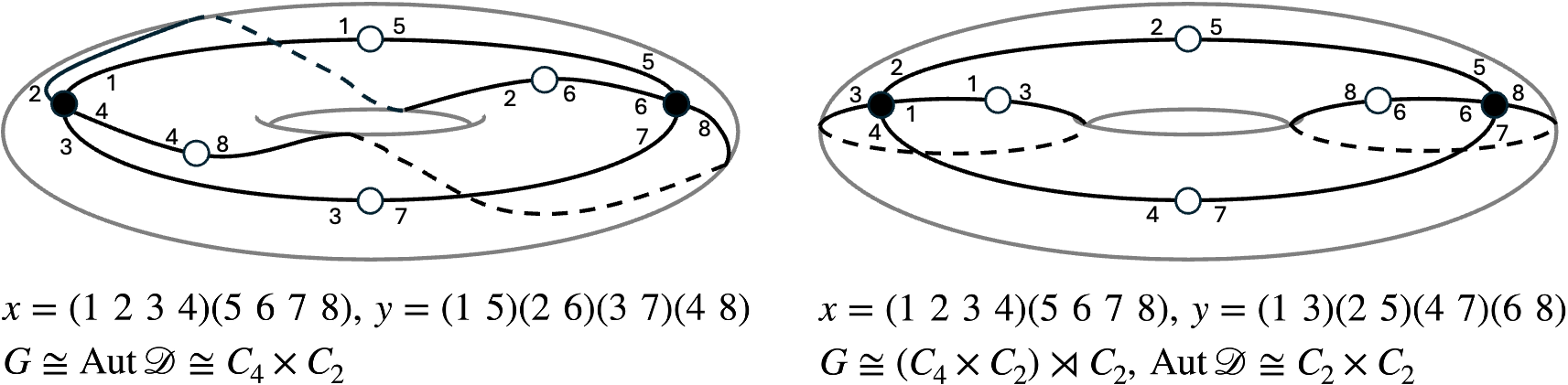} 
\caption{Regular and non-regular dessins with the same passport}
\label{fig:regular}
\end{figure}

Thus, regularity and automorphism groups may differ among dessins with the same passport.
However, they are invariant under the action of the absolute Galois group.
They therefore play an essential role in the study of families of dessins lying in the same Galois orbit.

Furthermore, regularity and automorphism groups have a significant impact on the relationship between
the field of moduli and the field of definition.

For a \dde $\msD$ corresponding to a \belyi pair $(X, \B)$, the \emph{field of moduli} $M(\msD)$
is defined as the fixed field of the subgroup
$G(\msD) = \{ \sigma \in \mbG \mid \msD \cong \msD^{\sigma} \} \le \mbG$.
That is, $M(\msD)$ is the subfield of $\Qb$ consisting of those elements that are fixed by every automorphism
$\sigma$ for which the conjugate dessin $\msD^{\sigma}$ is isomorphic to $\msD$.

A number field $K$ is called a \emph{field of definition} of $\msD$ if both $X$ and $\B$ can be defined over $K$.
Unlike the field of moduli, a field of definition of $\msD$ is in general not unique,  and there need not exist
a smallest field of definition.

The field of moduli depends only on the isomorphism class of $\msD$ and is contained in any field of definition of $\msD$.
For most \ddes, the field of moduli is also a field of definition; that is, $X$ and $\B$ can be defined over $M(\msD)$.
However, this is not always the case.
In such situations, there is no field of definition of $\msD$ that is fixed by all $\sigma \in G(\msD)$.

In the study of \ddes, the relationship between the field of moduli and fields of definition is strongly
influenced by both the genus and the combinatorial structure of the dessin.
If a dessin has a trivial automorphism group, then its field of moduli is necessarily a field of definition,
independently of the genus or the passport. This follows from rigidity of descent, which forces the
associated Weil cocycle to be trivial.

The same conclusion holds for regular dessins, that is, for dessins whose associated \belyi functions are Galois coverings.
In this case, the large symmetry encoded in the automorphism group allows the descent data to be normalized,
so that the dessin descends to its field of moduli.
Consequently, regular dessins do not give rise to genuine descent obstructions.

By contrast, when a dessin has a nontrivial automorphism group but is not regular, the behavior of the
field of moduli becomes more subtle. As shown by Cueto~\cite{Cueto14}, in this intermediate situation the field of moduli
need not be a field of definition. From a cohomological viewpoint, the obstruction to descent is measured by
a Weil cocycle with values in the automorphism group, and explicit examples arise precisely when this cocycle
represents a nontrivial class.

In this paper, we develop methods for analyzing how regular dessins and dessins with trivial
automorphism groups are distributed as the genus varies.
These results contribute to a deeper understanding of the relationship between the field of moduli and the field of definition.

%%%%%%%%%%%%%%%%%%%%%%%%%%%%%%%%%%%%
\subsection{Uniform Passports and Dessins}

\label{sec:uniform}
A \dde is said to be \emph{uniform} if the valencies of black vertices, white vertices, and faces are each constant.
One also says that it has a uniform passport (or uniform valency list).

Uniform dessins exhibit a high degree of symmetry. However, while every regular dessin is uniform, the converse does not hold: a uniform dessin need not be regular. In other words, uniformity does not represent the highest possible level of symmetry.

As an example, consider uniform passports $[a^{p}, b^{q}, c^{r}]$, where $n = pa = bq = rc$, and suppose that $n = 6$.
By symmetry among black vertices, white vertices, and faces, we may assume $c \ge a \ge b$ (equivalently, $r \le p \le q$).

In genus~$0$, the uniform passports are $[6, 1^{6}, 6]$ and $[2^{3}, 2^{3}, 3^{2}]$.
The corresponding dessins are shown in \figref{fig:genus0-n6}.
In both cases, the dessins are regular.

\begin{figure}[htbp]
\centering
\includegraphics[width=60mm]{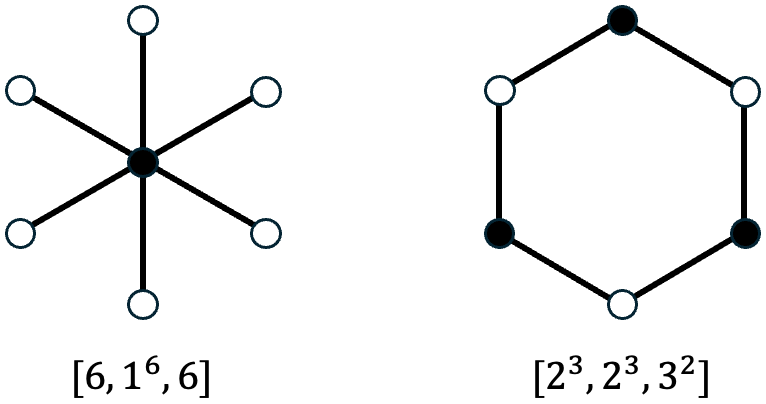} 
\caption{Uniform \ddes of genus~0, degree 6}
\label{fig:genus0-n6}
\end{figure}

In genus~1, the uniform passports are $[3^{2}, 2^{3}, 6]$ and $[3^{2}, 3^{2}, 3^{2}]$, and the corresponding dessins are shown in \figref{fig:genus1-n6}.
The dessin on the left is regular, whereas the one on the right is not.
The automorphism group of the right-hand dessin has order~2, which is strictly smaller than 6.

\begin{figure}[htbp]
\centering
\includegraphics[width=135mm]{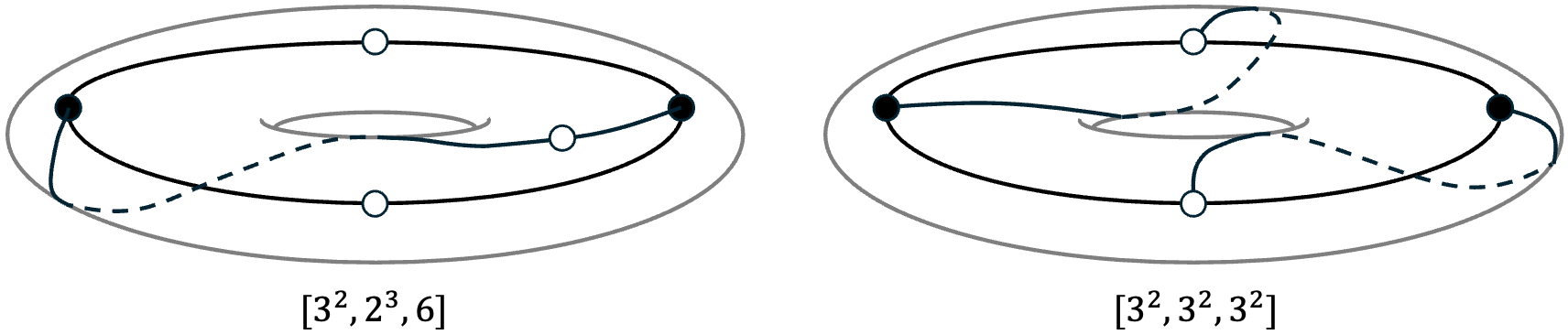} 
\caption{Uniform \ddes of genus~1, degree 6}
\label{fig:genus1-n6}
\end{figure}

For genus at least~2, the only uniform passport is $[6, 3^{2}, 6]$, which has genus~2.
There are four dessins with this passport, shown in \figref{fig:genus2-n6}.
The orders of their automorphism groups are 6, 3, 2, and 1, respectively, from the upper left to the lower right.
Only the upper-left dessin is regular; the others are non-regular, and the lower-right dessin has a trivial automorphism group.

\begin{figure}[htbp]
\centering
\includegraphics[width=135mm]{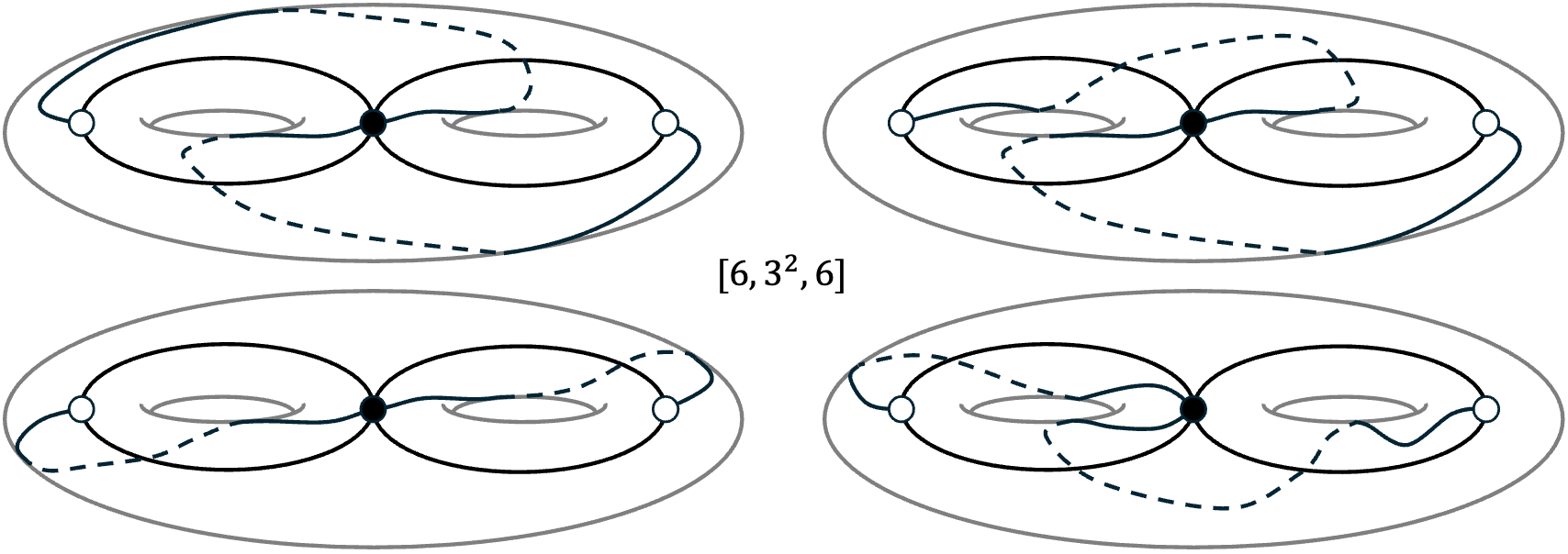} 
\caption{Uniform \ddes of genus~2, degree 6}
\label{fig:genus2-n6}
\end{figure}

%%%%%%%%%%%%%%%%%%%%%%%%%%%%%%%%%%%%
\subsection{Related Work}

The study of regularity and automorphism groups of \ddes has been approached from several complementary perspectives, including group-theoretic constructions, topological realization results, and arithmetic considerations, particularly in the presence of uniform or highly symmetric combinatorial data.
We briefly recall some relevant results in these directions.

Jones~\cite{Jones14} shows how, for a finite group $G$, regular dessins with automorphism group $G$
can be enumerated, represented as quotients of a single regular dessin $U(G)$,
and analyzed under the action of certain hypermap operations.
Examples for several classes of groups, such as the cyclic group $C_{n}$, the dihedral group $D_{n}$,
the symmetric group $S_{n}$, the alternating group $A_{n}$, and certain non-abelian finite simple groups, are
provided.

Complementarily, Hidalgo~\cite{Hidalgo19} proves that for every possible action of a finite group $G$
by orientation-preserving homeomorphisms on a closed orientable surface of genus $g \ge 2$,
there exists a dessin whose automorphism group is isomorphic to $G$
and which realizes the corresponding topological action of $G$.

Girondo, Torres-Teigell, and Wolfart~\cite{Girondo14} investigate the arithmetic and moduli-theoretic properties of
uniform dessins on quasiplatonic surfaces, that is, curves admitting a regular dessin.
In the arithmetic quasiplatonic setting they consider, they show that the field of moduli of a quasiplatonic curve
is remarkably stable: not only does the curve itself descend to its field of moduli, but the same field also serves as
the field of moduli for all uniform non-regular dessins arising from non-normal inclusions of triangle groups on that
curve.
Moreover, under mild additional geometric conditions, these dessins can in fact be defined over this field.

Horie~\cite{Horie24} studies the enumeration of equivalence classes of dessins whose automorphism groups have
a prescribed order~$r$. The focus is on dessins with two vertices, which necessarily have
passports (valency lists) of the form $[n^{1}, n^{1}, \lambda]$ with $\lambda \vdash n$.
Explicit enumeration formulas are obtained for two cases:
(1) dessins with $L$ faces, that is, with $l(\lambda)=L$;
and (2) dessins with $h$ faces of degree~$2$.

Taken together, these works demonstrate that regular dessins and their automorphism groups are highly constrained and
well understood from both group-theoretic and topological viewpoints, while uniformity alone does not control
regularity, descent, or symmetry.

%%%%%%%%%%%%%%%%%%%%%%%%%%%%%%%%%%%%
\subsection{Purpose and Main Results}

By computing automorphism groups of uniform \ddes and examining their regularity, one observes
a tendency for the proportion of regular dessins to decrease as the genus increases.

This naturally leads to the following questions:
\HL

For a given uniform passport,
\begin{itemize}
\item Under what conditions do regular \ddes exist, and how many are there?
\item How are the automorphism groups distributed?
\item How does this distribution change with the genus?
\end{itemize}
\HL

In response to these questions, we expect the situation to be as summarized in \tabref{tab:genusautd}, where $n$ denotes
the number of edges of the dessin.

\begin{table}[htbp]
  \centering
\begin{tabular}{|l|c|c|c|}
\hline
Genus & $\Aut \msD \cong \{ 1 \}$ & $1 < \lvert \AD \rvert < n$ & $\lvert \AD \rvert = n$ (regular) \\
\hline
0 & -- \textsuperscript{\ref{itm:g0reg}} & -- \textsuperscript{\ref{itm:g0reg}} & $\checkmark$\textsuperscript{\ref{itm:g0reg}} \\ \hline
1 (with an $(n)$-cycle) & -- \textsuperscript{\ref{itm:g1triv}}  & -- \textsuperscript{\ref{itm:g1nreg}} & $\checkmark$\textsuperscript{\ref{itm:anyg}} \\ \hline
1 (without an $(n)$-cycle) & -- \textsuperscript{\ref{itm:g1triv}}  & $\checkmark$\textsuperscript{\ref{itm:g1nreg}} & $\diamond$ \\ \hline
$\ge 2$ (with $\ge 2$ $(n)$-cycles) & $\checkmark$\textsuperscript{\ref{itm:g2triv}} & $\diamond$ & $\checkmark$\textsuperscript{\ref{itm:anyg}} \\ \hline
$\ge 2$ (with $\le 1$ $(n)$-cycle) & $\checkmark$* & $\checkmark$* & $\diamond$ \\
\hline
\end{tabular}\\
$\checkmark$: always occurs,\quad --: never occurs,\quad $\diamond$: depends on the passport.
\vspace{1\baselineskip}
\caption{Distribution of automorphism groups for uniform passports}
\label{tab:genusautd}
\end{table}

Asterisks (*) indicate statements that are conjectural at the present stage.
The numbers in parentheses refer to the corresponding items listed below.

\begin{enumerate}
\item\label{itm:g0reg} Genus~0: all uniform dessins are regular and have automorphism groups of order~$n$,
isomorphic to their monodromy groups.
\item\label{itm:g1nreg} Genus~1: a uniform passport admits a non-regular dessin if and only if it does not contain
an element of type $n^{1}$ (that is, it does not correspond to a tree).
\item\label{itm:g1triv} Genus~1: no uniform passport admits a dessin with a trivial automorphism group.
\item\label{itm:g2triv} Genus $\ge 2$: every uniform passport admits a dessin with a trivial automorphism group, and hence admits
a non-regular dessin.
Moreover, if the passport contains at most one cycle of type $(n)$, then it also admits a non-regular dessin with
a nontrivial automorphism group.
\item\label{itm:anyg} Any genus: a passport of the form $[a^{p}, b^{q}, n]$, or any permutation thereof in the tree case, admits a regular dessin if and only if $\gcd(p, q) = 1$.
\end{enumerate}

The main contribution of this paper is to provide proofs, in the following form, for statement~\ref{itm:anyg}
and for certain cases of statement~\ref{itm:g2triv}.
\HL

%%% Theorem for the regularity of the trees
\noindent
\textbf{\thmref{thm:trees}.}
\textit{\thmtrees}
\HL

%%% Theorem for the automorphism groups in $g \ge 2$
\noindent
\textbf{\thmref{thm:gge2}.}
\textit{\thmggetwo}
\HL

For the proof of \thmref{thm:gge2}, we make use of the relationship between the primitivity of the monodromy group
of a dessin and the triviality of its automorphism group. This approach is useful for investigating the distribution
of dessins with low symmetry.

From the perspective of the field of moduli and the field of definition, the above results broaden the range of cases
in which the field of moduli can be shown to be a field of definition. At the same time, they help to narrow down and
identify the cases in which this coincidence fails.

In addition, we provide proofs of \ref{itm:g0reg} (\textbf{\propref{prop:reggenus0}}),
\ref{itm:g1nreg} (\textbf{\propref{prop:reggenus1}}), and \ref{itm:g1triv} (\textbf{\propref{prop:autgenus1}}),
as well as a proposition concerning regularity based on properties of abstract groups (\textbf{\propref{prop:absgroup}}).
Although the proofs of these propositions are straightforward,
we could not find them explicitly stated in the literature. For completeness, we therefore provide proofs here.

We also establish the following two theorems concerning enumeration in symmetric groups, which are needed
to determine the distribution of dessins with trivial automorphism group.
\HL

In the following theorem, let $T(b, q)$ denote the number of permutations in the symmetric group $S_{n}$
with cycle type $(b^{q})$,
and let $N(b, q)$ denote the number of permutations $y \in S_{n}$ of cycle type $(b^{q})$ such that $xy$ has
cycle type $(n)$.
\HL

\noindent
\textbf{\thmref{thm:MN}.}
\textit{\thmNT}
\HL

\noindent
\textbf{\thmref{thm:imbq}.}
\textit{\thmITa}
\begin{align}
\thmITbit
\end{align}
\HL

For future work, two directions remain.
First, to establish the conjectural statements summarized in Table~\ref{tab:genusautd}.
This includes, in particular, proving the existence of dessins with trivial automorphism groups
for other uniform passports of genus at least~$2$.
Second, to obtain more quantitative results on regularity and automorphism groups in general.

%%%%%%%%%%%%%%%%%%%%%%%%%%%%%%%%%%%%%%%%%%%%%%%%
\section{Preliminaries}

The definitions and results in this chapter are based primarily on \cite{Jones16} and \cite{Adrianov20}.

\begin{definition}[\Dde]\cite[Definition~2]{Jones16}
\label{def:dessin}
A \emph{\dde}, or simply a \emph{dessin}, is a map consisting of a connected, finite, bipartite graph embedded
in a connected, compact, oriented surface without boundary.
Here, a bipartite graph is a graph whose vertices can be colored black and white in such a way that each edge
joins a black vertex to a white vertex.
\end{definition}

Since a compact Riemann surface provides a suitable surface on which a \dde can be embedded,
one can draw a dessin corresponding to a \belyi pair $(X, \B)$, as described in
Section~\ref{sec:belyi}.

\HL

A \emph{partition} $\la$ of a positive integer $n$, denoted by $\la \vdash n$, is a multiset of
positive integers whose sum is $n$, where the order of the parts is irrelevant.
 
\begin{definition}[Passport of a dessin]\cite[Definition~2.10]{Adrianov20}
\label{def:passport}
Let $n$ be the number of edges of a \dde.
The triple $[\lambda_{0}, \lambda_{1}, \lambda_{\infty}]$ of partitions $\lambda_{0}, \lambda_{1}, \lambda_{\infty} \vdash n$,
which correspond respectively to the valencies of the black vertices ($\B^{-1}(0)$), the white vertices ($\B^{-1}(1)$), and
the faces ($\B^{-1}(\infty)$) of the dessin, is called a $\emph{passport}$ of the dessin.
\end{definition}

The dessins in \figref{fig:dessins-ex} have passports $[41, 311, 41]$ and $[3^{3}, 3^{3}, 3^{3}]$,
respectively.
Here the symbols $41$, $311$, and $3^{3}$ denote the partitions $(4,1)$, $(3,1,1)$, and $(3,3,3)$,
respectively.

Since $\lambda_{0}, \lambda_{1}, \lambda_{\infty}$ are partitions of $n$, we have

\begin{align}
\label{eq:lambdan}
\lvert \lambda_{0} \rvert = \lvert \lambda_{1} \rvert = \lvert \lambda_{\infty} \rvert = n,
\end{align}
where $\lvert \la \rvert$ denotes the sum of the parts of the partition $\la$. 

The numbers of vertices, faces, and edges are $l(\lambda_{0}) + l(\lambda_{1})$,
$l(\lambda_{\infty})$, and $n$, respectively, where $l(\la)$ denotes the number of parts
of the partition $\la$.
Therefore, the genus $g$ of the underlying curve $X$ satisfies

\begin{align}
l(\lambda_{0}) + l(\lambda_{1}) + l(\lambda_{\infty}) - n = 2 - 2g,
\end{align}
and hence

\begin{align}
\label{eq:lambdag}
g = \f{n - (l(\lambda_{0}) + l(\lambda_{1}) + l(\lambda_{\infty}))}{2} + 1.
\end{align}
Since $g$ is a non-negative integer, it follows that

\begin{align}
\label{eq:lambdal}
l(\lambda_{0}) + l(\lambda_{1}) + l(\lambda_{\infty}) \le  n + 2, \q
l(\lambda_{0}) + l(\lambda_{1}) + l(\lambda_{\infty}) \equiv n \npmod{2}.
\end{align}

Note that not every triple of partitions satisfying \eqref{eq:lambdan} and \eqref{eq:lambdal}
necessarily corresponds to a \dde.
For example, although $[2^{2}, 2^{2}, 31]$ formally satisfies these conditions and would yield genus~$0$,
there exists no dessin with this passport.
Similarly, the passport $[3^{2}, 3^{2}, 42]$, which would correspond to genus~$1$, admits no dessin.

These facts can be proved by showing that there exists no corresponding monodromy group
(see \defref{def:monog}) in each case.

\begin{definition}[Uniform passports and dessins]\cite[Remark~3.2]{Jones16}
The passport of a \dde given by
$[\lambda_{0}, \lambda_{1}, \lambda_{\infty}] = [a_{1}\cdots a_{p},\ b_{1}\cdots b_{q},\ c_{1}\cdots c_{r}]$
is called \emph{uniform} if
\begin{align}
a_{1} = \cdots = a_{p}, \quad b_{1} = \cdots = b_{q}, \quad c_{1} = \cdots = c_{r},
\end{align}
\ie if the passport takes the form $[a^{p}, b^{q}, c^{r}]$.

A \dde is called \emph{uniform} if it has a uniform passport.
\end{definition}

For a uniform passport $[a^{p}, b^{q}, c^{r}]$ with $n = pa = qb = rc$ and the genus~$g$ of the
underlying curve $X$ corresponding to a dessin with this passport, by \eqref{eq:lambdan} we have

\begin{align}
\label{eq:genus-g}
p + q + r - n = 2 - 2g,
\end{align}
and hence

\begin{align}
\label{eq:genus}
g = \f{n-(p+q+r)}{2} + 1.
\end{align}

\begin{definition}[Monodromy group of a \dde]\cite[2.1.1]{Jones16}
\label{def:monog}
Define two permutations $x$ and $y$ acting on the set of edges $E$ of a \dde $\msD$ as follows.
For each edge $e \in E$, define $x \cdot e$ and $y \cdot e$ to be the next edges around the unique black vertex and
the unique white vertex incident to $e$, respectively, following the counterclockwise orientation.

The \emph{monodromy group} of $\msD$ is the subgroup $G = \langle x, y \rangle$ generated by $x$ and $y$ in
the symmetric group $\Sym(E)$ of all permutations of $E$.
\end{definition}

Since a \dde is connected, it follows that any edge in $E$ can be mapped to any other edge by the action of $G$.
Therefore, the monodromy group $G$ acts transitively on $E$.

\begin{rem}
Throughout this paper, permutations act on the left, and products are composed from right to left,
\ie $(xy) \cdot e = x \cdot (y \cdot e)$ for an edge $e$.
\end{rem}

Another important observation concerning the monodromy group is that, in addition to $x$ and $y$ encoding
the cycles of the black and white vertices, respectively, the permutation $z = (xy)^{-1}$ encodes the cycles
 corresponding to the faces.
In fact, for each face, half of the edges incident to it form a cycle of $z$, while the remaining edges belong to cycles of
$z$ corresponding to the neighboring faces.

It is known that if a group $G$ generated by two elements acts transitively on a set of
edges $E$, then there exists a \dde whose monodromy group is isomorphic to $G$\cite[Theorem~3.6]{Scodro24}.

Therefore, studying groups that act transitively is essential for investigating the properties of \ddes and, consequently, of algebraic curves.

\begin{definition}[Regular dessins]\cite[2.1.2]{Jones16}
\label{def:regular}
A \dde is called \emph{regular} if its monodromy group acts freely (that is, semiregularly) on the set of its edges.
\end{definition}

This implies that the monodromy group of a regular dessin acts freely and transitively --- hence regularly ---
on its edges.

Although we are not aware of a reference in which the following lemma is proved explicitly,
the proof is straightforward.

\begin{lemma}
\label{lem:order-n}
A \dde with $n$ edges is regular if and only if the order of its monodromy group is $n$.
\end{lemma}

\begin{proof}
Let $\msD$ be a \dde with $n$ edges.
Denote the set of edges by $E$ and the monodromy group by $G$.
Since $G$ acts transitively on $E$, for any $e \in E$ and for any $e' \in E$, there exists at least one element of $G$
sending $e$ to $e'$, and hence $\lvert G \rvert \ge n$.

When $\msD$ is regular, since $G$ acts freely on $E$, for each $e, e' \in E$ there is at most one element of $G$
sending $e$ to $e'$, hence $\lvert G \rvert \le n$.
Therefore, in this case we have $\lvert G \rvert = n$.

Conversely, when $\lvert G \rvert = n$, for each $e, e' \in E$ there is exactly one element of $G$ sending $e$ to $e'$,
hence $G$ acts freely on $E$, which means that $\msD$ is regular.
\end{proof}

\begin{definition}[Automorphism group of a \dde]\cite[2.1.2]{Jones16}
For a \dde $\msD$, we define its \emph{automorphism} to be a permutation of the set $E$ of $\msD$ which preserves the
cyclic order of edges around each vertex, that is, which commutes with $x$ and $y$, or equivalently, commutes with $G$.
Thus we can define an \emph{automorphism group} of $G$ as the centralizer:

\begin{align}
\Aut \msD \ceq C_{\Sym(E)}(G) &= \{ c \in \Sym(E) \mid cg = gc \text{ for all } g \in G \} \\
&= \{ c \in \Sym(E) \mid cx = xc,\ cy = yc \}.
\end{align}

When $\Aut \msD \cong \{1 \}$, it is said that $\msD$ has a \emph{trivial} automorphism group.
\end{definition}

Since $\langle x, y \rangle = \langle x, z \rangle = \langle y, z \rangle$, where $z = (xy)^{-1}$,
both the monodromy group and the automorphism group are invariant under any permutation of black vertices,
white vertices, and faces.

The automorphism group of a \dde has the following properties:

\begin{itemize}
\item $\AD$ acts freely on the edges of $\msD$.
\item $\lvert \AD \rvert$ divides the number of edges.
\item If $\lvert \AD \rvert$ equals the number of edges, then $\AD \cong G$.
\end{itemize}

The following proposition is a basic result describing the relationship between regularity and automorphism groups.

\begin{prop}
\label{prop:regaut}
A \dde $\msD$ is regular if and only if its monodromy group $G$ is isomorphic to $\AD$.
\end{prop}

\begin{proof}
See \cite[Theorem~2.1]{Jones16}.
\end{proof}

This also implies that a \dde $\msD$ is regular if and only if $\lvert \AD \rvert$ equals the number of edges.

Note that regularity does not imply $G = \Aut \msD$;
these groups act on the edges of the dessin in different ways.

\begin{prop}
\label{prop:reguni}
If a \dde is regular, then it has a uniform passport.
\end{prop}

\begin{proof}
A proof can be found in \cite[Proposition~4.42]{Girondo12}. For the reader's convenience, we include another proof.

We prove the contrapositive. Suppose that the dessin $\msD$ is not uniform.
Then $\msD$ has either two black vertices, two white vertices, or two faces whose valencies are different.

Since regularity is symmetric with respect to black vertices, white vertices, and faces, we may assume without loss of generality that these are black vertices. Let these vertices be $A$ and $B$, and let their valencies be $l$ and $m$, with $l < m$.

Under $x^{l}$, all the edges around $A$ are mapped to themselves, whereas the edges around $B$ are mapped to distinct edges. Hence $x^{l} \ne \id$, and there exist at least two distinct elements of $G$
that fix the edges around $A$. Therefore, $G$ does not act freely on $E$, and $\msD$ is not regular.
\end{proof}

The converse of this proposition does not hold in general.
This shows that regularity exhibits a higher degree of symmetry than uniformity.
This fact is one of the main topics of this paper.

\begin{exa}
The uniform passport $[4^{2}, 2^{4}, 4^{2}]$ (genus~1) corresponds to the two dessins shown in
Section~\ref{sec:regaut}, \figref{fig:regular}.
The dessin on the left is regular, and both its monodromy group and its automorphism group are
isomorphic to $C_{4} \times C_{2}$, the direct product of cyclic groups of orders 4 and 2, respectively.
In contrast, the dessin on the right is not regular; its monodromy group is isomorphic to $(C_{4} \times C_{2}) \rtimes C_{2}$
and has order~16, whereas its automorphism group has order~4 and is isomorphic to $C_{2} \times C_{2}$.
\end{exa}

We also recall the notions of primitive and imprimitive groups, which will play a role in investigating
when a dessin has a trivial automorphism group.

\begin{definition}[Block in a group action]\cite[1.5]{Dixon96}
\label{def:block}
Let $G$ be a group acting transitively on a set $E$. A nonempty subset $S$ of $E$ is called a \emph{block} for $G$
if for each $g \in G$ either $g\cdot S = S$ or $(g \cdot S) \cap\ S = \emptyset$.

Every group acting transitively on $E$ has $E$ and the singletons $\{ e \} \ (e \in E)$ as blocks. these are
called \emph{trivial} blocks. Any other block is called \emph{nontrivial}.
\end{definition}

In other words, a block is a subset that is always moved as a whole under the action of $G$.

\begin{definition}[Primitive and imprimitive groups]\cite[1.5]{Dixon96}
\label{def:prim}
Let $G$ be a group which acts transitively on a set $E$. $G$ is called \emph{primitive} if it has no nontrivial blocks on $E$;
otherwise $G$ is called \emph{imprimitive}.
\end{definition}

%%%%%%%%%%%%%%%%%%%%%%%%%%%%%%%%%%%%%%%%%%%%%%%%
\section{Estimates for the Symmetric Group}

In order to investigate the behavior of automorphism groups in genus at least~$2$ (see Section~\ref{sec:autgge2}),
we need to obtain estimates
for the number of permutations in the symmetric group $S_{n}$.
In this section, we prove two theorems (Theorems~\ref{thm:MN} and~\ref{thm:imbq}) that provide lower and upper bounds
for quantities related to monodromy groups.

%%%%%%%%%%%%%%%%%%%%%%%%%%%%%%%%%%%%
\subsection{Elements of Cycle Type $(b^{q})$}

\label{sec:tbq}
Let $x = \sigman \in S_{n}$ and assume that $n = bq$, where $b, q \in \Zp$.
We define the following subsets of $S_{n}$:

\begin{itemize}
\item $T(b, q)$: the set of permutations in $S_n$ with cycle type $(b^{q})$, that is, permutations consisting of
$q$ disjoint cycles of length $b$;
\item $N(b, q)$: the set of elements $y \in T(b, q)$ such that $xy$ has cycle type $(n)$;
\item $I(b, q)$: the set of elements $y \in T(b, q)$ such that $G = \langle x, y \rangle$ is imprimitive.
\end{itemize}

Note that the condition $y \in T(b, q)$ with $xy$ of cycle type $(n)$
is equivalent to requiring that $z = (xy)^{-1}$ has cycle type $(n)$,
since $xy$ and $(xy)^{-1}$ have the same cycle type.

In addition, for a divisor $m$ of $n$ with $2 \le m < n$, let $I_{m}(b, q)$ denote the subset of $I(b, q)$
consisting of those elements $y$ for which $G = \langle x, y \rangle$ is imprimitive with respect to the
partition of $\{1, \dotsc, n\}$ into residue classes modulo~$m$; that is, the residue classes modulo~$m$
form a system of blocks for $G$.

Hereafter, we use the same notation to denote the number of elements in a set (for example,
the number of elements of $T(b, q)$ is denoted by the function $T(b, q)$).

\begin{prop}
\label{prop:tbq}
The number of elements in $S_{n}$ of cycle type $(b^{q})$ $(n = bq)$ is given by

\begin{align}
T(b, q) = \f{n!}{b^{q}q!}.
\end{align}
\end{prop}

\begin{proof}
If we write an element of cycle type $(b^{q})$ as
\begin{align}
(a_{1} \ldots a_{b})(a_{b+1} \ldots a_{2b}) \cdots (a_{(q-1)b+1} \ldots a_{qb}),
\end{align}
then there are $n!$ possible arrangements of these numbers. Among them, the following duplications occur as permutations:

\begin{itemize}
\item the $b^{q}$ arrangements of elements within the $q$ cycles of length $b$,
\item the $q!$ arrangements of the $q$ cycles themselves.
\end{itemize}
Hence, we obtain $T(b, q) = n!/(b^{q} q!)$.
\end{proof}

%%%%%%%%%%%%%%%%%%%%%%%%%%%%%%%%%%%%
\subsection{Lower Bound for $N/T$}

A useful tool for computing the values of $N(b, q)/T(b, q)$ is the following theorem from \cite{Goupil98}.

\begin{thm}
\label{thm:goupil}
Let $\lambda = (\lambda_{1}, \dotsc, \lambda_{l})$ and $\mu = (\mu_{1}, \dotsc, \mu_{m})$
be partitions of $n$.
Define the genus associated with the pair $(\lambda, \mu)$ by

\begin{align}
g = \frac{n-(l+m)+1}{2},
\end{align}
and assume that $g \in \Zz$.

Let $c_{\lambda \mu}^{n}$ denote the number of solutions $(\sigma, \rho) \in C_{\lambda} \times C_{\mu}$
to the equation $\sigma \rho = \pi$,
where $\pi$ is a fixed $n$-cycle in $S_{n}$.
Then $c_{\lambda \mu}^{n}$ is given by

\begin{align}
c_{\la \mu}^{n} = \f{n}{z_{\la}z_{\mu}2^{2g}}\sum_{\substack{g_{1},g_{2}\ge 0\\ g_{1}+g_{2}=g}}(l+2g_{1}-1)! (m+2g_{2}-1)!
\sum_{\substack{(i_{1},\dotsc,i_{l})\vDash g_{1}\\(j_{1},\dotsc,j_{m})\vDash g_{2}}}
\prod_{k=1}^{l}\binom{\la_{k}}{2i_{k}+1}\prod_{k=1}^{m}\binom{\mu_{k}}{2j_{k}+1}.
\end{align}

Here $p \vDash n$ denotes a composition of $n$, that is, a finite sequence of non-negative integers
summing to $n$, where the order of the terms matters. We adopt the convention that $\binom{a}{b}=0$ if $b>a$.

Moreover, for a partition $\lambda = 1^{\A_{1}}\cdots n^{\A_{n}}$, where $\A_{i}$ denotes the multiplicity of $i$, we define
$z_{\lambda} = \prod_{i} \A_{i}! \, i^{\A_{i}}$.
\end{thm}

\begin{proof}
See \cite[Theorem~2.1]{Goupil98}.
\end{proof}
\HL

Using this theorem, we establish a lower bound for $N(b, q)/T(b, q)$.

\begin{thm}
\label{thm:MN}
\thmNT
\end{thm}

\begin{proof}
In \thmref{thm:goupil}, for the passport $[n, b^{q}, n]$, we have $\la = (n^{1})$, $l = 1$, $\mu = (b^{q})$, and $m = q$.
Moreover, the genus is given by $g = (n-q)/2 = q(b-1)/2$. Hence,

\begin{align}
N(b, q) = c_{(n) (b^{q})}^{(n)} = \f{1}{b^{q}q!2^{2g}}\sum_{\substack{g_{1},g_{2}\ge 0\\ g_{1}+g_{2}=g}}(2g_{1})! (2g_{2}+q-1)!
\binom{n}{2g_{1}+1}\sum_{(j_{1},\dotsc,j_{q})\vDash g_{2}}
\prod_{k=1}^{q}\binom{b}{2j_{k}+1}.
\end{align}
Since $T(b, q) = n!/(b^{q}q!)$ by \propref{prop:tbq}, we obtain

\begin{align}
\f{N}{T} = \f{1}{2^{2g}n!}\sum_{\substack{g_{1},g_{2}\ge 0\\ g_{1}+g_{2}=g}}(2g_{1})! (2g_{2}+q-1)!\binom{n}{2g_{1}+1}
\sum_{(j_{1},\dotsc,j_{q})\vDash g_{2}}
\prod_{k=1}^{q}\binom{b}{2j_{k}+1}.
\end{align}
For the terms involving $g_{1}$ and $g_{2}$ in the outer summation, we have

\begin{align}
\f{1}{n!}(2g_{1})!(2g_{2}+q-1)!\binom{n}{2g_{1}+1} &= \f{1}{n!}(2g_{1})!(2g_{2}+q-1)!\f{n!}{(2g_{1}+1)!(n-2g_{1}-1)!} \\
&= \f{(2g_{2}+q-1)!}{(2g_{1}+1)(n-2g_{1}-1)!}.
\end{align}
Since $g_{1}+g_{2} = g = (n-q)/2$, we have $q = n-2g_{1}-2g_{2}$. Hence,

\begin{align}
\f{(2g_{2}+q-1)!}{(2g_{1}+1)(n-2g_{1}-1)!}
&= \f{(2g_{2}+(n-2g_{1}-2g_{2})-1)!}{(2g_{1}+1)(n-2g_{1}-1)!} \\
&= \f{1}{2g_{1}+1} = \f{1}{2(g - g_{2})+1}.
\end{align}
Therefore,

\begin{align}
\f{N}{T} = \f{1}{2^{2g}}\sum_{g_{2}=0}^{g}\f{1}{2(g-g_{2})+1}
\sum_{(j_{1},\dotsc,j_{q})\vDash g_{2}}
\prod_{k=1}^{q}\binom{b}{2j_{k}+1} \q \left(g = \f{n-q}{2}\right).
\end{align}

Let

\begin{align}
A_{g_{2}} \ceq \sum_{(j_{1},\dotsc,j_{q})\vDash g_{2}}
\prod_{k=1}^{q}\binom{b}{2j_{k}+1},
\end{align}
then

\begin{align}
\f{N}{T} = \f{1}{2^{2g}}\sum_{g_{2}= 0}^{g}\f{A_{g_{2}}}{2(g-g_{2})+1}.
\end{align}

Consider the polynomial

\begin{align}
P_{b}(x) \ceq \sum_{j=0}^{\lfloor (b-1)/2 \rfloor}\binom{b}{2j+1}x^{j}.
\end{align}
Then we may write

\begin{align}
P_{b}(x)^{q} = \sum_{g_{2}=0}^{q\lfloor (b-1)/2 \rfloor}A_{g_{2}}x^{g_{2}}, \q A_{g_{2}} = [x^{g_{2}}]\,P_{b}(x)^{q},
\end{align}
where $[x^{g_{2}}]\,P_{b}(x)^{q}$ denotes the coefficient of $x^{g_{2}}$ in $P_{b}(x)^{q}$.

Moreover,

\begin{align}
P_{b}(1) &= \sum_{j=0}^{\lfloor (b-1)/2 \rfloor}\binom{b}{2j+1} = \sum_{\substack{l=1\\ l\colon \text{odd}}}^{b}\binom{b}{l}
= \f{1}{2}\left(\sum_{l=0}^{b}\binom{b}{l} - \sum_{l=0}^{b}(-1)^{l}\binom{b}{l}\right) \\
\label{eq:pb1}
&= \f{1}{2}((1+1)^{b} - (1-1)^{b}) = 2^{b-1}.
\end{align}
Therefore,

\begin{align}
\sum_{g_{2}=0}^{q\lfloor (b-1)/2 \rfloor}A_{g_{2}} &= P_{b}(1)^{q} = 2^{q(b-1)} = 2^{2g}.
\end{align}
If we define $\pi(g_{2}) \ceq A_{g_{2}}/2^{2g}$, then

\begin{align}
\sum_{g_{2}=0}^{q\lfloor (b-1)/2 \rfloor}\pi(g_{2}) = 1.
\end{align}
Hence, $\pi(g_{2})$ can be regarded as a probability mass function.

Since

\begin{align}
\f{N}{T} = \f{1}{2^{2g}}\sum_{g_{2}=0}^{g}\f{A_{g_{2}}}{2(g-g_{2})+1}
= \sum_{g_{2}=0}^{g}\f{\pi(g_{2})}{2(g-g_{2})+1},
\end{align}
if a random variable $X$ follows the probability distribution $\mbP(X = g_{2}) = \pi(g_{2})$,
then $N/T$ can be expressed in terms of its expected value as

\begin{align}
\label{eq:PX}
\f{N}{T} &= \mbE\left[\f{1}{2(g-X)+1}\right].
%\mbP(X = g_{2}) &=
%\begin{dcases}
%\f{1}{2^{2g}}\sum_{\substack{(j_{1},\dotsc,j_{q})\vDash g_{2}\\2j_{k}+1\le b}}
%\prod_{k=1}^{q}\binom{b}{2j_{k}+1} & (0 \le g_{2} \le g) \\
%0 & (\otherwise)
%\end{dcases}
\end{align}
On the other hand, since

\begin{align}
\sum_{j=0}^{\lfloor (b-1)/2 \rfloor} \binom{b}{2j+1} = 2^{b-1},
\end{align}
we can define a random variable $J$ by

\begin{align}
\label{eq:pjj}
\mbP(J=j) =
\begin{dcases}
\f{\binom{b}{2j+1}}{2^{b-1}} & (0 \le j \le \lfloor (b-1)/2 \rfloor) \\
0 & (\otherwise)
\end{dcases}.
\end{align}

%The probability generating function is

%\begin{align}
%p_{b}(x) \ceq \mbE[x^{J}] = \sum_{j}\mbP(J=j)x^{j} = \sum_{j=0}^{\lfloor (b-1)/2 \rfloor}\f{\binom{b}{2j+1}}{2^{b-1}}x^{j}
%\ \left(= \f{P_{b}(x)}{2^{b-1}}\right).
%\end{align}

Let $J_{1}, \dotsc, J_{q}$ be independent random variables following this distribution, and define
$S_{q} \ceq J_{1} + \cdots + J_{q}$. Then

\begin{align}
%\mbE[x^{S_{q}}] &= \mbE[x^{J_{1}+\cdots+J_{q}}] = p_{b}(x)^{q} = \sum_{g_{2}=0}^{g}A_{g_{2}}x^{g_{2}}, \\
\mbP(S_{q} = g_{2}) &= \sum_{j_{1}+\cdots+j_{q}=g_{2}}\prod_{k=1}^{q}\mbP(J_{k}=j_{k})
%= \sum_{j_{1}+\cdots+j_{q}=g_{2}}\prod_{k=1}^{q}\f{\binom{b}{2j+1}}{2^{b-1}} \\
= \f{1}{2^{q(b-1)}}\sum_{j_{1}+\cdots+j_{q}=g_{2}}\prod_{k=1}^{q}\binom{b}{2j+1} = \f{A_{g_{2}}}{2^{2g}} = \pi(g_{2}).
\end{align}
Therefore, $S_{q}$ follows the same probability distribution as in \eqref{eq:PX}.  Then we have

\begin{align}
\label{eq:gsq}
\f{N}{T} &= \mbE\left[\f{1}{2(g-S_{q})+1}\right].
\end{align}

Here,

\begin{align}
\mbE[J] &= \sum_{j\ge 0}j\cdot \mbP(J=j) = \f{1}{2^{b-1}}\sum_{j=0}^{\lfloor (b-1)/2 \rfloor}\binom{b}{2j+1} \\
&= \f{1}{2^{b-1}}\sum_{\substack{k=1\\ k\colon \text{odd}}}\f{k-1}{2}\binom{b}{k} \qquad (k = 2j+1) \\
\label{eq:EJ}
&= \f{1}{2^{b}}\sum_{\substack{k=1\\ k\colon \text{odd}}}^{b}k\binom{b}{k} - \f{1}{2^{b}}\sum_{\substack{k=1\\ k\colon \text{odd}}}^{b}\binom{b}{k}.
\end{align}
From \eqref{eq:pb1}, 

\begin{align}
\label{eq:pb2}
\sum_{\substack{k=1\\ k\colon \text{odd}}}^{b}\binom{b}{k} = 2^{b-1}.
\end{align}
Moreover, starting from

\begin{align}
(1 + x)^{b} = \sum_{k=0}^{b}\binom{b}{k}x^{k},
\end{align}
differentiating both sides with respect to $x$ and then multiplying by $x$ gives

\begin{align}
bx(1 + x)^{b-1} = \sum_{k=0}^{b}k\binom{b}{k}x^{k} = \sum_{k=1}^{b}k\binom{b}{k}x^{k}.
\end{align}
Substituting $x = 1$ and $x = -1$ respectively, we obtain

\begin{align}
2^{b-1}b &= \sum_{k=1}^{b}k\binom{b}{k}
= \sum_{\substack{k=1\\k\colon \text{odd}}}^{b}k\binom{b}{k} + \sum_{\substack{k=1\\k\colon \text{even}}}^{b}k\binom{b}{k}, \\
0 &= \sum_{k=1}^{b}k\binom{b}{k}(-1)^{k}
= -\sum_{\substack{k=1\\k\colon \text{odd}}}^{b}k\binom{b}{k} + \sum_{\substack{k=1\\k\colon \text{even}}}^{b}k\binom{b}{k}.
\end{align}
Subtracting these two equations and dividing by 2, we obtain

\begin{align}
\label{eq:pb3}
\sum_{\substack{k=1\\k\colon \text{odd}}}^{b}k\binom{b}{k} &= \f{2^{b-1}b-0}{2} = 2^{b-2}b.
\end{align}
From \eqref{eq:EJ}, \eqref{eq:pb2} and \eqref{eq:pb3}, we obtain

\begin{align}
\label{eq:ej}
\mbE[J] &= \f{2^{b-2}b}{2^{b}} - \f{2^{b-1}}{2^{b}} = \f{b-2}{4}.
\end{align}
Hence,

\begin{align}
g - \mbE[S_{q}] &= g - q\mbE[J] = g - \f{q(b-2)}{4} = \f{q(b-1)}{2} - \f{q(b-2)}{4} = \f{bq}{4} = \f{n}{4}.
\end{align}

For \eqref{eq:gsq}, consider the function $f(x) = 1/(2(g-x)+1) \ (0 \le x \le g)$. Then,

\begin{align}
f''(x) = \f{8}{(2(g-x)+1)^{3}} > 0.
\end{align}
Hence, $f(x)$ is a convex function.
By Jensen's inequality, we have

\begin{align}
\f{N}{T} &= \mbE[f(S_{q})] \ge f(\mbE[S_{q}]) = \f{1}{2(g-\mbE[S_{q}])+1} \\
&= \f{1}{2\cdot \f{n}{4}+1} = \f{2}{n+2}.
\end{align}

In particular, when $b = 2$, it follows from \eqref{eq:ej} that $\mbE[J] = 0$. As the random variable $J$ is always non-negative
by \eqref{eq:pjj}, in this case we conclude that $J \equiv 0$.

Consequently, $S_{q} = J_{1} + \cdots + J_{q} \equiv 0$. Therefore, by \eqref{eq:gsq},

\begin{align}
\f{N}{T} &= \mbE\left[\f{1}{2(g-S_{q})+1}\right]  = \f{1}{2g+1} = \f{1}{2\cdot\f{q(b-1)}{2}+1}
= \f{1}{q+1} = \f{2}{2q+2} = \f{2}{n+2}.
\end{align}
Therefore, the inequality is sharp, and equality holds precisely when $b = 2$.
\end{proof}

%%%%%%%%%%%%%%%%%%%%%%%%%%%%%%%%%%%%
\subsection{Computation of $I_{m}$}

In order to find an upper bound for $I(b, q)/T(b, q)$, we divide the problem according to the divisors
$m$ of $n$ with $2 \le m < n$.

\begin{lemma}
\label{lem:xres}
Let $G \le S_{n}$ be a subgroup containing $x = \sigman$.
If the natural action of $G$ on
$E = \{1, \dotsc, n\}$ is imprimitive, then every nontrivial block of $G$
is of the form
\begin{align}
B_{m,j} = \left\{ j,\ j+m,\ j+2m,\ \dotsc,\ j+\left(\f{n}{m}-1\right)m \right\}
\end{align}
for some divisor $m$ of $n$ with $2 \le m < n$ and some $1 \le j \le m$,
where the elements are taken modulo~$n$.

Moreover, the cyclic subgroup $H = \langle x \rangle$ has all of the above sets $B_{m, j}$ as its nontrivial blocks.
\end{lemma}

\begin{proof}
In the following, the elements of $E$ are taken modulo $n$.

Let $B \subset E$ be a nontrivial block of $G$. Since $B$ has at least two elements by
\defref{def:block}, we may choose $j,\ j+d \in B$ with $j \in E$ and $0 < d < n$.
Then $x^{d} \in G$ and $x^{d} \cdot j = j+d \in B$.
Hence, since $(x^{d} \cdot B) \cap B \ne \emptyset$,
it follows from \defref{def:block} that $x^{d} \cdot B = B$.

Applying $x^{d}$ once more, we obtain
$x^{d} \cdot (j+d) = j+2d \in B$.
Repeating this argument, it follows that

\begin{align}
\{j + kd \mid k \in \Zz \} \subset B.
\end{align}

Let $m_{0} = \gcd(n, d)$. Working modulo $n$, we have

\begin{align}
\{ j + kd \mid k \in \Zz \} &= \{j + km_{0} \mid k \in \Zz \} \\
&= \left\{ j,\ j+m_{0},\ j+2m_{0},\ \dotsc,\ j+\left(\f{n}{m_{0}}-1\right)m_{0} \right\} \\
&= B_{m_{0}, j}.
\end{align}
Hence,

\begin{align}
B_{m_{0}, j} \subset B.
\end{align}

If $B_{m_{0}, j} \subsetneq B$, we may choose $j+d' \in B \setminus B_{m_{0}, j}$ with
$m_{0} \nmid d'$. Applying the above argument again, we obtain

\begin{align}
\{ j + kd' \mid k \in \Zz \} \subset B,
\end{align}
and hence
\begin{align}
B_{m_{0}, j} \cup \{ j + kd' \mid k \in \Zz \} \subset B.
\end{align}
Let $m_{1} = \gcd(m_{0}, d')$. Then $m_{1} \mid n$ and we obtain

\begin{align}
B_{m_{0}, j} \cup \{ j + kd' \mid k \in \Zz \} &= \{ j + km_{0} + k'd' \mid k, k' \in \Zz \} \\
&= \{ j + km_{1} \mid k \in \Zz \} = B_{m_{1}, j}.
\end{align}
Hence,
\begin{align}
B_{m_{0}, j} \subsetneq B_{m_{1}, j} \subset B.
\end{align}

This procedure can be repeated for $m_{i}$ ($i \ge 0$), and since $B \subset E$ is finite,
the process terminates. Consequently, we eventually obtain

\begin{align}
B_{m, j} = B
\end{align}
for some divisor $m$ of $n$ with $m < n$.

If $m = 1$, then $B = E$, which is not a nontrivial block.
Therefore, we must have  $2 \le m < n$.

Since $x \in H$, the nontrivial blocks of $H$ are also restricted to those of the form $B_{m,j}$.
Conversely, let $m$ and $j$ satisfy $m \mid n$, $2 \le m < n$, and $1 \le j \le m$.
For any $x^{d} \in H$, we have

\begin{align}
x^{d} \cdot B_{m, j} = \left\{ j+d,\ j+m+d,\ j+2m+d,\ \dotsc,\ j+\left(\f{n}{m}-1\right)m+d \right\} = B_{m, j+d}.
\end{align}

If $m \mid d$, then $B_{m,j+d} = B_{m,j}$, since for each $0 \le k \le n/m-1$,
there exists $0 \le k' \le n/m-1$ such that $j+km+d = j+k'm$.
Hence $x^{d} \cdot B_{m,j} = B_{m,j}$.

If $m \nmid d$, then for each $0 \le k \le n/m-1$,
we have $j+km+d \neq j+k'm$ for any $0 \le k' \le n/m-1$.
Hence $(x^{d} \cdot B_{m,j}) \cap B_{m,j} = \emptyset$.

Therefore, each $B_{m,j}$ satisfies the defining condition of a block of $H$.
Since $1 < \lvert B_{m,j} \rvert = n/m < n$, every $B_{m,j}$ is a nontrivial block of $H$.
\end{proof}

By the above lemma, $G = \langle x, y \rangle$ is imprimitive if and only if
$\langle y \rangle$ admits the residue classes modulo some $m \mid n$
($2 \le m < n$) as a block system.
Equivalently, $y \in I_m$ (see Section~\ref{sec:tbq}).

We now consider what kinds of such $y$ can occur, taking the case $(n, b, q, m) = (18, 3, 6, 6)$ as an example.
In this case, the blocks are the residue classes of $\{1, \dotsc, 18\}$ modulo 6.
Let these blocks be $B_{0}, \dotsc, B_{5}$ (where $B_{0} = \{6, 12, 18\}$).
The permutations $y$ that preserve these blocks fall into the following three types:

\begin{enumerate}
\item\label{itm:block1} $(B_{1}\ B_{1}\ B_{1})(B_{2}\ B_{2}\ B_{2})(B_{3}\ B_{3}\ B_{3})(B_{4}\ B_{4}\ B_{4})(B_{5}\ B_{5}\ B_{5})(B_{0}\ B_{0}\ B_{0})$ \\
\eg $(1\ 7\ 13)(2\ 14\ 8)(3\ 9\ 15)(4\ 10\ 16)(5\ 17\ 11)(6\ 12\ 18)$
\item\label{itm:block2} $(B_{1}\ B_{2}\ B_{3})(B_{1}\ B_{2}\ B_{3})(B_{1}\ B_{2}\ B_{3})(B_{4}\ B_{5}\ B_{0})(B_{4}\ B_{5}\ B_{0})(B_{4}\ B_{5}\ B_{0})$ \\
\eg $(1\ 2\ 15)(7\ 8\ 9)(13\ 14\ 3)(4\ 5\ 6)(10\ 17\ 12)(16\ 11\ 18)$ \\
and those obtained by permuting the blocks.
\item\label{itm:block3} $(B_{1}\ B_{1}\ B_{1})(B_{2}\ B_{2}\ B_{2})(B_{3}\ B_{3}\ B_{3})(B_{4}\ B_{5}\ B_{0})(B_{4}\ B_{5}\ B_{0})(B_{4}\ B_{5}\ B_{0})$ \\
\eg $(1\ 7\ 13)(2\ 8\ 14)(3\ 15\ 9)(10\ 5\ 6)(4\ 17\ 12)(16\ 11\ 18)$ \\
and those obtained by permuting the blocks.
\end{enumerate}

In case \ref{itm:block1}, all permutations are performed within the same blocks.
In case \ref{itm:block2}, the six blocks are divided into two groups of three, and permutations are performed within each group according to a fixed order.
Case \ref{itm:block3} represents a mixed type of the previous two.

These correspond to the partitions of $m = 6$:
\ref{itm:block1} $1 \times 6 = 6$, \ref{itm:block2} $3 \times 2 = 6$, and \ref{itm:block3} $1 \times 3 + 3 \times 1 = 6$.
If there are $t_{i}$ cycles, each involving $d_{i}$ distinct blocks, then $\sum_{i} d_{i} t_{i} = m$.
For the permutation to cycle within each group of $d_{i}$ blocks, it must hold that $d_{i} \mid b$.
Furthermore, since each block (of which there are $n/m = bq/m$) must rotate its elements in units of $b/d_{i}$, it must also hold
$b/d_{i} \mid bq/m$, \ie $m \mid d_{i} q$.
The partitions of $6$ satisfying these conditions are
$\{(d_{i}, t_{i})\} = \{(1, 6)\}, \{(3, 2)\}, \{(1, 3), (3, 1)\}$,
corresponding precisely to the three types described above.

\begin{thm}
\label{thm:imbq}
\thmITa

\begin{align}
\label{eq:diti}
\thmITb
\end{align}
\end{thm}

\begin{proof}
By \lemref{lem:xres}, it suffices to count the number of elements $y$ of cycle type $(b^{q})$
whose blocks under the action of $\langle y \rangle$ are exactly the residue classes modulo $m$.
As discussed above, such a permutation $y$ determines a partition of $m$, as defined by \eqref{eq:diti}.

For each $i$, there are $t_{i}$ cycles of length $b$, each containing $d_{i}$ distinct blocks.
For these $d_{i}$ blocks to circulate within a cycle, we must have $d_{i} \mid b$.
Moreover, since each block (and there are $n/m = bq/m$ such blocks) rotates its elements in units of $b/d_{i}$,
we must have $b/d_{i} \mid bq/m$, that is, $m \mid d_{i} q$.

Fix a partition $\{(d_{i}, t_{i})\}$ of $m$.
First, determine the assignment of blocks corresponding to $d_{1} \cdot t_{1}, d_{2} \cdot t_{2}, \ldots$ and the order of each set of $d_{i}$ blocks (this corresponds to the block arrangement in the preceding example).
There are $m!$ ways to arrange the $m$ blocks in total.
For each pair $(d_{i}, t_{i})$, there are $d_{i}$ ways to represent the same permutation of $d_{i}$ distinct blocks, and $t_{i}!$ ways to order the $t_{i}$ types.
Hence, the number of possible assignments and orderings of blocks is given by

\begin{align}
\label{eq:mdt}
\f{m!}{\prod_{i}d_{i}^{t_{i}}t_{i}!}.
\end{align}

In a cycle of length $b$ containing $d_{i}$ distinct blocks, each cycle consists of $d_{i}$ blocks, with $b/d_{i}$ elements from each block.
The same type of cycle appears $(bq/m) / (b/d_{i}) = d_{i}q/m$ times.
Focusing on one of these $d_{i}$ blocks, since it contains $n/m = bq/m$ elements and these are arranged into $d_{i}q/m$ cycles, the number of such configurations is, by \propref{prop:tbq},

\begin{align}
T\left(\f{\f{bq}{m}}{\f{d_{i}q}{m}}, \ \f{d_{i}q}{m}\right) = \f{\left(\f{n}{m}\right)!}{\left(\f{b}{d_{i}}\right)^{\f{d_{i}q}{m}}\left(\f{d_{i}q}{m}\right)!}.
\end{align}

The number of ways to place the elements of the remaining $d_{i}-1$ blocks after those of the first block is $((n/m)!)^{d_{i}-1}$, since each block contains $n/m$ elements.
Hence, the total number of ways to arrange the elements of these $d_{i}$ blocks is

\begin{align}
 \f{\left(\left(\f{n}{m}\right)!\right)^{d_{i}}}{\left(\f{b}{d_{i}}\right)^{\f{d_{i}q}{m}}\left(\f{d_{i}q}{m}\right)!}.
\end{align}

There are $t_{i}$ such types of blocks, and the number of possible assignments and orderings of blocks is given by \eqref{eq:mdt}.
Summing over all partitions $\{(d_{i}, t_{i})\}$ of $m$, we obtain

\begin{align}
I_{m}(b, q) = \sum_{\{(d_{i}, t_{i})\}}
\prod_{i}\f{m!}{d_{i}^{t_{i}}t_{i}!}\left(\f{\left(\left(\f{n}{m}\right)!\right)^{d_{i}}}{\left(\f{b}{d_{i}}\right)^{\f{d_{i}q}{m}}\left(\f{d_{i}q}{m}\right)!}\right)^{t_{i}}
\q (n = bq),
\end{align}
where the summation is taken over all partitions $\{(d_{i}, t_{i})\}$ of $m$ satisfying \eqref{eq:diti}.

The total products of the terms $(n/m)!$ and $b$ are

\begin{align}
&\prod_{i}\left(\left(\f{n}{m}\right)!\right)^{d_{i}t_{i}} = \left(\left(\f{n}{m}\right)!\right)^{\sum_{i}d_{i}t_{i}} = \left(\left(\f{n}{m}\right)!\right)^{m}, \\
&\prod_{i}\left(\f{1}{b^{\f{d_{i}q}{m}}}\right)^{t_{i}} = \left(\f{1}{b}\right)^{\f{q}{m}\sum_{i}d_{i}t_{i}} = \f{1}{b^{q}}.
\end{align}

Also, $m!$ can be factored out. Then we obtain

\begin{align}
I_{m}(b, q) = \f{m!}{b^{q}}\left(\left(\f{n}{m}\right)!\right)^{m}\sum_{\{(d_{i}, t_{i})\}}
\prod_{i}\f{1}{d_{i}^{t_{i}}t_{i}!}\left(\f{d_{i}^{\f{d_{i}q}{m}}}{\left(\f{d_{i}q}{m}\right)!}\right)^{t_{i}}.
\end{align}
\end{proof}

%%%%%%%%%%%%%%%%%%%%%%%%%%%%%%%%%%%%%%%%%%%%%%%%
\section{Genus 0}

From this section onward, we study the regularity and automorphism groups of uniform \ddes by genus.

In the case of uniform \ddes of genus~0, the situation is simple.

\begin{prop}
\label{prop:reggenus0}
Every uniform \dde of genus~$0$ is regular, and its automorphism group is isomorphic to
its monodromy group and has order~$n$ $($the number of edges$)$.
\end{prop}

\begin{proof}
Let a uniform \dde of genus~$0$ have the passport $[a^{p}, b^{q}, c^{r}]$ with $n = pa = qb = rc$.
Since regularity is invariant under permutations of black vertices, white vertices, and faces, we may assume
$c \ge a \ge b$ (and hence $r \le p \le q$).

Since the genus is $0$, by \eqref{eq:genus-g} we have
\begin{align}
p + q + r = \f{n}{a} + \f{n}{b} + \f{n}{c} = n + 2,
\end{align}
and hence

\begin{align}
\f{1}{a} + \f{1}{b} + \f{1}{c} = 1 + \f{2}{n} > 1.
\end{align}
Since $c \ge a \ge b$, it follows that $b < 3$, and therefore $b = 1$ or $b = 2$.

When $b = 1$, we have

\begin{align}
\f{1}{a} + \f{1}{c} = \f{2}{n},
\end{align}
which implies

\begin{align}
\f{p}{n} + \f{r}{n} = \f{2}{n},
\end{align}
and hence

\begin{align}
p = r = 1.
\end{align}
Therefore, the passport is $[n, 1^{n}, n]$.

When $b = 2$, the integer $n$ is even, and we have
\begin{align}
\f{1}{a} + \f{1}{c} = \f{1}{2} + \f{2}{n} > \f{1}{2}.
\end{align}
Since $2 \le a \le c$, it follows that $a = 2$ or $a = 3$.

\begin{itemize}
\item When $a = 2$, we have $c = n/2$, and hence the passport is $[2^{m}, 2^{m}, m^{2}]$ (with $n = 2m$).
\item When $a = 3$, we have $c = 6n/(n+12)$. The pairs $(n, c)$ satisfying this equation are $(12, 3)$, $(24, 4)$, and $(60, 5)$.
\end{itemize}

From the above, all uniform \ddes of genus~0 are included in the following:

\begin{align}
\label{g0-unip}
[n, 1^{n}, n],\ [2^{m}, 2^{m}, m^{2}]\ (n = 2m),\ [3^{4}, 2^{6}, 3^{4}],\ [3^{8}, 2^{12}, 4^{6}],\ [3^{20}, 2^{30}, 5^{12}].
\end{align}

The dessins corresponding to these passports are shown in \figref{fig:genus0}.
They correspond to the shapes of a star, an $n$-gon, a tetrahedron, a hexahedron, and an icosahedron,
respectively.

\begin{figure}[htbp]
\centering
\includegraphics[width=150mm]{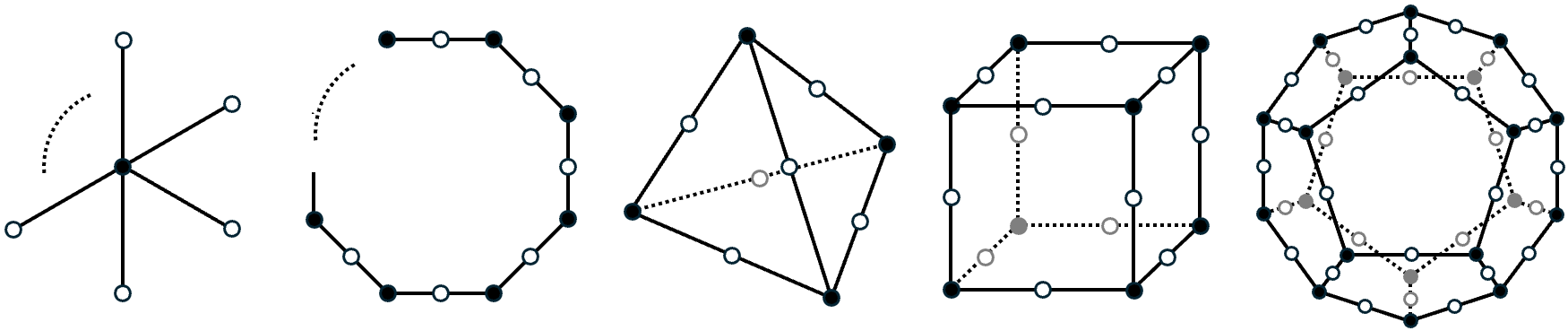} 
\caption{Uniform \ddes of genus~0}
\label{fig:genus0}
\end{figure}

It is easy to show that no other dessins exist for the passports listed in \eqref{g0-unip}, and that their monodromy groups and automorphism groups are isomorphic.
These groups are $C_{n}$ (the cyclic group of order~$n$), $D_{m}$ (the dihedral group of order~$2m$), $A_{4}$ (the alternating group of degree~$4$), $S_{4}$ (the symmetric group of degree~$4$), and $A_{5}$ (the alternating group of degree~$5$),
respectively\footnote{By exchanging black vertices and faces, the passport $[4^{6}, 2^{12}, 3^{8}]$ corresponds to
an octahedron and $[5^{12}, 2^{30}, 3^{20}]$ to a dodecahedron; their monodromy (and hence automorphism)
groups are $S_{4}$ and $A_{5}$, respectively.}.
Therefore, all uniform \ddes of genus~$0$ are regular.

Moreover, by \propref{prop:regaut} and \lemref{lem:order-n}, it follows that
the automorphism group of every uniform \dde of genus~$0$ is isomorphic to its
monodromy group and has order~$n$.
\end{proof}
\HL

Thus, when restricted to genus~0, regularity and uniformity are equivalent.

%%%%%%%%%%%%%%%%%%%%%%%%%%%%%%%%%%%%%%%%%%%%%%%%
\section{Genus 1}

An explicit analysis of regularity and automorphism groups of genus-1 dessins indicates that,
except for passports containing a partition of the form $n^{1}$, all such passports appear to admit
non-regular dessins, while trivial automorphism groups appear not to arise in genus~1.
In this section, we prove these statements.

In general, for a \dde $\msD$ whose valencies of the black vertices, the white vertices, and the faces have
least common multiples $a$, $b$, and $c$, respectively, we consider the triangle group

\begin{align}
\Delta = \Delta(a, b, c) = \langle X, Y, Z \mid X^{a} = Y^{b} = Z^{c} = XYZ = 1\rangle.
\end{align}
Then, for the monodromy group $G = \langle x, y\rangle$ of $\msD$ defined in \defref{def:monog},
one can define an action of $\Delta$ on $\msD$ via the group homomorphism

\begin{align}
\vp\colon X \mapsto x,\ Y \mapsto y,\ Z \mapsto (xy)^{-1},
\end{align}
so that $G$ may be viewed as a quotient of $\Delta$.

Let $\Delta_{e} \ceq \ker \vp$. Then $\Delta_{e}$ is the subgroup of $\Delta$ consisting of all elements that fix
every edge of $\msD$, and it satisfies

\begin{align}
&[\Delta : \Delta_{e}] < \infty.
\end{align}
Moreover,

\begin{align}
\label{eq:regd}
&\msD \text{ is regular} \ \Lra \ \Delta_{e} \triangleleft \Delta.
\end{align}
See \cite[Theorem~3.13]{Jones16}.

For a \dde $\msD$ with a uniform passport $[a^{p}, b^{q}, c^{r}]$ with $n = pa = qb = rc$, the corresponding triangle
group is $\Delta(a,b,c)$.
If $\msD$ has genus~$1$, then by \eqref{eq:genus-g} we have

\begin{align}
p + q + r = \f{n}{a} + \f{n}{b} + \f{n}{c} = n,
\end{align}
and hence

\begin{align}
\f{1}{a} + \f{1}{b} + \f{1}{c} = 1.
\end{align}
This case is referred to as the Euclidean case of triangle groups. In this situation, $\Delta$ is infinite, and
$\Delta_{e}$ is torsion-free (that is, it contains no nontrivial element of finite order).

Assuming $a \le b \le c$, the only possible triples are
\begin{align}
(a, b, c) = (2, 4, 4),\ (3, 3, 3),\ (2, 3, 6).
\end{align}
Hence the uniform passports of genus~1 are exactly

\begin{align}
[2^{2m}, 4^{m}, 4^{m}],\ [3^{m}, 3^{m}, 3^{m}],\ [2^{3m}, 3^{2m}, 6^{m}]\q (m \in \Zp),
\end{align}
together with the ones obtained from these by permuting the three parts.

For these passports, if we set

\begin{align}
d \ceq \max\{ a, b, c \},
\end{align}
then $d=4,3,6$, respectively, and the number of edges of the dessin
(equivalently, the degree of the associated \belyi function) is always $md$.

Consider a triangle $ABC$ in the Euclidean plane with angles
$\angle A = \pi/a$, $\angle B = \pi/b$, and $\angle C = \pi/c$.
The generators $X$, $Y$, $Z$ of the triangle group $\Delta$ correspond to
the following rotations of the whole plane:

\begin{align}
&X\colon \text{rotation by }\f{2\pi}{a}\text{ about the point }A,\\
&Y\colon \text{rotation by }\f{2\pi}{b}\text{ about the point }B,\\
&Z\colon \text{rotation by }\f{2\pi}{c}\text{ about the point }C.
\end{align}

Moreover, $\Delta$ contains as a normal subgroup the group $T\cong \Z^{2}$
consisting of all translations of the plane.

We fix the vertices as $A=(0,0)$ and $B=(1,0)$, and choose the remaining vertex $C$,
the translation subgroup $T$, and the quotient $\Delta/T$ as follows:

\begin{itemize}
\item $\Delta = \Delta(2, 4, 4)$:
\begin{align}
\label{eq:deft4}
C\colon (0, 1), \q T\colon (u, v) \mapsto (u + 2k,\ v + 2l) \ (k, l \in \Z), \q \Delta / T \cong C_{4}
\end{align}
\item $\Delta = \Delta(3, 3, 3)$:
\begin{align}
\label{eq:deft3}
&C\colon \left(\f{1}{2}, \f{\sqrt{3}}{2}\right), \q T\colon (u, v) \mapsto \left(u + \f{3}{2}(k+l),\ v + \f{\sqrt{3}}{2}(k-l)\right) \ (k, l \in \Z),\\
&\Delta / T \cong C_{3}
\end{align}
\item $\Delta = \Delta(2, 3, 6)$:
\begin{align}
\label{eq:deft6}
&C\colon \left(0, \sqrt{3}\right), \q T\colon (u, v) \mapsto (u + 3(k+l),\ v + \sqrt{3}(k-l)) \ (k, l \in \Z), \\
&\Delta / T \cong C_{6}
\end{align}
\end{itemize}

In each case, the elements of $\Delta$ fall into the following categories:

\begin{enumerate}
\item\label{itm:delta1} the identity transformation;
\item\label{itm:delta2} rotations about a point;
\item\label{itm:delta3} translations (elements of $T \setminus \{ \id \}$);
\item\label{itm:delta4} rotations about a point composed with a translation,
which are equivalently rotations about some other point.
\end{enumerate}

Since the elements in \ref{itm:delta1}, \ref{itm:delta2}, and \ref{itm:delta4} have finite order,
every torsion-free subgroup of $\Delta$ must be contained in $T$.
Hence $\Delta_{e} \le T$, and therefore

\begin{align}
[\Delta : \Delta_{e}] = [\Delta : T] \cdot [T : \Delta_{e}] = d \cdot [T : \Delta_{e}].
\end{align}
On the other hand, it is known that $[\Delta : \Delta_{e}]$ equals the number of edges of the dessin. Hence

\begin{align}
[\Delta : \Delta_{e}] = md.
\end{align}
Combining these equalities yields

\begin{align}
\label{eq:tdm}
[T : \Delta_{e}] = m.
\end{align}

Using the results above, we obtain the following proposition.

\begin{prop}
\label{prop:reggenus1}
A uniform passport of genus~$1$ admits a non-regular dessin if and only if it does not contain an element of type $n^{1}$.

In other words, the genus-$1$ uniform passports
$[2^{2m}, 4^{m}, 4^{m}]$,
$[3^{m}, 3^{m}, 3^{m}]$,
$[3^{2m}, 2^{3m}, 6^{m}]$,
and those obtained by permuting their entries
admit a non-regular dessin if and only if $m \ge 2$.
\end{prop}

\begin{proof}
When $m = 1, $ \eqref{eq:tdm} implies $\Delta_{e} = T$. Since $T \triangleleft \Delta$, by \eqref{eq:regd},
the corresponding dessin is always regular.

When $m \ge 2$, \eqref{eq:tdm} shows that $\Delta_{e}$ is a subgroup of $T$ of index~$m$.
Thus, in each of \eqref{eq:deft4}, \eqref{eq:deft3}, and \eqref{eq:deft6}, we may take the subgroup $T'$
obtained by replacing $k$ with $mk'$ ($k' \in \Z$), regard it as $\Delta_{e}$, and consider the dessin
$\msD$ corresponding to this subgroup. In this situation, by \eqref{eq:regd} and the definition
of a normal subgroup, we have

\begin{align}
\msD\text{ is regular} \q \Lra \q T' \triangleleft \Delta
\q \Lra \q \text{for all }g \in \Delta \text{ and all } h \in T',\ ghg^{-1} \in T'.
\end{align}

Taking $g$ to be $Z$ (the rotation about the point $C$ by angle $2\pi/c$) and taking $h$ to be the element of
$T'$ obtained by setting $k=m$, $l=1$ in each of \eqref{eq:deft4}, \eqref{eq:deft3}, and \eqref{eq:deft6},
we compute $ghg^{-1}$ as follows:

\begin{itemize}
\item $\Delta = \Delta(2, 4, 4)$:
\begin{align}
&g\colon (u, v) \mapsto (-v+1,\ u+1), \\
&h\colon (u, v) \mapsto (u + 2m,\ v+2), \\
&ghg^{-1}\colon (u, v) \mapsto (u-2,\ v+2m)
\end{align}
Since $m\ge 2$, $u-2$ cannot be written as $u + 2mk'$ with $k' \in \Z$. Hence $ghg^{-1} \notin T'$.
\item $\Delta = \Delta(3, 3, 3)$:
\begin{align}
&g\colon (u, v) \mapsto \left(-\f{1}{2}u-\f{\sqrt{3}}{2}v+\f{3}{2},\ \f{\sqrt{3}}{2}u-\f{1}{2}v+\f{\sqrt{3}}{2}\right), \\
&h\colon (u, v) \mapsto \left(u + \f{3}{2}(m+1),\ v + \f{\sqrt{3}}{2}(m-1)\right), \\
&ghg^{-1}\colon (u, v) \mapsto \left(u-\f{3}{2}m,\ v+\f{\sqrt{3}}{2}(m+2)\right)
\end{align}
Thus the corresponding translation parameters satisfy $k+l=-m$ and $k-l=m+2$, hence $k=1$.
Since $m \ge 2$, this cannot be expressed in the form $k = mk'$ with $k' \in \Z$; therefore $ghg^{-1} \notin T'$.
\item $\Delta = \Delta(2, 3, 6)$:
\begin{align}
&g\colon (u, v) \mapsto \left(\f{1}{2}u-\f{\sqrt{3}}{2}v+\f{3}{2},\ \f{\sqrt{3}}{2}u+\f{1}{2}v+\f{\sqrt{3}}{2}\right), \\
&h\colon (u, v) \mapsto \left(u + 3(m+1),\ v + \sqrt{3}(m-1)\right), \\
&ghg^{-1}\colon (u, v) \mapsto \left(u+3,\ v+(2m+1)\sqrt{3}\right)
\end{align}
Here the translation satisfies $k+l=1$ and $k-l=2m+1$, hence $k=m+1$.
Since $m \ge 2$, this cannot be expressed in the form $k = mk'$ with $k' \in \Z$; therefore $ghg^{-1} \notin T'$.
\end{itemize}

In each case we conclude that $T'$ is not normal in $\Delta$. Hence, by \eqref{eq:regd}, the corresponding
dessin $\msD$ is not regular. This proves that for every $m \ge 2$, for each of \eqref{eq:deft4}, \eqref{eq:deft3},
and \eqref{eq:deft6}, a non-regular dessin of genus~$1$ with the given uniform passport exists.

Combining all the above, we conclude that the existence of a non-regular dessin is equivalent to
the condition $m \ge 2$.
\end{proof}

For $m \ge 2$, a regular dessin does not necessarily exist.
There are genus-1 passports for which only non-regular dessins exist; see \exref{ex:c15} below.

%%%%%%%%%%%%

For the automorphism group $\AD$ of a dessin $\msD$, it is known that

\begin{align}
\label{eq:nautd}
\Aut \msD \cong N_{\Delta}(\Delta_{e}) / \Delta_{e},
\end{align}
where

\begin{align}
N_{\Delta}(\Delta_{e}) \ceq \{ g \in \Delta \mid g\Delta_{e} g^{-1}=\Delta_{e} \}
\end{align}
is the normalizer of $\Delta_{e}$ in $\Delta$ (see \cite[Theorem 2.2]{Jones16}).

Using this, we obtain the following proposition.

\begin{prop}
\label{prop:autgenus1}
For any uniform \dde $\msD$ of genus~$1$, we have $\AD \ncong \{1\}$.
\end{prop}

\begin{proof}

As described above, the passports of genus-1 dessins are exactly
$[2^{2m}, 4^{m}, 4^{m}]$, $[3^{m}, 3^{m}, 3^{m}]$, and $[2^{3m}, 3^{2m}, 6^{m}]$
together with those obtained by permuting their entries, and they satisfy
$[T : \Delta_{e}] = m$.

For $m = 1$, \propref{prop:reggenus1} implies that $\msD$ is regular.
Therefore, by \propref{prop:regaut}, $\Aut \msD$ is isomorphic to the monodromy group of $\msD$.
Moreover, by \lemref{lem:order-n}, the order of the monodromy group equals the number of edges of the dessin.
Thus $\lvert \AD \rvert > 1$, that is, $\AD \ncong \{1\}$.

For $m \ge 2$, since $T$ commutes with every element of $\Delta$, we have
$g \Delta_{e} g^{-1} = \Delta_{e}$ for any $g \in T$. Hence

\begin{align}
T \le N_{\Delta}(\Delta_{e}).
\end{align}
Together with \eqref{eq:nautd}, this implies

\begin{align}
\Aut \msD \cong N_{\Delta}(\Delta_{e}) / \Delta_{e} \ge T / \Delta_{e}.
\end{align}
Since $\lvert T / \Delta_{e} \rvert = m \ge 2$, it follows that $\lvert \AD \rvert \ge 2$, and therefore
$\Aut \msD \not\cong \{ 1 \}$.

Thus, for all $m \ge 1$, we have $\Aut \msD \ncong \{1\}$.
\end{proof}

%%%%%%%%%%%%%%%%%%%%%%%%%%%%%%%%%%%%%%%%%%%%%%%%
\section{Regularity Criteria in Arbitrary Genus}

In this section, we describe two genus-independent criteria for the existence of regular dessins
associated with uniform passports.

In what follows, let $\sigma_{n} \ceq \sigman \in S_{n}$.

%%%%%%%%%%%%%%%%%%%%%%%%%%%%%%%%%%%%
\subsection{Tree Cases}

A passport $[\lambda_{0}, \lambda_{1}, \lambda_{\infty}]$ is called a \emph{tree} if at least one of the $\lambda_{i}$ is $n^{1}$.

\begin{thm}
\label{thm:trees}
\thmtrees
\end{thm}

\begin{proof}
As regularity is symmetric with respect to black vertices, white vertices, and faces, it suffices to prove
the statement for $[a^{p}, b^{q}, n]$ ($p, q \ge 1$).

By \eqref{eq:genus}, the genus is $(n - (p + q) + 1)/2$. Hence $n - (p + q)$ is odd.

Since $xy$ and $(xy)^{-1}$ have the same cycle type, it follows that $xy$ is of cycle type $(n)$.
Hence we may assume $xy = \sigma_n$.
Since $\langle \sigma_{n} \rangle \cong C_{n}$ and $\langle \sigma_{n} \rangle \le G$, and since a regular dessin satisfies $\lvert G \rvert = n$ by \lemref{lem:order-n}, it follows that if the dessin is regular, then
$G = \langle \sigma_{n} \rangle \cong C_{n}$.
Consequently, both $x$ and $y$ must lie in $\langle \sigma_{n} \rangle$.

Therefore,

\begin{align}
\text{The dessin is regular} \Lra{}&x, y \in \langle \sigma_{n} \rangle, \ x\colon \text{type }(a^{p}), \ y\colon \text{type }(b^{q}),\ xy = \sigma_{n} \\
\Lra{}&\text{There exist }l, m \text{ such that }1 \le l < a, \ 1 \le m < b, \\
&x = \sigma_{n}^{lp}, \ \gcd(a, l) = 1, \ y = \sigma_{n}^{mq},\ \gcd(b, m) = 1, \\
\label{eq:reglm}
&lp+mq \equiv 1 \npmod{n}.
\end{align}

When the dessin is regular, since $lp + mq \equiv 1 \pmod{n}$ and $p, q \mid n$, it follows that $\gcd(p, q) = 1$.

Conversely, assume that $\gcd(p, q) = 1$.
Since $n = pa = qb$, we have $q \mid a$ and $p \mid b$.
Let $a = uq$ with $u \in \Zp$. Then we have $n = upq$ and $b = up$.

Moreover, since $\gcd(p, q) = 1$, there exist $\A, \B \in \Z$ such that $p\A + q\B = n + 1$.
For these, we have

\begin{align}
p(\A - kq) + q(\B + kp) = n + 1 \quad \text{for all } k \in \Z.
\end{align}

Let $k_{0}$ be the smallest $k$ such that $1 \le \A - kq < a$ and $1 \le \B + kp < b$, and set
$l_{0} = \A - k_{0}q$, $m_{0} = \B + k_{0}p$. Then we obtain

\begin{align}
l_{0}p + m_{0}q &= n+1, \\
(l_{0}-q)p + (m_{0}+p)q &= n+1, \\
&\cdots \\
(l_{0}-(u-1)q)p + (m_{0}+(u-1)p)q &= n+1,
\end{align}
that is, $u$ such equations.

The $u$ pairs of coefficients $(l_{0} - kq, m_{0} + kp)$ ($0 \le k \le u - 1$) all satisfy
$1 \le l_{0} - kq < a$ and $1 \le m_{0} + kp < b$.

To establish \eqref{eq:reglm}, it suffices to show that there exists some $k_{1}$ such that

\begin{align}
\gcd(a, l_{0} - k_{1}q) = \gcd(b, m_{0} + k_{1}p) = 1,\  0 \le k_{1} \le u - 1.
\end{align}

When $(l_{0} - kq)p + (m_{0} + kp)q = n + 1$, we have $(l_{0} - kq)p \equiv 1 \pmod{q}$, and hence $\gcd(q, l_{0} - kq) = 1$.
Since $a = uq$, if $\gcd(u, l_{0} - kq) = 1$, then $\gcd(a, l_{0} - kq) = 1$ also holds.

Similarly, if $\gcd(u, m_{0} + kp) = 1$, then we have $\gcd(b, m_{0} + kp) = 1$.

Therefore, it suffices to show that there exists $k$ such that

\begin{align}
\label{eq:reglm2}
\gcd(u, l_{0} - kq) = \gcd(u, m_{0} + kp) = 1, \quad 0 \le k \le u - 1.
\end{align}

First, when $u = 1$,
%we have $\gcd(1, l_{0}) = \gcd(1, m_{0}) = 1$, therefore
$k = 0$ satisfies \eqref{eq:reglm2}.

When $u \ge 2$, let the prime factors of $u$ be $p_{1}, \dotsc, p_{s}$ ($s \ge 1$, $p_{1} < \cdots < p_{s}$).
Since $\gcd(p, q) = 1$, no $p_{i}$ divides both $p$ and $q$. Hence, we can choose $0 \le k_{1} \le u - 1$ such that

\begin{align}
\text{for all } 1 \le i \le s, \quad
\label{eq:k1}
\begin{dcases}
k_{1} \not\equiv l_{0} q^{-1} \npmod{p_{i}} & (p_{i} \mid p,\ p_{i} \nmid q) \\
k_{1} \not\equiv -m_{0} p^{-1} \npmod{p_{i}} & (p_{i} \nmid p,\ p_{i} \mid q) \\
k_{1} \not\equiv l_{0} q^{-1},\ -m_{0} p^{-1} \npmod{p_{i}} & (p_{i} \nmid p, q)
\end{dcases}.
\end{align}

Note that if $p_{i} = 2$, then $u$ is even, and hence $n$ is even.
Since $n - (p + q)$ is odd, it follows that one of $p$ or $q$ must be even.
Therefore, if $p_{i} \nmid p, q$, we must have $p_{i} \ge 3$, and hence a $k_{1}$ satisfying
\eqref{eq:k1} can always be chosen.

For this $k_{1}$ and each $p_{i}$, we have:

\begin{itemize}
\item When $p_{i} \mid p$ and $p_{i} \nmid q$, \\
$l_{0} - k_{1} q \not\equiv l_{0} - l_{0} q^{-1} q \equiv 0 \pmod{p_{i}}$, \\
and from $(l_{0} - k_{1} q)p + (m_{0} + k_{1} p)q = n + 1$, we obtain $(m_{0} + k_{1} p)q \equiv 1 \pmod{p_{i}}$,
hence $m_{0} + k_{1} p \not\equiv 0 \pmod{p_{i}}$.
\item When $p_{i} \nmid p$ and $p_{i} \mid q$, \\
$m_{0} + k_{1} p \not\equiv m_{0} - m_{0} p^{-1} p \equiv 0 \pmod{p_{i}}$, \\
and from $(l_{0} - k_{1} q)p + (m_{0} + k_{1} p)q = n + 1$, we obtain $(l_{0} - k_{1} q)p \equiv 1 \pmod{p_{i}}$,
hence $l_{0} - k_{1} q \not\equiv 0 \pmod{p_{i}}$.
\item When $p_{i} \nmid p, q$, \\
$l_{0} - k_{1} q \not\equiv l_{0} - l_{0} q^{-1} q \equiv 0 \pmod{p_{i}}$, \\
$m_{0} + k_{1} p \not\equiv m_{0} - m_{0} p^{-1} p \equiv 0 \pmod{p_{i}}$.
\end{itemize}

Hence, for every prime divisor $p_{i}$ of $u$, we have $p_{i} \nmid l_{0} - k_{1} q$ and $p_{i} \nmid m_{0} + k_{1} p$,
thus $k = k_{1}$ satisfies \eqref{eq:reglm2}.

Therefore, when $\gcd(p, q) = 1$, a regular dessin exists.
\end{proof}

In genus~$0$, the tree case corresponds to $[n, 1^{n}, n]$.
Moreover, the following result holds from \thmref{thm:trees}.

\begin{cor}
The passport $[n, b^{q}, n]$ always admits a regular dessin.
On the other hand, the passport $[a^{p}, a^{p}, n]$ admits a regular dessin if and only if $p = 1$, equivalently $a = n$.
\end{cor}

%%%%%%%%%%%%%%%%%%%%%%%%%%%%%%%%%%%%
\subsection{General Criteria}

The following holds regardless of whether the passport corresponds to a tree or not.

\begin{prop}
\label{prop:absgroup}
A uniform passport $[a^{p}, b^{q}, c^{r}]$, where $n = pa = qb = rc$, admits a regular dessin if and only if
there exists an $($abstract$)$ group $G$ of order~$n$ satisfying the following condition:

\begin{itemize}
\item There exist elements $x, y \in G$ such that $\ord(x) = a$, $\ord(y) = b$, $\ord(xy) = c$, and $G = \langle x, y \rangle$.
\end{itemize}
\end{prop}

\begin{proof}
Assume that a uniform passport $[a^{p}, b^{q}, c^{r}]$ admits a regular dessin.
As in \defref{def:monog}, we define permutations $x$ and $y$ acting on the set of edges of the dessin,
and set $G = \langle x, y \rangle$, the monodromy group.
Then $\ord(x) = a$, $\ord(y) = b$, and $\ord(xy) = \ord((xy)^{-1}) = c$.
By \lemref{lem:order-n}, we have $\lvert G \rvert = n$.

Conversely, suppose that a group $G$ and elements $x, y \in G$ satisfy $\lvert G \rvert = n$, $\ord(x) = a$, $\ord(y) = b$, $\ord(xy) = c$, and $G = \langle x, y \rangle$.
Take an arbitrary bijection $\psi\colon E \to G$ on the set of edges $E = \{1, \dotsc, n\}$, and define an action of $G$ on $E$ by
\begin{align}
\text{for all } g \in G \text{ and all } e \in E, \quad g \cdot e \ceq \psi^{-1}(g\psi(e)).
\end{align}
For any $g, g' \in G$ and $e \in E$, we have

\begin{align}
\id\cdot\, e &= \psi^{-1}(\id\psi(e)) = \psi^{-1}(\psi(e)) = e, \\
g' \cdot (g \cdot e) &= g' \cdot \psi^{-1}(g\psi(e)) = \psi^{-1}(g'\psi(\psi^{-1}(g\psi(e)))) = \psi^{-1}(g'g\psi(e)) = (g'g)\cdot e.
\end{align}
Therefore, this defines a group action of $G$ on $E$.

Moreover, for any $e_{1}, e_{2} \in E$, consider the element $\psi(e_{2})(\psi(e_{1}))^{-1} \in G$. Then

\begin{align}
\psi(e_{2})(\psi(e_{1}))^{-1} \cdot e_{1} = \psi^{-1}(\psi(e_{2})(\psi(e_{1}))^{-1}\psi(e_{1})) = \psi^{-1}(\psi(e_{2})) = e_{2}.
\end{align}
Hence this action is transitive.

Therefore, there exists a dessin whose monodromy group is $G$\cite[Theorem~3.6]{Scodro24}.
Since $\lvert G \rvert = n$, it follows from \lemref{lem:order-n} that the dessin is regular.
Moreover, the orders of $x$, $y$, and $xy$ are $a$, $b$, and $c$, respectively, and since the dessin is uniform by
\propref{prop:reguni}, its passport is $[a^{p}, b^{q}, c^{r}]$ with $n = pa = qb = rc$.
\end{proof}

\begin{exa}
A regular \dde with the passport $[3^{4}, 3^{4}, 3^{4}]$ ($n = 12$, genus~1) must satisfy $\ord(x) = \ord(y) = \ord(xy) = 3$. 

Furthermore, since the action of $G = \langle x, y \rangle$ must be transitive, we have $x \ne y$.
Also, as $x, y \ne 1$, it follows that $xy \ne x$ and $xy \ne y$.
Hence $x$, $y$, and $xy$ must all be distinct, and therefore $G$ must be of order~12 and contain at least three elements of order~3.

There are five abstract groups of order~12: the cyclic group $C_{12}$, the direct product $C_{6} \times C_{2}$,
the alternating group $A_{4}$, the dihedral group $D_{6}$, and the dicyclic group $Q_{12}$. Among these,
 only $A_{4}$ contains three or more elements of order~3 (in fact, four).
By examining the elements of $A_{4}$, for example, we find that

\begin{align}
x &= (1\ 2\ 3)(4\ 5\ 6)(7\ 8\ 9)(10\ 11\ 12), \\
y &= (1\ 4\ 7)(2\ 11\ 6)(3\ 8\ 10)(5\ 12\ 9), \\
xy &= (1\ 5\ 10)(2\ 12\ 7)(3\ 9\ 6)(4\ 8\ 11)
\end{align}
satisfy the required conditions, showing that a regular dessin exists in this passport.
\end{exa}

\begin{exa}
\label{ex:c15}
A regular \dde with the passport $[3^{5}, 3^{5}, 3^{5}]$ ($n = 15$, genus~1) must satisfy $\ord(x) = \ord(y) = \ord(xy) = 3$.
Since the only group of order~15 is the cyclic group $C_{15}$, if $\sigma$ is a generator of $C_{15}$, then $x$, $y$,
and $xy$ must each be one of $\sigma^{5}$ or $\sigma^{10}$. However, $\{ \sigma^{5}, \sigma^{10} \}$ does not generate
$C_{15}$.

Therefore, no regular dessin exists for the passport $[3^{5}, 3^{5}, 3^{5}]$.

By the same reasoning, no regular dessin exists for the passport $[5^{3}, 5^{3}, 5^{3}]$ ($n = 15$, genus~4) either.
\end{exa}

%%%%%%%%%%%%%%%%%%%%%%%%%%%%%%%%%%%%%%%%%%%%%%%%
% Force newpage
\needspace{5\baselineskip}
\section{Automorphism Groups in Genus $\ge 2$}

\label{sec:autgge2}
In this section, we prove the existence of a dessin $\msD$ with $\AD \cong \{ 1 \}$
for passports of genus at least~$2$ in which at least two of the entries are of type $n^{1}$.

%%%%%%%%%%%%%%%%%%%%%%%%%%%%%%%%%%%%
\subsection{Passports $[n, n, n]$}

\begin{prop}
\label{prop:nnn}
Every passport $[n, n, n]$ of genus~$\ge 2$ admits a dessin with a trivial automorphism group.
\end{prop}

\begin{proof}
By \eqref{eq:genus}, the genus is $(n-3)/2 + 1 = (n-1)/2$. Since it is an integer at least~$2$,
$n$ is odd and $n \ge 5$.

Since the monodromy group does not depend on the labeling of the edges, we may fix
$x = \sigman$.

Let

\begin{align}
y = (2\ 4\ 6\ \ldots\ n{-}1\ n\ n{-}2\ \ldots\ 3\ 1),
\end{align}
where the even numbers appear in ascending order, followed by the odd numbers in descending order. Then

\begin{align}
xy = (1\ 3\ 2\ 5\ 4\ \ldots\ n{-}2\ n{-}3\ n\ n-1),
\end{align}
which is of type $(n)$.

For any $i, j \in E = \{ 1, \dotsc, n \}$, we have $x^{j-i}\cdot i = j$. Hence, $G = \langle x, y \rangle$ acts
transitively on $E$.

Therefore, $G$ is the monodromy group of a dessin $\msD$ with passport $[n, n, n]$.
Since $n$ is odd, we have $G \le A_n$, the alternating group.

Computing $x^{2}(xyxy)^{-1}$ yields $(1\ 2\ 3)$. As is well known, the permutations $x$ and $(1\ 2\ 3)$
generate $A_n$ for odd $n \ge 5$, and hence $G = A_n$. Therefore,

\begin{align}
\AD \cong C_{S_{n}}(A_{n}) \cong \{1\}.
\end{align}
\end{proof}

%%%%%%%%%%%%%%%%%%%%%%%%%%%%%%%%%%%%
\subsection{Approach for the Passports $[n, b^{q}, n]\ (q \ge 2)$}

In this section, we describe our approach to proving that every uniform passport
$[a^{p}, b^{q}, c^{r}]$ of genus at least~2, with exactly two of $p$, $q$, $r$ equal to 1,
admits a \dde with trivial automorphism group.
The proof combines counting arguments, analytic estimates, and a finite verification of
exceptional cases.

Since the automorphism group is invariant under interchanging black vertices, white vertices,
and faces, it suffices to consider the case $[n, b^{q}, n]$, where $n = bq$ and $q \ge 2$.
Since the genus $(n-q)/2$ is an integer at least~$2$, we have $n - q \ge 4$, $n \equiv q \pmod{2}$,
and hence $b = n/q \ge 2$.

In this case, unlike in \propref{prop:nnn}, it is difficult to construct a general pair $(x, y)$ such that $\Aut \msD \cong \{1\}$.
Therefore, we exploit the relationship between the primitivity (see \defref{def:block} and \ref{def:prim})
of the monodromy group and the triviality of the automorphism group.

\begin{prop}
\label{prop:primaut}
Let $G = \langle x, y \rangle$ be the monodromy group of a dessin $\msD$ with $n$ edges.
If $n$ is not prime and $G$ is primitive, the automorphism group $\Aut \msD$ is trivial.
\end{prop}

\begin{proof}
Assume that $n$ is not prime, $G$ is primitive, and let $c \in \Aut \msD$. Then, for every $g \in G$, we have $cg = gc$.

For $e \in \{1, \ldots, n\}$, let $O_e = \langle c \rangle \cdot e$ denote the orbit of $e$ under the action of
 $\langle c \rangle$. For any $g \in G$, we have

\begin{align}
g \cdot O_{e} = g \cdot (\langle c \rangle \cdot e) = (g \langle c \rangle) \cdot e = (\langle c \rangle g) \cdot e = \langle c \rangle \cdot (g \cdot e) = O_{g \cdot e}.
\end{align}

Assume that $(g \cdot O_{e}) \cap O_{e} \ne \emptyset$.
Since $g \cdot O_{e} = O_{g \cdot e}$ and each element of $O_{e}$ can be written as $c^{m} \cdot e$, while each element of $O_{g \cdot e}$ can be written as $c^{l} g \cdot e$, there exist integers $l, m$ such that $c^{l} g \cdot e = c^{m} \cdot e$.
Then $g \cdot e = c^{m-l} \cdot e \in O_{e}$, and hence $g \cdot O_{e} \subset O_{e}$.
Similarly, since $e = c^{l-m}g \cdot e \in O_{g \cdot e}$, we have $O_{e} \subset O_{g \cdot e}$.
Therefore, $g \cdot O_{e} = O_{g \cdot e} = O_{e}$.

Since this holds for all $g$ and all $e$, we have, for every $g$ and $e$, that either $g \cdot O_{e} = O_{e}$
or $(g \cdot O_{e}) \cap O_{e} = \emptyset$.
Hence, each element of the family $\{ O_{e} \}$ is a block for the action of $G$.

Furthermore, since $G$ acts transitively, all blocks are mapped to one another under
the action of $G$ and therefore have the same size $d$.
Since the size of the orbit of $c$ is $d$, the cycle type of $c$ is $(d^{n/d})$.

If $1 < d < n$, then $G$ has nontrivial blocks of size $d$, contradicting the assumption that $G$ is primitive.

If $d = n$, then $c$ has cycle type $(n)$.
It is well known that $C_{S_{n}}(c) = \langle c \rangle$.
Since $cg = gc$ for every $g \in G$, it follows that $G \le \langle c \rangle$,
and hence $G$ is cyclic.
A cyclic group acts primitively only when $n$ is prime; since $n$ is not prime,
this yields a contradiction.

Consequently, we have $d = 1$, and the cycle type of $c$ is $(1^{n})$; that is, $c = \id$. Therefore, $\Aut \msD \cong \{1\}$.
\end{proof}

For the passport $[n, b^{q}, n]$ with $b \ge 2$, fix $x = \sigman \in S_{n}$.
As described in the proof of \propref{prop:nnn}, the group $G = \langle x, y \rangle$ acts transitively on the edges
$\{ 1, \dotsc , n \}$ for any $y \in S_{n}$. Therefore, if $y$ has cycle type $(b^{q})$ and
$xy$ has cycle type $(n)$, then $G$ is the monodromy group of a dessin corresponding to this passport.

We define the following subsets of $S_{n}$:

\begin{itemize}
\item $T(b, q)$: the set of elements of $S_{n}$ with cycle type $(b^{q})$;
\item $N(b, q)$: the set of elements $y \in T(b, q)$ such that $xy$ has cycle type $(n)$;
\item $I(b, q)$: the set of elements $y \in T(b, q)$ such that $G = \langle x, y \rangle$ is imprimitive;
\item $C(b, q)$: the set of elements $y \in T(b, q)$ such that $C_{S_{n}}(G)$ is nontrivial (equivalently, $\Aut \msD \ncong \{1\}$).
\end{itemize}

$T$, $N$, and $I$ are the same as those defined in Section~\ref{sec:tbq}.

Since $n = bq$ is not prime, it follows from \propref{prop:primaut} that $I \supset C$.

The existence of a dessin $\msD$ such that $\Aut \msD \cong \{1\}$ is equivalent to $N \cap C^{c} \ne \emptyset$
 (the gray region in \figref{fig:tnic}).
To prove this, however, it is sufficient to show that $N \cap I^{c} \ne \emptyset$.
Furthermore, to establish this, it suffices to show that $\lvert N \rvert > \lvert I \rvert$.

\begin{figure}[htbp]
\centering
\includegraphics[width=66mm]{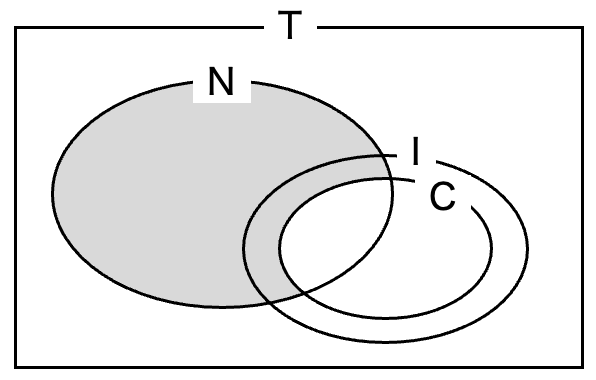} 
\caption{The region corresponding to trivial automorphism groups}
\label{fig:tnic}
\end{figure}

Hereafter, we use the same notation to denote the number of elements in a set (for example,
the number of elements of $T(b, q)$ is denoted by the function $T(b, q)$).

For a passport $[n, b^{q}, n] \ (b \ge 2)$ of genus at least~2, we show that

\begin{align}
\label{eq:itnt}
\f{I(b, q)}{T(b, q)} < \f{N(b, q)}{T(b, q)},
\end{align}
which implies that there always exists a dessin $\msD$ whose automorphism group $\Aut \msD$ is trivial.

By directly computing the automorphism groups of dessins belonging to the corresponding passports,
we find that most of them are in fact trivial, and that this tendency becomes more pronounced as $n$ increases.
Accordingly, we allow a slightly coarse upper bound for $I/T$,
and in the remaining cases where \eqref{eq:itnt} cannot be established for finitely many small pairs $(b, q)$,
we complete the proof by explicitly exhibiting pairs $(x, y)$ for which the automorphism group is trivial.

As in Section~\ref{sec:tbq}, let $I_{m}(b, q)$, where $m \mid n$ and $2 \le m < n$, denote the subset of
$I(b, q)$ consisting of elements $y$ such that $G = \langle x, y \rangle$ is imprimitive
with respect to the residue classes modulo~$m$.

\propref{prop:tbq} and \thmref{thm:imbq} give an explicit formula for $I_{m}/T$ and provide
an upper bound for $I/T$:

\begin{align}
\label{eq:imbqtbq}
\f{I_{m}(b, q)}{T(b, q)} &= \f{q!}{n!}m!\left(\left(\f{n}{m}\right)!\right)^{m}\sum_{\{(d_{i}, t_{i})\}}
\prod_{i}\f{1}{d_{i}^{t_{i}}t_{i}!}\left(\f{d_{i}^{\f{d_{i}q}{m}}}{\left(\f{d_{i}q}{m}\right)!}\right)^{t_{i}}, \\
\label{eq:ibqtbq}
\f{I(b, q)}{T(b, q)} &\le \f{q!}{n!}\sum_{\substack{2\le m<n\\ m\mid n}}m!\left(\left(\f{n}{m}\right)!\right)^{m}\sum_{\{(d_{i}, t_{i})\}}
\prod_{i}\f{1}{d_{i}^{t_{i}}t_{i}!}\left(\f{d_{i}^{\f{d_{i}q}{m}}}{\left(\f{d_{i}q}{m}\right)!}\right)^{t_{i}},
\end{align}
where the summations over $\{(d_{i}, t_{i})\}$ are taken over all partitions $\{(d_{i}, t_{i})\}$ of $m$ satisfying
\eqref{eq:diti}, that is:

\begin{align}
\thmITb
\end{align}

Since there exist pairs $(x, y)$ that admit blocks corresponding to multiple values of $m$, the total sum on the right-hand side of the latter generally exceeds $I/T$.

Since 

\begin{align}
\sum_{i} t_{i} \f{d_{i} q}{m} = \f{q}{m} \sum_{i} d_{i} t_{i} = q,
\end{align}
 it follows from the logarithmic convexity of the factorial
(equivalently, Gamma function) that the sum of the terms $\log (d_{i} q/m)$, each counted with multiplicity $t_{i}$, is not
smaller than the sum of $m$ terms of $\log (q/m)$, where $m$ is the total number of terms. Hence,

\begin{align}
&\log\prod_{i}\left(\left(\f{d_{i}q}{m}\right)!\right)^{t_{i}} = \sum_{i}t_{i}\log\left(\f{d_{i}q}{m}\right)! \ge m\log\left(\f{q}{m}\right)!
\end{align}
where, since $q/m$ is not necessarily an integer, we define $(q/m)! \ceq \Gamma(q/m + 1)$.

Therefore,

\begin{align}
&\prod_{i}\left(\f{1}{\left(\f{d_{i}q}{m}\right)!}\right)^{t_{i}} \le \f{1}{\left(\left(\f{q}{m}\right)!\right)^{m}}.
\end{align}
This yields

\begin{align}
\label{eq:imconv}
\f{I_{m}(b,q)}{T(b,q)} \le \f{q!}{n!}m!\left(\f{\left(\f{n}{m}\right)!}{\left(\f{q}{m}\right)!}\right)^{m}\sum_{\{(d_{i}, t_{i})\}}
\prod_{i}\f{d_{i}^{\left(\f{d_{i}q}{m}-1\right)t_{i}}}{t_{i}!}.
\end{align}
Applying Stirling's formula

\begin{align}
\label{eq:stirling}
\sqrt{2\pi M}\left(\f{M}{e}\right)^{M}e^{\f{1}{12M+1}} \le M! \le \sqrt{2\pi M}\left(\f{M}{e}\right)^{M}e^{\f{1}{12M}} \quad (M > 0),
\end{align}
we obtain

\begin{align}
\f{q!}{n!}m!\left(\f{\left(\f{n}{m}\right)!}{\left(\f{q}{m}\right)!}\right)^{m}
&\le m!\cdot \f{\sqrt{2\pi q}\left(\f{q}{e}\right)^{q}\exp\left(\f{1}{12q}\right)}{\sqrt{2\pi n}\left(\f{n}{e}\right)^{n}\exp\left(\f{1}{12n+1}\right)}\cdot
\f{\left(\sqrt{2\pi\f{n}{m}}\right)^{m}\left(\f{n}{me}\right)^{n}\exp\left(\f{m}{12\f{n}{m}}\right)}
{\left(\sqrt{2\pi\f{q}{m}}\right)^{m}\left(\f{q}{me}\right)^{q}\exp\left(\f{m}{12\f{q}{m}+1}\right)} \\
&= m!\,b^{\f{m-1}{2}}m^{-q(b-1)}\exp\left(\f{1}{12q}+\f{m^{2}}{12n}-\f{1}{12n+1}-\f{m^{2}}{12q+m}\right).
\end{align}

If we regard the exponent of $\exp$ as a function of $m$ ($2 \le m \le n$), denoted by $f(m)$, then

\begin{align}
f(m) &= \f{1}{12q}+\f{m^{2}}{12n}-\f{1}{12n+1}-\f{m^{2}}{12q+m}, \\
f'(m) &= m\left(\f{1}{6n}-\f{m+24q}{(m+12q)^{2}}\right).
\end{align}
Let $g(m) = -(m + 24q)/(m + 12q)^{2}$. Then $g'(m) = (m+36q)/(m+12q)^{3} > 0$,
thus $g(m)$ is monotonically increasing. Hence,

\begin{align}
f'(m) &\le f'(n) = n\left(\f{1}{6n}-\f{n+24q}{(n+12q)^{2}}\right) = bq\left(\f{1}{6bq}-\f{bq+24q}{(bq+12q)^{2}}\right) \\
&= b\left(\f{1}{6b}-\f{b+24}{(b+12)^{2}}\right) = - \f{5b^{2}+120b-144}{6(b+12)^{2}} < 0,
\end{align}
where the last inequality follows from $b \ge 2$.

Therefore, since $f(m)$ is monotonically decreasing, we have

\begin{align}
f(m) \le f(2) &= \f{1}{12q} + \f{1}{3n} - \f{1}{12n+1} - \f{2}{6q+1} < \f{1}{12q} + \f{1}{3n} - \f{2}{6q+1} \\
&= \f{1}{12q} + \f{1}{3bq} - \f{2}{6q+1} = \f{-18bq+b+24q+4}{12bq(6q+1)} = - \f{18\left(b-\f{4}{3}\right)\left(q-\f{1}{18}\right)-\f{16}{3}}{12bq(6q+1)} \\
&\le - \f{18\left(2-\f{4}{3}\right)\left(2-\f{1}{18}\right)-\f{16}{3}}{12bq(6q+1)} < 0,
\end{align}
where the last inequality follows from $b, q \ge 2$.

Hence, $\exp(f(m)) < 1$, and therefore

\begin{align}
\label{eq:factqnm}
&\f{q!}{n!}m!\left(\f{\left(\f{n}{m}\right)!}{\left(\f{q}{m}\right)!}\right)^{m} \le m!\,b^{\f{m-1}{2}}m^{-q(b-1)},
\end{align}
and then,

\begin{align}
\label{eq:imconv2}
&\f{I_{m}(b,q)}{T(b,q)} \le m!\,b^{\f{m-1}{2}}m^{-q(b-1)}S(b, q, m), \quad S(b, q, m) \ceq \sum_{\{(d_{i}, t_{i})\}}
\prod_{i}\f{d_{i}^{\left(\f{d_{i}q}{m}-1\right)t_{i}}}{t_{i}!},
\end{align}
where the summation in $S(b, q, m)$ is taken over all partitions $\{(d_{i}, t_{i})\}$ of $m$ satisfying \eqref{eq:diti}.

Applying \eqref{eq:stirling} to $m!$ as well, we obtain

\begin{align}
\f{I_{m}(b,q)}{T(b,q)} &\le \sqrt{2\pi m}\left(\f{m}{e}\right)^{m}e^{\f{1}{12m}}b^{\f{m-1}{2}}m^{-q(b-1)}S(b, q, m) \\
\label{eq:imconv3}
&= \sqrt{2\pi}\,b^{\f{m-1}{2}}m^{m-q(b-1)+\f{1}{2}}e^{-m+\f{1}{12m}}S(b, q, m).
\end{align}

In the following, we prove \eqref{eq:itnt} by dividing the argument into two cases: $b = 2$, where the dessins are
called \emph{clean} and which turns out to be the critical case, and $b \ge 3$.

%%%%%%%%%%%%%%%%%%%%%%%%%%%%%%%%%%%%
\subsection{The Case $b = 2$}

\begin{prop}
\label{prop:beq2}
Every passport $[n, 2^{q}, n]\ (n = 2q)$ of genus~$\ge 2$ admits a dessin with a trivial automorphism group.
\end{prop}

\begin{proof}
By \eqref{eq:genus}, the genus is

\begin{align}
 \f{2q - (1 + q + 1)}{2} + 1 = \f{q}{2} \ge 2. 
\end{align}
Therefore, $q$ must be an even integer with $q \ge 4$.

By \thmref{thm:MN},

\begin{align}
\f{N}{T} = \f{2}{n+2} = \f{2}{2q+2} = \f{1}{q\left(1+\f{1}{q}\right)} \ge \f{1}{q\left(1+\f{1}{4}\right)}
= \f{4}{5q}.
\end{align}
Therefore, it suffices to show that

\begin{align}
\f{I(b, q)}{T(b, q)} < \f{4}{5q}.
\end{align}

Let $\tau(n)$ denote the number of positive divisors of $n$.
Then the number of values of $m$ appearing in \eqref{eq:ibqtbq} is $\tau(n) - 2$.
Since $\tau(n) \le 2\sqrt{n} = 2\sqrt{2q}$, it suffices to show that for each
$m$ with $m \mid n$ and $2 \le m < n$,

\begin{align}
\f{I_{m}}{T} < \f{4}{5q}\cdot\f{1}{2\sqrt{2q}} = \f{\sqrt{2}}{5q^{\f{3}{2}}}.
\end{align}
Since $n = 2q$, the largest divisor $m$ of $n$ with $m < n$ is $m = q$.
Therefore, it suffices to show that

\begin{align}
\label{eq:logneg}
\log \left(\f{5}{\sqrt{2}}q^{\f{3}{2}}\f{I_{m}}{T}\right) < 0
\end{align}
holds for all $2 \le m \le q$.

Since $d_{i} \mid 2$ in \eqref{eq:diti}, the only possible values of $d_{i}$ are $1$ and $2$.

We now divide the analysis into the cases $m = 2$ and $3 \le m \le q$, and obtain an upper bound for
the left-hand side of \eqref{eq:logneg} in each case.

%%%%%%%%%%%%%%%%%%%%%%%%%%%%%%%%%%%%%%%%
\HL

\noindent
(i) The case $m=2$.

Since $q$ is even, the partitions $\{(d_{i}, t_{i})\}$ of $2$ satisfying \eqref{eq:diti} are $\{(1,2)\}$ and  $\{(2,1)\}$.
Hence, by \eqref{eq:imbqtbq},

\begin{align}
\f{I_{2}}{T} &= \f{q!}{(2q)!}2!(q!)^{2}\left(\f{1}{2}\left(\f{1}{\left(\f{q}{2}\right)!}\right)^{2} + \f{1}{2}\cdot\f{2^{q}}{q!}\right)
= \f{(q!)^{2}}{(2q)!}\left(\f{q!}{\left(\left(\f{q}{2}\right)!\right)^{2}} + 2^{q}\right).
\end{align}
Using Stirling's formula \eqref{eq:stirling}, we obtain

\begin{align}
\f{I_{2}}{T} &\le \f{2\pi q \left(\f{q}{e}\right)^{2q}e^{\f{1}{6q}}}{\sqrt{2\pi\cdot 2q}\left(\f{2q}{e}\right)^{2q}e^{\f{1}{24q+1}}}
\left(\f{\sqrt{2\pi q}\left(\f{q}{e}\right)^{q}e^{\f{1}{12q}}}{2\pi\cdot \f{q}{2}\left(\f{q}{2e}\right)^{q}e^{\f{2}{6q+1}}} + 2^{q}\right) \\
&= \f{\sqrt{\pi q}}{2^{q}}\exp\left(\f{1}{6q} - \f{1}{24q+1}\right)\left(\sqrt{\f{2}{\pi q}}\exp\left(\f{1}{12q}-\f{2}{6q+1}\right) + 1\right).
\end{align}

Here, the first exponential term is decreasing in $q$, hence we bound it by its value at $q = 4$.
The argument of the second exponential term is always negative, and thus we bound it by $e^{0} = 1$.
Hence,

\begin{align}
\label{eq:bm2}
\f{I_{2}}{T} &\le  \f{e^{\f{73}{2328}}}{2^{q}}\left(\sqrt{2} + \sqrt{\pi q}\right).
\end{align}
Therefore,

\begin{align}
\label{eq:Gb2}
\log \left(\f{5}{\sqrt{2}}q^{\f{3}{2}}\f{I_{2}}{T}\right) &\le \f{73}{2328} - q\log 2 + \log(\sqrt{2}+\sqrt{\pi q}) + \log 5 - \f{1}{2}\log 2 + \f{3}{2}\log q.
\end{align}
Differentiating the right-hand side with respect to $q$ gives

\begin{align}
- \log 2 + \f{1}{2}\sqrt{\f{\pi}{q}}\f{1}{\sqrt{2}+\sqrt{\pi q}} + \f{3}{2q} \le - \log 2 + \f{1}{2}\sqrt{\f{\pi}{4}}\f{1}{\sqrt{2}+\sqrt{4\pi}} + \f{3}{2\cdot 4} < 0 \q (q \ge 4).
\end{align}

Therefore, the right-hand side of \eqref{eq:Gb2} is decreasing for all $q \ge 4$, and a direct evaluation shows that
it becomes negative for $q \ge 10$.

%%%%%%%%%%%%%%%%%%%%%%%%%%%%%%%%%%%%%%%%
\HL

\noindent
(ii) The case $3 \le m \le q$.

Substituting $b = 2$ into \eqref{eq:imconv3}, we obtain

\begin{align}
\label{eq:qnm}
\f{I_{m}}{T} &\le 2^{\f{m}{2}}\sqrt{\pi}m^{m-q+\f{1}{2}}e^{-m+\f{1}{12m}}S(2, q, m).
\end{align}

Regarding \eqref{eq:diti}, $d_i$ must be $1$ or $2$ since $d_i \mid 2$.
For $d_i = 1$, the condition $m \mid d_i q$ holds only if $m \mid q$.
For $d_i = 2$, the condition $m \mid d_i q$ always holds, since $m \mid n\ (= 2q)$.
Furthermore, if $m \nmid q$, then $m$ must be even.
In this case, the unique partition satisfying \eqref{eq:diti} is $\{(2, m/2)\}$.

Therefore, the partitions of $m$ satisfying \eqref{eq:diti} are

\begin{align}
&\left\{ \{(1, m-2j), (2, j)\} \relmiddle{|} 0 \le j \le \left\lfloor\f{m}{2} \right\rfloor \right\} \q (m \mid q), \\
&\left\{ \left\{ \left(2,  \f{m}{2}\right) \right\}\right\} \q (m \nmid q).
\end{align}
As the second set is contained in the first, in \eqref{eq:qnm} we obtain the inequality

\begin{align}
\label{eq:s2qm}
S(2, q, m) = \sum_{\{(d_{i}, t_{i})\}}\prod_{i}\f{d_{i}^{\left(\f{d_{i}q}{m}-1\right)t_{i}}}{t_{i}!}
\le  \sum_{j=0}^{\left\lfloor\f{m}{2}\right\rfloor}\f{2^{\left(\f{2q}{m}-1\right)j}}{(m-2j)!j!}.
\end{align}

Let $\A = 2^{2q/m-1}$. Then

\begin{align}
e^{x+\A x^{2}} = e^{x}e^{\A x^{2}} = \sum_{k=0}^{\infty}\f{x^{k}}{k!}\sum_{j=0}^{\infty}\f{\A^{j}x^{2j}}{j!}
= \sum_{k=0}^{\infty}\sum_{j=0}^{\infty}\f{\A^{j}x^{k+2j}}{k!j!}.
\end{align}
Extracting the coefficient of $x^{m}$ on the right-hand side,
only the terms satisfying $k + 2j = m$ (\ie $k = m - 2j$) contribute. Hence we obtain

\begin{align}
[x^{m}]\, e^{x + \A x^{2}} = \sum_{j=0}^{\left\lfloor\f{m}{2}\right\rfloor}\f{\A^{j}}{(m-2j)!j!},
\end{align}
where the left-hand side denotes  the coefficient of $x^{m}$ in $e^{x + \A x^{2}}$.

Since the right-hand side of \eqref{eq:s2qm} coincides with this expression,
we may bound $S(2, q, m)$ using the generating function $e^{x + \A x^{2}}$ as

\begin{align}
\label{eq}
S(2, q, m) \le [x^{m}]\, e^{x + \A x^{2}}.
\end{align}
Let $F(x) = e^{x+\A x^{2}}$. By Cauchy's integral formula, we have

\begin{align}
S(2, q, m) &\le [x^{m}]\,F(x) = \f{1}{2\pi i}\int_{\lvert z \rvert=R}\f{F(z)}{z^{m+1}}\,dz \q (R>0).
\end{align}
Hence, for any $R > 0$,

\begin{align}
S(2, q, m) = \lvert S(2, q, m) \rvert &\le \f{1}{2\pi}\cdot 2\pi R\cdot \f{\max_{\lvert z \rvert=R}\lvert F(z) \rvert}{R^{m+1}} = \f{\max_{\lvert z \rvert=R}\lvert F(z) \rvert}{R^{m}}.
\end{align}

For fixed $\lvert z \rvert = R$, the quantity $\lvert F(z) \rvert$ attains its maximum at $z = R$ (a positive real number). Thus,

\begin{align}
\label{eq:smxm}
S(2, q, m) \le \f{e^{R+\A R^{2}}}{R^{m}} \q \text{for all }R > 0.
\end{align}
Viewing the right-hand side as a function of $R$, let $L(R)$ denote its logarithm. Then

\begin{align}
L(R) &= R + \A R^{2} - m \log R, \\
L'(R) &= 1 + 2\A R - \f{m}{R} = \f{1}{R}(2\A R^{2} + R - m).
\end{align}
Since the right-hand side of \eqref{eq:smxm} (the upper bound for $S(2, q, m)$) is minimized when
$L'(R) = 0$, \ie at $R = (\sqrt{\,1 + 8\A m} - 1)/(4\A)$, we choose a nearby value

\begin{align}
R = R_{0} \ceq \f{\sqrt{8\A m}}{4\A} = \sqrt{\f{m}{2\A}} \ (> 0).
\end{align}
Then

\begin{align}
S(2, q, m) &\le \f{e^{R_{0}+\A R_{0}^{2}}}{R_{0}^{m}} = \exp\left(\f{m}{2} + \sqrt{\f{m}{2\A}}\right)\left(\f{2\A}{m}\right)^{\f{m}{2}} \\
\label{eq:smupper}
&= \exp\left(\f{m}{2} + \sqrt{\f{m}{2^{\f{2q}{m}}}}\right)\left(\f{2^{\f{2q}{m}}}{m}\right)^{\f{m}{2}}
= 2^{q}m^{-\f{m}{2}}\exp\left(\f{m}{2} + 2^{-\f{q}{m}}\sqrt{m}\right).
\end{align}

From \eqref{eq:imconv2}, \eqref{eq:qnm}, and \eqref{eq:smupper}, we have

\begin{align}
\f{5}{\sqrt{2}}q^{\f{3}{2}}\f{I_{m}}{T} \le{}&\f{5}{\sqrt{2}}q^{\f{3}{2}}2^{\f{m}{2}+q}\sqrt{\pi}m^{\f{m+1}{2}-q}\exp\left(-\f{m}{2} + \f{1}{12m} + 2^{-\f{q}{m}}\sqrt{m}\right), \\
\log \left(\f{5}{\sqrt{2}}q^{\f{3}{2}}\f{I_{m}}{T}\right) \le{} &\f{3}{2}\log q + \log 5 + \left(\f{m}{2}+q-\f{1}{2}\right)\log 2 + \f{1}{2}\log \pi \\
&+ \left(\f{m+1}{2}-q\right)\log m - \f{m}{2} + \f{1}{12m} + 2^{-\f{q}{m}}\sqrt{m}.
\end{align}

Let the right-hand side be denoted by $G(q,m)$. For $q \ge 4$ and $3 \le m \le q$, we have

\begin{align}
\label{eq:b2ddG}
\pdv[2]{G}{m} &= \f{1}{2m} + \f{2q-1}{2m^{2}} + \f{1}{6m^{3}}
+ \f{2^{-\f{q}{m}}}{m^{\f{7}{2}}}\left(-\f{m^{2}}{4}-qm\log 2+(q\log 2)^{2}\right).
\end{align}

Let

\begin{align}
g(q, m) = -\f{m^{2}}{4}-qm\log 2+(q\log 2)^{2}.
\end{align}
Since $q \ge m$, we have

\begin{align}
\pdv{g}{q} = 2(\log 2)^{2}q - m\log 2 \ge 2(\log 2)^{2}m - m\log 2 = m\log 2(2\log 2-1) > 0.
\end{align}
Hence $g$ is increasing in $q$ for each fixed $m$, and therefore

\begin{align}
g(q, m) &\ge g(m, m) = -\f{m^{2}}{4}-m^{2}\log 2+(m\log 2)^{2} = - Cm^{2}, \\
C &\ceq  \log 2 + \f{1}{4} - (\log 2)^{2} > 0.
\end{align}
Therefore, in \eqref{eq:b2ddG} we have

\begin{align}
\f{2^{-\f{q}{m}}}{m^{\f{7}{2}}}\left(-\f{m^{2}}{4}-qm\log 2+(q\log 2)^{2}\right)
\ge - \f{2^{-\f{q}{m}}}{m^{\f{7}{2}}}Cm^{2}
\ge - \f{2^{-\f{m}{m}}}{m^{\f{7}{2}}}Cm^{2} = - \f{C}{2m^{\f{3}{2}}},
\end{align}
and it follows that

\begin{align}
\pdv[2]{G}{m} &\ge \f{1}{2m} + \f{2q-1}{2m^{2}} + \f{1}{6m^{3}} - \f{C}{2m^{\f{3}{2}}}
\ge \f{1}{2m} + \f{2m-1}{2m^{2}} + \f{1}{6m^{3}} - \f{C}{2m^{\f{3}{2}}} \\
&= \f{3}{2m} - \f{1}{2m^{2}} + \f{1}{6m^{3}} - \f{C}{2m^{\f{3}{2}}} >  \f{3}{2m} - \f{1}{2m^{2}} - \f{C}{2m^{\f{3}{2}}}
= \f{1}{2m^{\f{3}{2}}}\left(3\sqrt{m} - \f{1}{\sqrt{m}} - C\right) \\
&\ge \f{1}{2m^{\f{3}{2}}}\left(3\sqrt{3} - \f{1}{\sqrt{3}} - \log 2 - \f{1}{4} + (\log 2)^{2}\right) > 0.
\end{align}

Therefore, $G$ is convex in $m$. Hence its maximum on the interval $3 \le m \le q$ is attained at one of
the endpoints, $m = 3$ or $m = q$.

Let $H_{1}(q) \ceq G(q, 3)$ and $H_{2}(q) \ceq G(q, q)$. Then

\begin{align}
G(q, m) &\le \max\{H_{1}(q), H_{2}(q)\} \q (3 \le m \le q),
\end{align}
where

\begin{align}
H_{1}(q) &= \f{3}{2}\log q + \log 5 + (q+1)\log 2 + \f{1}{2}\log \pi + (2-q)\log 3 - \f{53}{36} + 2^{-\f{q}{3}}\sqrt{3}, \\
H_{1}'(q) &= \f{3}{2q} + \log 2 - \log 3- \f{2^{-\f{q}{3}}}{\sqrt{3}}\log 2 < \f{3}{2\cdot 4} + \log 2 - \log 3 < 0 \q (q \ge 4),
\end{align}
and

\begin{align}
H_{2}(q) &= \log 5 + \left(\f{3}{2}q-\f{1}{2}\right)\log 2 + \f{1}{2}\log \pi + \left(-\f{q}{2}+2\right)\log q - \f{q}{2} + \f{1}{12q} + \f{1}{2}\sqrt{q}, \\
H_{2}'(q) &= -\f{1}{2}\log q - 1 + \f{3}{2}\log 2 + \f{2}{q} - \f{1}{12q^{2}} + \f{1}{4\sqrt{q}} \\
&< -\f{1}{2}\log 4 - 1 + \f{3}{2}\log 2 + \f{2}{4} - 0 + \f{1}{4\sqrt{4}} < 0 \q (q \ge 4).
\end{align}

Thus both $H_{1}$ and $H_{2}$ are strictly decreasing for $q \ge 4$.
Direct computation shows that $H_{1}(q) < 0$ for $q \ge 22$, and $H_{2}(q) < 0$ for $q \ge 14$.
Consequently, $G(q, m)$ is negative for all $q \ge 22$ and all $m$ with $3 \le m \le q$.

Combining this with the case $m = 2$, we see that \eqref{eq:logneg} holds for all $q \ge 22$ and all $m$
with $m \mid n$, $2 \le m < n$.

For even integers $4 \le q \le 20$, \tabref{tab:exceptionbeq2} lists 
those permutations $y$ such that, with $x = \sigman$ fixed,
the automorphism group is trivial.

When $q = 4$, the monodromy group $G = \langle x, y \rangle$ has order~336 and
$C_{S_n}(G) \cong \{ 1 \}$.
On the other hand, when $6 \le q \le 20$, we have $\langle x, y \rangle = S_n$.
Since $C_{S_{n}}(S_{n}) = \{ 1 \}$,
it follows that for each such $q$ there exists a dessin $\msD$ with $\Aut \msD \cong \{ 1 \}$.

Consequently, every passport $[n, 2^{q}, n] \ (n = 2q)$ of genus~$\ge 2$
admits a dessin with a trivial automorphism group.
\end{proof}

%%%%%%%%%%%%
%%% table-beq2.tex: begin
\begin{table}[htbp]
  \centering
  \scriptsize
  \begin{tabular}{|c|c|c|c|c|p{11.1cm}|}
    \hline
    \multirow{2}{*}{$b$} &
    \multirow{2}{*}{$q$} &
    \multicolumn{4}{|l|}{$y$} \\ \cline{3-6}
    & & $G$ & $w(x,y)$ & $p_{w}$ & $w$ \\ \hline
    \multirow{2}{*}{$2$} &
    \multirow{2}{*}{$4$} &
    \multicolumn{4}{|p{7.5cm}|}{$(1\ 4)(2\ 5)(3\ 7)(6\ 8)$} \\ \cline{3-6}
     & & $336$ & & & \\ \hline
    \multirow{2}{*}{$2$} &
    \multirow{2}{*}{$6$} &
    \multicolumn{4}{|p{7.5cm}|}{$(1\ 4)(2\ 9)(3\ 6)(5\ 8)(7\ 11)(10\ 12)$} \\ \cline{3-6}
     & & $S_{12}$ & $xyxyx^{4}yx^{3}yx$ & $5$ & $(4\ 6\ 9\ 5\ 11)$ \\ \hline
    \multirow{2}{*}{$2$} &
    \multirow{2}{*}{$8$} &
    \multicolumn{4}{|p{7.5cm}|}{$(1\ 13)(2\ 10)(3\ 16)(4\ 6)(5\ 9)(7\ 14)(8\ 11)(12\ 15)$} \\ \cline{3-6}
     & & $S_{16}$ & $x^{3}yxyxyx^{3}$ & $13$ & $(1\ 15\ 8\ 7\ 16\ 10\ 11\ 2\ 3\ 5\ 6\ 12\ 13)$ \\ \hline
    \multirow{2}{*}{$2$} &
    \multirow{2}{*}{$10$} &
    \multicolumn{4}{|p{7.5cm}|}{$(1\ 15)(2\ 19)(3\ 12)(4\ 18)(5\ 8)(6\ 16)(7\ 10)(9\ 11)(13\ 20)(14\ 17)$} \\ \cline{3-6}
     & & $S_{20}$ & $yx^{7}yxyx^{2}$ & $17$ & $(1\ 10\ 8\ 3\ 4\ 15\ 5\ 6\ 12\ 9\ 17\ 2\ 11\ 19\ 20\ 13\ 16)$ \\ \hline
    \multirow{2}{*}{$2$} &
    \multirow{2}{*}{$12$} &
    \multicolumn{4}{|p{7.5cm}|}{$(1\ 16)(2\ 11)(3\ 9)(4\ 14)(5\ 20)(6\ 21)(7\ 15)(8\ 17)(10\ 24)(12\ 22)(13\ 18)(19\ 23)$} \\ \cline{3-6}
     & & $S_{24}$ & $yx^{2}yx^{4}yx^{2}yx^{4}$ & $19$ & $(1\ 9\ 20\ 18\ 23\ 19\ 11\ 8\ 21\ 24\ 12\ 5\ 22\ 4\ 2\ 6\ 3\ 10\ 13)$ \\ \hline
    \multirow{2}{*}{$2$} &
    \multirow{2}{*}{$14$} &
    \multicolumn{4}{|p{7.5cm}|}{$(1\ 27)(2\ 18)(3\ 20)(4\ 8)(5\ 22)(6\ 13)(7\ 26)(9\ 12)(10\ 15)(11\ 28)(14\ 17)(16\ 23)(19\ 25)(21\ 24)$} \\ \cline{3-6}
     & & $S_{28}$ & $x^{3}yx^{11}y$ & $23$ & $(1\ 18\ 9\ 19\ 7\ 15\ 27\ 12\ 6\ 24\ 11\ 3\ 23\ 4\ 28\ 8\ 13\ 17\ 22\ 26\ 5\ 25\ 21)$ \\ \hline
    \multirow{2}{*}{$2$} &
    \multirow{2}{*}{$16$} &
    \multicolumn{4}{|p{7.5cm}|}{$(1\ 26)(2\ 27)(3\ 29)(4\ 15)(5\ 20)(6\ 12)(7\ 14)(8\ 24)(9\ 19)(10\ 30)(11\ 17)(13\ 16)(18\ 23)(21\ 25)(22\ 32)(28\ 31)$} \\ \cline{3-6}
     & & $S_{32}$ & $x^{3}yx^{3}$ & $29$ & $(1\ 18\ 28\ 31\ 30\ 29\ 25\ 2\ 23\ 4\ 17\ 8\ 20\ 21\ 11\ 10\ 19\ 3\ 15\ 26\ 6\ 22\ 24\ 5\ 27\ 13\ 16\ 12\ 7)$ \\ \hline
    \multirow{2}{*}{$2$} &
    \multirow{2}{*}{$18$} &
    \multicolumn{4}{|p{7.5cm}|}{$(1\ 28)(2\ 9)(3\ 29)(4\ 7)(5\ 36)(6\ 24)(8\ 16)(10\ 12)(11\ 33)(13\ 21)(14\ 17)(15\ 35)(18\ 31)(19\ 25)(20\ 26)(22\ 32)(23\ 34)\allowbreak(27\ 30)$} \\ \cline{3-6}
     & & $S_{36}$ & $xyx^{2}yx^{2}yx^{2}yx$ & $29$ & $(1\ 15\ 26\ 17\ 35\ 21\ 14\ 4\ 16\ 18\ 7\ 22\ 25\ 20\ 28\ 34\ 13\ 23\ 12\ 5\ 24\ 29\ 10\ 9\ 31\ 32\ 6\ 33\ 36)$ \\ \hline
    \multirow{2}{*}{$2$} &
    \multirow{2}{*}{$20$} &
    \multicolumn{4}{|p{7.5cm}|}{$(1\ 39)(2\ 40)(3\ 10)(4\ 32)(5\ 35)(6\ 17)(7\ 14)(8\ 22)(9\ 19)(11\ 31)(12\ 38)(13\ 36)(15\ 26)(16\ 18)(20\ 28)(21\ 30)(23\ 33)\allowbreak(24\ 29)(25\ 27)(34\ 37)$} \\ \cline{3-6}
     & & $S_{40}$ & $x^{4}yxyxyxyx^{4}yx$ & $37$ & $(1\ 25\ 27\ 23\ 12\ 8\ 40\ 39\ 18\ 21\ 3\ 37\ 36\ 9\ 19\ 11\ 14\ 32\ 13\ 17\ 24\ 4\ 6\ 30\ 22\ 20\ 38\ 26\ 34\ 31\allowbreak16\ 33\ 35\ 29\ 15\ 28\ 2)$ \\ \hline
  \end{tabular} \\
  \vspace{1\baselineskip}
  \caption{Elements $y$ such that $\AD \cong \{ 1 \}$ for $b = 2$ ($x = (1\ 2\ \ldots\ n)$)}\label{tab:exceptionbeq2}
\end{table}
%%% table-beq2.tex: end
%%%%%%%%%%%%

To find elements $y$ such that $\langle x, y \rangle = S_{n}$ for $x = \sigman$, we make use of the following theorem.

\begin{thm}
\label{thm:pn-3}
Let $G \le S_{n}$ be a primitive group which contains a cycle $(p)$ of prime length $p$.
Then either $G \ge A_{n}$ or $n \le p+2$.
\end{thm}

\begin{proof}
See \cite[Theorem~3.3E]{Dixon96}.
\end{proof}

By this theorem, for a permutation $y$ of cycle type $(2^{q})$, if $G = \langle x, y \rangle$ has no blocks with respect to any residue class of a divisor $m$ of $n = 2q$ ($2 \le m < n$), and if $G$ contains an element $w$ of cycle type $(p_{w})$ for a prime $p_{w} \le n - 3$, then $G = S_{n}$ (since $n$ is even, it cannot be $A_{n}$).

A list of such examples of $y$ and $w$ obtained by computation is shown in \tabref{tab:exceptionbeq2}.

In practice, we generated random permutations $y$ of cycle type $(2^{q})$, produced several tens of thousands of elements of $\langle x, y \rangle$, and stopped as soon as an element $w$ satisfying the required conditions was found; otherwise, we discarded $y$ and tried again.
By repeating this procedure, we completed the table.

When $(b, q) = (2, 4)$, the monodromy group never becomes $S_{8}$.
Instead, it can be a group of order~$336$, in which case the automorphism group is trivial.
A dessin realizing this case is shown in \figref{fig:g2-336}.

\begin{figure}[htbp]
\centering
\includegraphics[width=80mm]{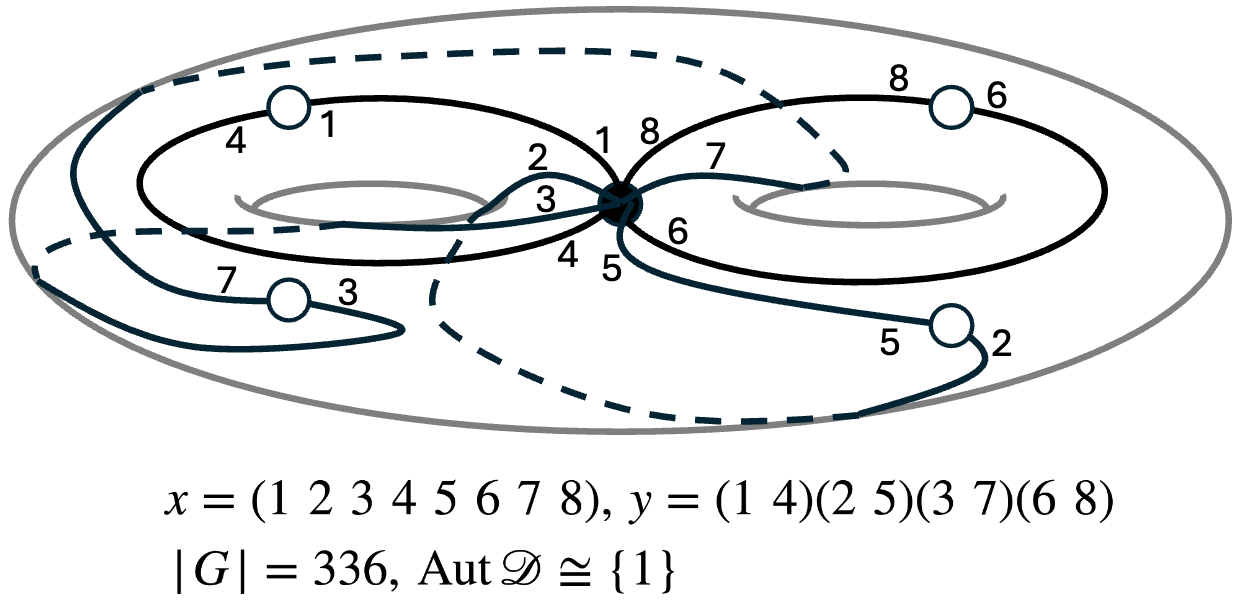}
\caption{A dessin with trivial automorphism group having passport $[8,2^{4},8]$}
\label{fig:g2-336}
\end{figure}

%%%%%%%%%%%%%%%%%%%%%%%%%%%%%%%%%%%%
\subsection{The Case $b \ge 3$}

\begin{prop}
\label{prop:bge3}
Every passport $[n, b^{q}, n]$ of genus~$\ge 2$ with $b \ge 3$ and $q \ge 2$ admits a dessin with a trivial automorphism group.
\end{prop}

\begin{proof}
By \eqref{eq:genus}, the genus is

\begin{align}
 \f{n - (1 + q + 1)}{2} + 1 = \f{n-q}{2} \ge 2. 
\end{align}
Therefore, we have $n - q \ge 4$ and $n \equiv q \pmod{2}$, and hence $n \ge 6$.

Our goal is to show that 

\begin{align}
\f{I(b,q)}{T(b,q)} < \f{N(b, q)}{T(b,q)}.
\end{align}

By \thmref{thm:MN},

\begin{align}
\f{N(b,q)}{T(b,q)} \ge \f{2}{n+2} = \f{2}{n\left(1+\f{2}{n}\right)} \ge \f{2}{n\left(1+\f{2}{6}\right)} = \f{3}{2n}.
\end{align}
Thus it suffices to prove that

\begin{align}
\f{I(b, q)}{T(b, q)} < \f{3}{2n}.
\end{align}
Since there are $\tau(n) - 2$ integers $m$ such that $m \mid n$ and $2 \le m < n$ (where $\tau(n)$
denotes the divisor-counting function), it will be enough to show that for each such $m$,

\begin{align}
\f{I_{m}}{T} < \f{3}{2n\tau(n)}.
\end{align}
Since $\tau(n) \le 2\sqrt{n}$, we have

\begin{align}
\f{3}{2n\tau(n)} < \f{3}{4n^{\f{3}{2}}} = \f{3}{4(bq)^{\f{3}{2}}}.
\end{align}
Thus it suffices to show

\begin{align}
\label{eq:dest3}
F(b, q, m) \ceq \f{I_{m}}{T}\cdot \f{4}{3}(bq)^{\f{3}{2}} < 1
\end{align}
for all $m$ such that $m \mid n$ and $2 \le m < n$.

By \eqref{eq:imconv2},

\begin{align}
\label{eq:fbqm1}
F(b, q, m) &\le \f{4}{3}m!\,b^{\f{m}{2}+1}q^{\f{3}{2}}m^{-q(b-1)}S(b, q, m),
\end{align}
where

\begin{align}
\label{eq:sbqm}
S(b, q, m) &= \sum_{\{(d_{i}, t_{i})\}}\prod_{i}\f{d_{i}^{\left(\f{d_{i}q}{m}-1\right)t_{i}}}{t_{i}!}.
\end{align}
Furthermore, from \eqref{eq:imconv3},

\begin{align}
F(b, q, m) \le \f{4\sqrt{2\pi}}{3}b^{\f{m}{2}+1}q^{\f{3}{2}}m^{m-q(b-1)+\f{1}{2}}e^{-m+\f{1}{12m}}S(b, q, m).
\end{align}
Since $m \ge 2$, we have $e^{1/(12m)} \le e^{1/24}$, and therefore

\begin{align}
\label{eq:fbqm2}
F(b, q, m) \le C_{0}e^{-m}b^{\f{m}{2}+1}q^{\f{3}{2}}m^{m-q(b-1)+\f{1}{2}}S(b, q, m),
\q C_{0} \ceq \f{4\sqrt{2\pi}}{3}e^{\f{1}{24}}.
\end{align}
Taking logarithms gives

\begin{align}
\log F(b, q, m) \le{} &\log C_{0} - m + \left(\f{m}{2}+1\right)\log b + \f{3}{2}\log q + \left(m-q(b-1)+\f{1}{2}\right)\log m \\
\label{eq:fbqm3}
&+ \log S(b, q, m).
\end{align}
What we want to show is that $\log F(b, q, m) < 0$.

For $S(b, q, m)$ in \eqref{eq:sbqm}, let us define the set of admissible $d_{i}$ for \eqref{eq:diti} by

\begin{align}
\label{eq:c12}
D = \{ d \in \Zp \mid d \le m,\ d \mid b,\ m \mid dq \}.
\end{align}
Here,

\begin{align}
\exp\left(\sum_{d\in D}d^{\f{dq}{m}-1}x^{d}\right) = \prod_{d \in D}\exp\left(d^{\f{dq}{m}-1}x^{d}\right)
= \prod_{d \in D}\sum_{t=0}^{\infty}\f{d^{\left(\f{dq}{m}-1\right)t}x^{dt}}{t!}.
\end{align}

When we extract the coefficient of $x^{m}$ on the right-hand side, the remaining terms are precisely the
products and sums over all $\{(d_{i}, t_{i})\}$ satisfying $\sum_{i} d_{i} t_{i} = m$.
Since each $d_{i} \in D$, these satisfy \eqref{eq:diti}.
Thus, the expression coincides with \eqref{eq:sbqm}.
Therefore, by an estimate similar to that used in \eqref{eq:smxm}, we obtain

\begin{align}
\label{eq:c13}
S(b, q, m) &= [x^{m}]\,\exp\left(\sum_{d \in D}d^{\f{qd}{m}-1}x^{d}\right),
\end{align}
and furthermore,

\begin{align}
\label{eq:c05}
S(b, q, m) &\le \f{\exp\left(\sum_{d \in D}d^{\f{qd}{m}-1}R^{d}\right)}{R^{m}} \q \text{for all }R > 0.
\end{align}

Define

\begin{align}
g \ceq \gcd(m, q), \q d_{0} \ceq \f{m}{g}, \q q' \ceq \f{q}{g}, \q b' \ceq \f{b}{d_{0}}.
\end{align}
Since $m \mid n\ (= bq)$, we have $d_{0} \mid b$, and hence $b' \in \Zp$.

As $m \mid d_{i}q \Lra d_{0} \mid d_{i}$, 
the set of all $d_{i}$ satisfying \eqref{eq:diti} coincides with the set of integers of the form

\begin{align}
d_{i} = d_{0}r, \ 1 \le r \le g, \ r \mid b'.
\end{align}
Thus, \eqref{eq:c12} can be written as

\begin{align}
D = \left\{ d_{0}r \mid 1 \le r \le g, \ r \mid b' \right\}.
\end{align}
Focusing on the values of $r$, define

\begin{align}
D_{r} = \left\{ r \mid 1 \le r \le g, \ r \mid b' \right\}.
\end{align}
Then \eqref{eq:c13} becomes

\begin{align}
S(b, q, m) &= [x^{m}]\,\exp\left(\sum_{r \in D_{r}}(d_{0}r)^{\f{qd_{0}r}{m}-1}x^{d_{0}r}\right)
= [x^{m}]\,\exp\left(\sum_{r\in D_{r}}(d_{0}r)^{\f{qmr}{mg}-1}x^{d_{0}r}\right) \\
\label{eq:c13Dr}
&= [x^{m}]\,\exp\left(\sum_{r\in D_{r}}(d_{0}r)^{q'r-1}x^{d_{0}r}\right).
\end{align}

Furthermore, for the parameters required in estimating $S(b, q, m)$, set

\begin{align}
\label{eq:ABdef}
\A \ceq d_{0}^{q'-1}, \q \B \ceq \sum_{r \in D_{r}, r\ge 2}(d_{0}r)^{q'r-1}.
\end{align}

In what follows, we proceed by considering the following cases according to the values of $b$, $q$, and $m$.
\HL

\noindent
(A) $\lvert D_{r} \rvert = 1$ (\ie $D_{r}=\{1\}$)
\begin{itemize}[leftmargin=3.5em]
\item[(A-1)] $d_0=b$
\item[(A-2)] $m<b$
\end{itemize}

\medskip
\noindent
(B) $\lvert D_{r} \rvert \ge 2$
\begin{itemize}[leftmargin=3.5em]
\item[(B-1)] $\B < g/2$
\item[(B-2)] $\B \ge g/2$
  \begin{itemize}
  \item[(B-2-1)] $m \mid q$
  \item[(B-2-2)] $m \nmid q$
  \end{itemize}
\end{itemize}
\HL

In each case, we proceed as follows:

\begin{enumerate}
\item Find an upper bound $G(b, q, m)$ satisfying $\log F(b, q, m) \le G(b, q, m)$ and containing no factorials, summations, or products.
\item Show that $\partial^{2} G/\partial m^{2} > 0$.
Thus, for each pair $(b, q)$, the function $G$ is convex in $m$, and consequently its maximum occurs at
one of the endpoints of the admissible range of $m$.
\item Determine the range of $(b, q)$ for which $G$ is negative at all relevant boundary points of $m$
(namely, its minimum and maximum, and any junction points arising from case distinctions).
\end{enumerate}

In this way, we show that $G < 0$ holds except for finitely many exceptional cases.

For the finitely many pairs $(b, q)$ for which $G \ge 0$, we proceed as in \propref{prop:beq2} and explicitly
find permutations $x$ and $y$ such that $\Aut \msD \cong \{1\}$.
\HL

%%%%%%%%%%%%%%%%%%%%%%%%%%%%%%%%%%%%%%%%
\noindent
(A) The case $\lvert D_{r} \rvert = 1$ (\ie $D_{r}= \{ 1 \}$).

In this case, there is no $r$ satisfying both $2 \le r \le g$ and $r \mid b'$.

Moreover, since $d_{0} \mid b$, we have $d_{0} \le b$.

If $d_{0} = b$, then $b' = 1$, and hence $D_{r}=\{1\}$.

If $d_{0} < b$, assume $m \ge b$. Then taking $r = b' = b/d_{0}\ (> 1)$ yields $2\le r \le m/d_{0} = g$,
contradicting the assumption.

Therefore, when $D_{r} = \{ 1 \}$, we must have

\begin{align}
d_{0} = b \q \text{or} \q m < b.
\end{align}

Since $\lvert D_{r} \rvert = 1$, the only partition satisfying \eqref{eq:diti} is
$\{(d_{i}, t_{i})\} = \{(d_{0}, g)\}$. Hence,

\begin{align}
S(b, q, m) = \f{d_{0}^{\left(\f{d_{0}q}{m}-1\right)g}}{g!} = \f{d_{0}^{\left(\f{mq}{gm}-1\right)g}}{g!} = \f{d_{0}^{q-g}}{g!}.
\end{align}
Substituting this into \eqref{eq:fbqm1} and applying \eqref{eq:stirling}, we obtain

\begin{align}
F(b, q, m) &\le \f{4}{3}m!\,b^{\f{m}{2}+1}q^{\f{3}{2}}m^{-q(b-1)}\f{d_{0}^{q-g}}{g!} \\
&\le \f{\sqrt{2\pi m} \left(\f{m}{e}\right)^{m}e^{\f{1}{12m}}}{\sqrt{2\pi g} \left(\f{g}{e}\right)^{g}e^{\f{1}{12g+1}}}\cdot
\f{4}{3}b^{\f{m}{2}+1}q^{\f{3}{2}}m^{-q(b-1)}\left(\f{m}{g}\right)^{q-g} \\
&= \f{4}{3}b^{\f{m}{2}+1}q^{\f{3}{2}}m^{-q(b-1)}\left(\f{m}{e}\right)^{m-g}\left(\f{m}{g}\right)^{q+\f{1}{2}}
\exp\left(\f{1}{12m} - \f{1}{12g+1}\right).
\end{align}
Since $m \ge 2$ and $m \ge g$, we have

\begin{align}
\f{1}{12m} - \f{1}{12g+1} \le \f{1}{12m} - \f{1}{12m+1} = \f{1}{12m(12m+1)}\le \f{1}{12\cdot 2(12\cdot 2+1)} = \f{1}{600}.
\end{align}
Therefore,

\begin{align}
F(b, q, m) &\le C_{1}b^{\f{m}{2}+1}q^{\f{3}{2}}m^{-q(b-1)}\left(\f{m}{e}\right)^{m-g}\left(\f{m}{g}\right)^{q+\f{1}{2}}, \q C_{1} \ceq \f{4}{3}e^{\f{1}{600}}.
\end{align}
Let $G(b, q, m)$ denote the logarithm of the right-hand side. Then

\begin{align}
G(b, q, m) ={} &\log C_{1} + \left(\f{m}{2}+1\right)\log b + (m-g)(\log m - 1) + \f{1}{2}(\log m - \log g) \\
\label{eq:c17}
&+ \f{3}{2}\log q - q((b-2)\log m + \log g).
\end{align}
\HL

%%%%%%%%%%%%%%%%%%%%%%%%%%%%%%%%%%%%%%%%
\noindent
(A-1) The subcase $d_{0} = b$.

\noindent
\eg $(b, q, m) = (3, 2, 3), (5, 2, 5), (3, 4, 3), (3, 4, 6), (7, 2, 7)$

%$m \ge 3, d_{0} \ge 3, q' \ge 2$
\HL

In this case, since $d_{0} = m/g = b$, we have $g=m/b$ and $\log g = \log m - \log b$. Hence \eqref{eq:c17} becomes

\begin{align}
G(b, q, m) ={} &\log C_{1} + \left(\f{m}{2}+1\right)\log b + m\left(1-\f{1}{b}\right)(\log m - 1) + \f{1}{2}\log b + \f{3}{2}\log q \\
& - q((b-1)\log m - \log b),
\end{align}
and

\begin{align}
\f{\partial^{2} G}{\partial m^{2}} = \f{b-1}{m}\left(\f{q}{m}+\f{1}{b}\right) \ge \f{3-1}{m}\left(\f{q}{m}+\f{1}{b}\right) > 0.
\end{align}

Thus, for fixed $b$ and $q$, the function $G$ is convex in $m$.
Since $g = m/b \ge 1$, we have $m \ge b$; moreover, from $m \mid n$ and $m < n$, it follows that
$m \le n/2 = bq/2$. Hence $b \le m \le bq/2$.

Setting $H_{1}(b, q) \ceq G(b, q, b)$ and $H_{2}(b, q) \ceq G(b, q, bq/2)$, we obtain

\begin{align}
G(b, q, m) &\le \max\{H_{1}(b, q), H_{2}(b, q)\},
\end{align}
where

\begin{align}
H_{1}(b, q) &= \log C_{1} - b + 1 + \f{3}{2}\log q + \left(-(b-2)q+\f{3}{2}b+\f{1}{2}\right)\log b, \\
\f{\partial H_{1}}{\partial q} &= \f{3}{2q} - (b-2)\log b \le \f{3}{2\cdot 2} - (3-2)\log 3 < 0,
\end{align}
and

\begin{align}
H_{2}(b, q) ={} &\log C_{1} + \f{q}{2}(b-1)(\log 2-1) + \f{1}{4}(-bq+6q+6)\log b \\
&+ \f{1}{2}(-q(b-1)+3)\log q, \\
\f{\partial H_{2}}{\partial q} ={} &\f{3}{2q} + \f{1}{2}(b-1)(\log 2 - \log q) -b + 1 + \left(-\f{1}{4}b+\f{3}{2}\right)\log b \\
\le{} &\f{3}{2\cdot 2} + \f{1}{2}(b-1)(\log 2 - \log 2) -b + 1 + \left(-\f{1}{4}b+\f{3}{2}\right)\log b \\
={} &\f{3}{4} -b + 1 + \left(-\f{1}{4}b+\f{3}{2}\right)\log b.
\end{align}

Let the right-hand side be $g(b)$. Then

\begin{align}
g'(b) = - \f{5}{4} - \f{1}{4}\log b + \f{3}{2b} \le - \f{5}{4} - \f{1}{4}\log 3 + \f{3}{2\cdot 3} < 0 \q (b \ge 3).
\end{align}
Hence $g(b)$ is decreasing in $b$. Therefore

\begin{align}
\f{\partial H_{2}}{\partial q} \le g(3) = \f{3}{4} - 3 + 1 + \left(-\f{1}{4}\cdot 3+\f{3}{2}\right)\log 3 < 0.
\end{align}
Hence both $H_{1}(b, q)$ and $H_{2}(b, q)$ are decreasing in $q$. Therefore, neither exceeds
their value at $q = 2$.

Since $H_{1}(b, 2) = H_{2}(b, 2) = G(b, 2, b)$, we obtain

\begin{align}
G(b, q, m) &\le G(b, 2, b) = \log C_{1} + \f{3}{2}\log 2  - b + 1 - \f{b-9}{2}\log b.
\end{align}
Differentiating with respect to $b$, we have

\begin{align}
\dv{b}G(b, 2, b) &= - \f{3}{2} - \f{\log b}{2} + \f{9}{2b} \le - \f{3}{2} - \f{\log 3}{2} + \f{9}{2\cdot 3} < 0 \q (b \ge 3).
\end{align}
Thus $G(b, 2, b)$ is strictly decreasing for $b \ge 3$, and direct computation shows that it becomes negative
for $b \ge 6$.
Hence, in this range, we have $G(b, q, m) < 0$ for all $q$ and $m$.

For $3 \le b \le 5$, taking into account that $q$ must be even whenever $b$ is even, direct computation shows that

\begin{align}
&H_{1}(3, k)\ (2 \le k \le 5), H_{1}(4, 2), H_{1}(5, 2) > 0, \\
&H_{1}(3, 6), H_{1}(4, 4), H_{1}(5, 3) < 0, \\
&H_{2}(3, k)\ (2 \le k \le 4), H_{2}(4, 2), H_{2}(5, 2) > 0, \\
&H_{2}(3, 5), H_{2}(4, 4), H_{2}(5, 3) < 0.
\end{align}

Since both $H_{1}$ and $H_{2}$ are decreasing in $q$, $H_{1}$ or $H_{2}$ becomes non-negative only for

\begin{align}
\label{eq:exA1}
(b, q) = (3, 2), (3, 3), (3, 4), (3, 5), (4, 2), (5, 2).
\end{align}

Therefore, when $\lvert D_{r} \rvert = 1$ and $d_{0} = b$, inequality \eqref{eq:dest3} holds except for these pairs.
\HL

%%%%%%%%%%%%%%%%%%%%%%%%%%%%%%%%%%%%%%%%
\noindent
(A-2) The subcase $m < b$.

\noindent
\eg $(b, q, m) = (3, 2, 2), (5, 2, 2), (3, 4, 2), (6, 2, 3), (6, 2, 4), (7, 2, 2)$

%$b' \ge 2$
\HL

Differentiating \eqref{eq:c17} with respect to $g$ gives

\begin{align}
-(\log m - 1) - \f{1}{2g} - \f{q}{g} < - (\log 2 - 1) - 0 - 1 < 0,
\end{align}
therefore \eqref{eq:c17} is strictly decreasing in $g$.

Hence, letting $G_{2}(b, q, m)$ denote the expression obtained by substituting $g = 1$, we have

\begin{align}
G(b, q, m) \le{} &G_{2}(b, q, m) \\
={} &\log C_{1} + \left(\f{m}{2}+1\right)\log b + (m-1)(\log m - 1) + \f{3}{2}\log q \\
&- \left(q(b-2) - \f{1}{2}\right)\log m.
\end{align}
Moreover,

\begin{align}
\pdv[2]{G_{2}}{m} &= \f{1}{m} + \f{q(b-2) + \f{1}{2}}{m^{2}} \ge \f{1}{m} + \f{q(3-2) + \f{1}{2}}{m^{2}} > 0.
\end{align}

Thus $G_{2}$ is convex in $m$.
Since $2 \le m < b$, its maximum is attained at the endpoints of this interval.

Setting

\begin{align}
H_{1}(b, q) \ceq G_{2}(b, q, 2), \q H_{2}(b, q) \ceq G_{2}\left(b, q, b\right),
\end{align}
we obtain

\begin{align}
G(b, q, m) &\le G_{2}(b, q, m) \le \max\left\{H_{1}(b, q), H_{2}\left(b, q\right)\right\},
\end{align}
where

\begin{align}
H_{1}(b, q) &= \log C_{1} + \left(\f{3}{2} - q(b-2)\right)\log 2 - 1 + 2\log b + \f{3}{2}\log q , \\
\pdv{H_{1}}{b} &= -q\log 2+\f{2}{b} \le -2\log 2 + \f{2}{3} < 0, \\
\pdv{H_{1}}{q} &= - (b-2)\log 2 + \f{3}{2q}
\begin{dcases}
= - (3-2)\log 2 + \f{3}{2\cdot 2} > 0 & (b = 3, q = 2) \\
\le - (3-2)\log 2 + \f{3}{2\cdot 3} < 0 & (b = 3, q \ge 3) \\
\le - (4-2)\log 2 + \f{3}{2\cdot 2} < 0 & (b \ge 4)
\end{dcases}, \\
\end{align}
and

\begin{align}
H_{2}(b, q) &= \log C_{1} - b + 1 + \left(\f{3b+1}{2}-q(b-2)\right)\log b + \f{3}{2}\log q, \\
\pdv{H_{2}}{b} &= \f{1}{2} + \left(\f{3}{2}-q\right)\log b - q\left(1-\f{2}{b}\right) + \f{1}{2b} \\
&\le \f{1}{2} + \left(\f{3}{2}-2\right)\log 3 - 2\left(1-\f{2}{3}\right) + \f{1}{2\cdot 3} < 0, \\
\pdv{H_{2}}{q} &= -(b-2)\log b + \f{3}{2q} \le -(3-2)\log 3 + \f{3}{2\cdot 2} < 0.
\end{align}

Thus $H_{1}$ is strictly decreasing in $b$, and, except for the case $(b, q) = (3, 2)$, also strictly decreasing in $q$.
Similarly, $H_{2}$ is strictly decreasing in both $b$ and $q$.
Direct computation shows that

\begin{align}
&H_{1}(3, k)\ (2 \le k \le 8), H_{1}(4, 2), H_{1}(5, 2) > 0, \\
&H_{1}(3, 9), H_{1}(4, 4), H_{1}(5, 3) , H_{1}(6, 2) < 0, \\
&H_{2}(3, k)\ (2 \le k \le 5), H_{2}(4, 2), H_{2}(5, 2) > 0, \\
&H_{2}(3, 6), H_{2}(4, 4), H_{2}(5, 3) , H_{2}(6, 2) < 0.
\end{align}

Therefore, when $\lvert D_{r} \rvert = 1$ and $m < b$,
inequality \eqref{eq:dest3} holds except in the following cases:

\begin{align}
\label{eq:exA2}
(b, q) = (3, k) \ (2 \le k \le 8), (4, 2), (5, 2).
\end{align}

\HL

%%%%%%%%%%%%%%%%%%%%%%%%%%%%%%%%%%%%%%%%
\noindent
(B) The case $\lvert D_{r} \rvert \ge 2$.

In this case, since there exists an integer $r$ such that $2 \le r \le g$ and $r \mid b'$, we must have $g \ge 2$ and $b' \ge 2$.

Moreover, when $(b, q) = (3, 2)$, we have $n = 6$, and hence $m = 2$ or $3$. In this case,

\begin{align}
&m = 2\ \Ra\ g = 2,\ b' = 3\ \Ra\ D_{r} = \{ 1 \}, \\
&m = 3\ \Ra\ g = 1\ \Ra\ D_{r} = \{ 1 \}.
\end{align}
Thus, for $(b, q) = (3, 2)$, we always have $D_{r} = \{1\}$.
Therefore, this pair can be excluded from the present case.

From \eqref{eq:c13Dr} we have

\begin{align}
S(b, q, m) &= [x^{m}]\,\exp\left(\sum_{r\in D_{r}}(d_{0}r)^{q'r-1}x^{d_{0}r}\right) \\
&= [y^{g}]\,\exp\left(\sum_{r\in D_{r}}(d_{0}r)^{q'r-1}y^{r}\right) \q (y = x^{d_{0}}) \\
&\le \f{\exp\left(\sum_{r\in D_{r}}(d_{0}r)^{q'r-1}R^{r}\right)}{R^{g}} \q (\text{for all } R>0).
\end{align}
Therefore,

\begin{align}
\label{eq:d02}
\log S(b, q, m) &\le \sum_{r\in D_{r}}(d_{0}r)^{q'r-1}R^{r} - g \log R \q (\text{for all } R>0).
\end{align}
Using $\A$ and $\B$ defined in \eqref{eq:ABdef}, we obtain, for $0 < R \le 1$,

\begin{align}
\log S(b, q, m) &\le \A R + \B R^{2} - g\log R.
\end{align}
Differentiating the right-hand side with respect to $R$ yields $(2\B R^{2} + \A R - g)/R$,
and hence it attains its minimum at
$R = (-\A + \sqrt{\A^{2}+8\B g})/(4\B)$.
Taking a nearby point $R = \sqrt{8\B g}/(4\B) = \sqrt{g/(2\B)}$ and setting

\begin{align}
\label{eq:R0}
R_{0} \ceq \min\left\{\sqrt{\f{g}{2\B}},  \ 1\right\},
\end{align}
we have $0 < R_{0} \le 1$, and hence

\begin{align}
\label{eq:smr1}
\log S(b, q, m) &\le \A R_{0} + \B R_{0}^{2} - g\log R_{0}.
\end{align}

\HL

%%%%%%%%%%%%%%%%%%%%%%%%%%%%%%%%%%%%%%%%
\noindent
(B-1) The subcase $\B < g/2$.

\noindent
\eg $(b, q, m) = (4, 10, 20), (4, 12, 24), (4, 14, 28), (3, 19, 19), (3, 20, 20)$

%$b' \ge 2$
\HL

In this situation, we have $R_{0} = 1$ by \eqref{eq:R0}. Since the sum $\B$ contains at least one term with $r \ge 2$, we obtain

\begin{align}
 \B = \sum_{r \in D_{r}, r\ge 2}(d_{0}r)^{q'r-1} \ge 2^{2-1} = 2,
\end{align}
and since $\B < g/2$, we obtain $g \ge 5$, which implies $q, m \ge 5$.

Assume $g = 5$. Then:

\begin{itemize}
\item If $q$ is odd, then $b$ is also odd, and therefore every divisor $r \ge 2$ of $b'$ is at least~$3$. \\
Hence $\B \ge 3^{3-1} = 9$, but since $\B < g/2$, this would force $g \ge 19$, a contradiction.
\item If $q$ is even, then $q \ge 10$ since $5 \mid q$, and we have $q' = q/5 \ge 2$. \\
Therefore, $\B \ge 2^{2q'-1} \ge 2^{3} = 8 > 5/2 = g/2$, again a contradiction.
\end{itemize}
Thus we must have $g \ge 6$ (and in fact $g \ge 19$ when $q$ is odd).
Consequently, $q, m \ge 6$.

We have

\begin{align}
\f{\A}{\B} \le \f{d_{0}^{q'-1}}{(2d_{0})^{2q'-1}} = \f{1}{2^{2q'-1}d_{0}^{q'}} \le \f{1}{2},
\end{align}
and $R_{0} = 1$. Therefore, from \eqref{eq:smr1} we obtain

\begin{align}
\log S(b, q, m) \le \A + \B \le \f{3}{2}\B < \f{3}{4}g \le \f{3}{4}m.
\end{align}
Substituting this into the right-hand side of \eqref{eq:fbqm3}, we obtain the following upper bound,
which we denote by $G(b,q,m)$:

\begin{align}
G(b, q, m) &= \log C_{0} - \f{1}{4}m + \left(\f{m}{2}+1\right)\log b + \f{3}{2}\log q + \left(m-q(b-1)+\f{1}{2}\right)\log m,
\end{align}
and

\begin{align}
\pdv[2]{G}{m} &= \f{1}{m} + \f{q(b-1)-\f{1}{2}}{m^{2}}  \ge \f{1}{m} + \f{6(3-1)-\f{1}{2}}{m^{2}} > 0.
\end{align}
Thus $G$ is convex in $m$.
Since $6 \le m \le bq/2$, the maximum of $G$ over this range is attained at one of the two endpoints.

Let $H_{1}(b, q) \ceq G(b, q, 6),\ H_{2}(b, q) \ceq G(b, q, bq/2)$, so that

\begin{align}
G(b, q, m) &\le \max\{H_{1}(b, q), H_{2}(b, q)\},
\end{align}
where

\begin{align}
H_1(b, q) &= \left(- q \left(b - 1\right) + \f{13}{2}\right) \log 6 - \f{3}{2} + \log C_{0} + 4\log b + \f{3}{2}\log q, \\
\pdv{H_{1}}{b} &= - q \log 6 + \f{4}{b} \le - 6 \log 6 + \f{4}{3} < 0, \\
\pdv{H_{1}}{q} &= \left(1 - b\right) \log 6 + \f{3}{2 q} \le \left(1 - 3\right) \log 6  + \f{3}{2 \cdot 6} < 0,
\end{align}
and

\begin{align}
H_2(b, q) ={} &\left(-\f{b q}{4} + q + \f{3}{2}\right) \log b + \left(-\f{b q}{2} + q +2\right) \log q
+ \left(\f{bq}{2}-q-\f{1}{2}\right)\log 2 \\
&+ \log C_{0} - \f{bq}{8}, \\
\pdv{H_{2}}{b} ={} &\f{q}{8}\left(-2\log b - 4\log q - 3 + 4\log 2+ \f{8}{b}\right) + \f{3}{2b} \\
\le{} &\f{6}{8}\left(-2\log 3 - 4\log 6 - 3 + 4\log 2+ \f{8}{3}\right) + \f{3}{2\cdot 3} < 0, \\
 \pdv{H_{2}}{q} ={} &\left(1-\f{b}{4}\right)\log b + \left(1-\f{1}{2}b\right)\log q + b\left(\f{1}{2}\log 2 - \f{5}{8}\right) - \log 2 + 1 + \f{2}{q}\\
\le{} &\left(1-\f{3}{4}\right)\log 3 + \left(1-\f{1}{2}\cdot 3\right)\log 6 + 3\left(\f{1}{2}\log 2 - \f{5}{8}\right) - \log 2 + 1 + \f{2}{6} < 0.
\end{align}

Thus both $H_{1}$ and $H_{2}$ are strictly decreasing in $b$ and $q$.
Direct computation shows that

\begin{align}
&H_{1}(3, 6) < 0, \\
&H_{2}(3, 6), H_{2}(3, 7), H_{2}(3, 8) > 0, \\
&H_{2}(3, 9), H_{2}(4, 6) < 0.
\end{align}
Therefore, in the case $\lvert D_{r} \rvert \ge 2$ and $\B < g/2$, inequality \eqref{eq:dest3} holds except for
the following pairs:

\begin{align}
\label{eq:exB1}
(b, q) = &(3, 6), (3, 7), (3, 8).
\end{align}

\HL

%%%%%%%%%%%%%%%%%%%%%%%%%%%%%%%%%%%%%%%%
\noindent
(B-2) The Subcase $\B \ge g/2$.

In this situation, we have $R_{0} = \sqrt{g/(2\B)}$ by \eqref{eq:R0}. Since $D_{r}$ contains at least two elements not exceeding $g$, we have $g \ge 2$,
and since $b'$ has at least two divisors, we also have $b' \ge 2$.

Let $p_{0}$ be the smallest prime divisor of $b' = b/d_{0} = bg/m$ and write the prime factorization of $b'$ as

\begin{align}
b' = p_{0}^{e_{0}}\cdots p_{k-1}^{e_{k-1}} \q (k \ge 1,\ p_{0} < \cdots < p_{k-1},\ e_{i} \ge 1).
\end{align}
Then

\begin{align}
D_{r} &= \{ r \mid 1 \le r \le g,\ r \mid b' \} = \{ r = p_{0}^{f_{0}}\cdots p_{k-1}^{f_{k-1}} \mid 0 \le f_{i} \le e_{i},\ r \le g \}.
\end{align}
Fixing $f_{1}, \dotsc, f_{k-1}$, consider the $p_{0}$-adic chain

\begin{align}
\label{eq:chain}
l,\ p_{0}l,\ p_{0}^{2}l,\ \dotsc,\ p_{0}^{e_{0}}l \q (l = p_{1}^{f_{1}}\cdots p_{k-1}^{f_{k-1}}),
\end{align}
then every $r \in D_{r}$ lies in exactly one such chain.

For $\B$ defined by \eqref{eq:ABdef},
the ratio of two adjacent terms within the same chain
corresponding to $r = p_{0}^{f_{0}}l$ is

\begin{align}
\f{(d_{0}p_{0}^{f_{0}+1}l)^{q'p_{0}^{f_{0}+1}l-1}} {(d_{0}p_{0}^{f_{0}}l)^{q'p_{0}^{f_{0}}l-1}}
&= \left(\f{d_{0}p_{0}^{f_{0}+1}l} {d_{0}p_{0}^{f_{0}}l}\right)^{q'p_{0}^{f_{0}}l-1}
(d_{0}p_{0}^{f_{0}+1}l)^{q'(p_{0}-1)p_{0}^{f_{0}}l} \\
&= p_{0}^{q'p_{0}^{f_{0}}l-1} (d_{0}p_{0}^{f_{0}+1}l)^{q'(p_{0}-1)p_{0}^{f_{0}}l} \\
&\ge p_{0}^{q'-1} (d_{0}p_{0})^{q'(p_{0}-1)} = d_{0}^{q'(p_{0}-1)}p_{0}^{q'p_{0}-1} \\
&\ge 2^{2q'-1}d_{0}^{q'}.
\end{align}
Let this final expression be denoted by $\Lambda$.
If $M_{l}$ is the largest value of $(d_{0}r)^{q'r-1}$ within the chain corresponding to $l$,
then the total contribution $\B_{l}$ of that chain satisfies

\begin{align}
\B_{l} < M_{l}\f{\Lambda}{\Lambda-1}.
\end{align}

Let $\tau(b')$ denote the number of divisors of $b'$. Then we have

\begin{align}
\tau(b') = (e_{0} + 1) \cdots (e_{k-1} + 1).
\end{align}
Moreover, it is well known that $\tau(b') \le 2\sqrt{b'}$.

The number $c$ of the chains \eqref{eq:chain} is equal to the number of divisors of $b'$ that are not divisible by $p_{0}$.
Hence

\begin{align}
c = \f{\tau(b')}{e_{0}+1} \le \f{\tau(b')}{2} \le \sqrt{b'}.
\end{align}

Therefore, if $M$ denotes the largest value among the $M_{l}$, then

\begin{align}
\B \le \sum_{l}\B_{l} < \f{\Lambda}{\Lambda-1}\sum_{l}M_{l} \le M\sqrt{b'}\f{\Lambda}{\Lambda-1}.
\end{align}
Since $M = (d_{0}r)^{q'r-1}$ for some $r \in D_{r}$, and since every such $r$ satisfies
$r \le \min\{g, b'\} = \min\{g, bg/m\} = g \min\{1, b/m\}$, we obtain

\begin{align}
M &\le 
\begin{dcases}
(d_{0}g)^{q'g-1} = \left(\f{m}{g}\cdot g\right)^{\f{q}{g}\cdot g - 1} = m^{q-1} & (m \le b) \\
(d_{0}b')^{q'\cdot\f{bg}{m}-1} = \left(d_{0} \cdot \f{b}{d_{0}}\right)^{\f{q}{g}\cdot \f{bg}{m} - 1} = b^{\f{bq}{m}-1} & (m \ge b)
\end{dcases}.
\end{align}
Thus,

\begin{align}
\label{eq:beta-lambda}
\B &<
\begin{dcases}
\f{\Lambda}{\Lambda-1}\sqrt{b'}\,m^{q-1} = \f{\Lambda}{\Lambda-1}\sqrt{\f{bg}{m}}\,m^{q-1} & (m \le b) \\
\f{\Lambda}{\Lambda-1}\sqrt{b'}\,b^{\f{bq}{m}-1} = \f{\Lambda}{\Lambda-1}\sqrt{\f{bg}{m}}\,b^{\f{bq}{m}-1} & (m \ge b)
\end{dcases},
\end{align}
and in \eqref{eq:d02}, for $0 < R \le 1$,

\begin{align}
\log S(b, q, m) &\le \sum_{r\in D_{r}}(d_{0}r)^{q'r-1}R^{r} - g \log R \\
&\le \A R + \B R^{2} - g \log R,
\end{align}
where $\A$ and $\B$ are defined by \eqref{eq:ABdef}.

We can put $R = R_{0} = \sqrt{g/(2\B)}$, and therefore

\begin{align}
\log S(b, q, m) &\le \A R_{0} + \B R_{0}^{2} - g \log R_{0} \\
&= \A\sqrt{\f{g}{2\B}} + \f{g}{2}\left(1+\log\f{2\B}{g}\right) \\
\label{eq:B2lsm}
&= \A\sqrt{\f{g}{2\B}} + \f{g}{2}\left(1+\log 2 + \log \B - \log g\right).
\end{align}
\HL

%%%%%%%%%%%%%%%%%%%%%%%%%%%%%%%%%%%%%%%%
\noindent
(B-2-1) The subcase $m \mid q$.

\noindent
\eg $(b, q, m) = (4, 2, 2), (3, 3, 3), (3, 4, 4), (6, 2, 2), (3, 5, 5)$
\HL

In this case, we have $2 \le g = m \le q$, together with $d_{0} = m/g = 1$ and $b' = b/d_{0} = b$. Then

\begin{align}
&\Lambda = 2^{2q'-1}d_{0}^{q'} = 2^{2q'-1} \ge 2, \\
&\f{\Lambda}{\Lambda-1} \le \f{2}{2-1} = 2, \\
&\A\sqrt{\f{g}{2\B}} = d_{0}^{q'-1}\sqrt{\f{g}{2\B}} = \sqrt{\f{g}{2\B}} \le 1,
\end{align}
and by \eqref{eq:beta-lambda},

\begin{align}
\B &<
\begin{dcases}
2\sqrt{b}\,m^{q-1} & (m \le b) \\
2\sqrt{b}\,b^{\f{bq}{m}-1} & (m \ge b)
\end{dcases}.
\end{align}
Therefore, by \eqref{eq:B2lsm}, for $m \le b$ we have

\begin{align}
\log S(b, q, m) &\le 1 + \f{g}{2}\left(1+\log 2 + \log 2\sqrt{b}\,m^{q-1} - \log g\right) \\
&= 1 + \f{m}{2}\left(1+\log 2 + \log 2\sqrt{b}\,m^{q-1} - \log m\right) \\
&= 1 + \f{m}{2}\left(1 + 2\log 2 + \f{1}{2}\log b + (q-2)\log m\right),
\end{align}
and for $m \ge b$ we have

\begin{align}
\log S(b, q, m) &\le 1 + \f{g}{2}\left(1+\log 2 + \log 2\sqrt{b}\,b^{\f{bq}{m}-1} - \log g\right) \\
&= 1 + \f{m}{2}\left(1+\log 2 + \log 2\sqrt{b}\,b^{\f{bq}{m}-1} - \log m\right) \\
&= 1 + \f{m}{2}\left(1 + 2\log 2 + \left(\f{bq}{m}-\f{1}{2}\right)\log b - \log m\right).
\end{align}

Substituting these estimates into the right-hand side of \eqref{eq:fbqm3} and denoting the
resulting quantity by $G(b, q, m)$, we obtain

\begin{align}
G(b, q, m) ={} &\log C_{0} - m + \left(\f{m}{2}+1\right)\log b + \f{3}{2}\log q + \left(m-q(b-1)+\f{1}{2}\right)\log m \\
&+
\begin{dcases}
1 + \f{m}{2}\left(1 + 2\log 2 + \f{1}{2}\log b + (q-2)\log m\right) & (m \le b) \\
1 + \f{m}{2}\left(1 + 2\log 2 + \left(\f{bq}{m}-\f{1}{2}\right)\log b - \log m\right) & (m \ge b)
\end{dcases},
\end{align}
and

\begin{align}
\pdv[2]{G}{m} =
\begin{dcases}
 \f{q (2b  + m - 2) - 1}{2m^{2}} \ge \f{2 (2\cdot 3  + 2 - 2) - 1}{2m^{2}} > 0 & (m \le b) \\
 \f{2q (b-1)+m - 1}{2m^{2}} \ge \f{2\cdot 2 (3-1)+2 - 1}{2m^{2}}  > 0 & (m \ge b)
 \end{dcases}.
\end{align}

Therefore,  in both ranges $m \le b$ and $m \ge b$, $G$ is convex in $m$. Since $2 \le m \le q$,
we need only evaluate $G$ at
$m = 2$ and $m = q$ when $q < b$, and at $m = 2$, $m = b$, and $m = q$ when $q \ge b$.

Let $H_{1}(b, q) \ceq G(b, q, 2)$, $H_{2}(b, q) \ceq G(b, q, q)$, and $H_{3}(b, q) \ceq G(b, q, b)$. Then

\begin{align}
G(b, q, m) &\le
\begin{dcases}
\max\{H_{1}(b, q), H_{2}(b, q)\} & (q < b) \\
\max\{H_{1}(b, q), H_{2}(b, q), H_{3}(b, q)\} & (q \ge b)
\end{dcases}.
\end{align}
Since $m = 2$ always satisfies $m \le b$, we have

\begin{align}
H_{1}(b, q) &= \log C_{0} + \left(-bq +2q +\f{5}{2}\right)\log 2  + \f{5}{2}\log b + \f{3}{2}\log q, \\
\pdv{H_{1}}{b} &= - q \log 2 + \f{5}{2b} \le - 2 \log 2 + \f{5}{2\cdot 3} < 0, \\
\pdv{H_{1}}{q} &= -(b-2)\log 2 + \f{3}{2q} \le
\begin{dcases}
-(3-2)\log 2 + \f{3}{2\cdot 3} < 0 & (b = 3) \\
-(4-2)\log 2 + \f{3}{2\cdot 2} < 0 & (b \ge 4)
\end{dcases}.
\end{align}
The last inequality holds because $(b, q) \ne (3, 2)$.

When $q < b$, note that $(b, q) \ne (3,2)$ implies $b \ge 4$, and when $b = 4$ the integer $q$ must be even,
hence $q = 2$. Under these assumptions we have

\begin{align}
H_{2}(b, q) &= \log C_{0} + 1 + \left(\log 2 - \f{1}{2}\right)q + \left(\f{3}{4}q+1\right)\log b + \left(-bq+\f{1}{2}q^{2}+q+2\right)\log q, \\
\pdv{H_{2}}{b} &= q\left(\f{3}{4b}-\log q\right) + \f{1}{b} \le 2\left(\f{3}{4\cdot 3}-\log 2\right) + \f{1}{3} < 0, \\
\pdv{H_{2}}{q} &= (- b+q+1) \log q - b + \f{q}{2} + \f{3}{4}\log b + \f{1}{2} + \log 2 + \f{2}{q} \\
&\begin{dcases}
= (- 4+2+1) \log 2 - 4 + \f{2}{2} + \f{3}{4}\log 4 + \f{1}{2} + \log 2 + \f{2}{2} < 0 & (b = 4,\ q = 2) \\
= 0 \log q - b + \f{b-1}{2} + \f{3}{4}\log b + \f{1}{2} + \log 2 + \f{2}{b-1} < 0 & (b \ge 5,\ q = b-1) \\
\le - \log 2 - b + \f{b-2}{2} + \f{3}{4}\log b + \f{1}{2} + \log 2 + \f{2}{2} < 0 & (b \ge 5,\ q \le b-2)
\end{dcases}.
\end{align}

When $q \ge b$, we have

\begin{align}
H_{2}(b, q) &= \log C_{0} + 1 + \left(\log 2- \f{1}{2}\right)q + \left(\f{bq}{2}+\f{1}{4}q+1\right)\log b + \left(-bq+\f{3}{2}q+2\right)\log q, \\
\pdv{H_{2}}{b} &= \f{q}{2}\left(\log b - 2\log q + 1 + \f{1}{2b}\right) + \f{1}{b} \\
&\begin{dcases}
= \f{3}{2}\left(\log 3 - 2\log 3 + 1 + \f{1}{2\cdot 3}\right) + \f{1}{3} > 0 & (b = q = 3) \\
= \f{b}{2}\left(\log b - 2\log b + 1 + \f{1}{2b}\right) + \f{1}{b} < 0 & (4 \le b = q) \\
\le \f{b+1}{2}\left(\log b - 2\log (b+1) + 1 + \f{1}{2b}\right) + \f{1}{b} < 0 & (3 \le b < q)
\end{dcases}, \\
\pdv{H_{2}}{q} &= \left(\f{b}{2}+\f{1}{4}\right)\log b + \left(-b+\f{3}{2}\right)\log q - b + 1 + \log 2 + \f{2}{q} \\
&\le \left(\f{b}{2}+\f{1}{4}\right)\log b + \left(-b+\f{3}{2}\right)\log b - b + 1 + \log 2 + \f{2}{b} \\
&= \left(-\f{b}{2}+\f{7}{4}\right)\log b - b + 1 + \log 2 + \f{2}{b} \\
&\begin{dcases}
= \left(-\f{3}{2}+\f{7}{4}\right)\log 3 - 3 + 1 + \log 2 + \f{2}{3} < 0 & (b = 3) \\
\le \underbrace{\left(-\f{b}{2}+\f{7}{4}\right)\log b}_{<0} - 4 + 1 + \log 2 + \f{2}{4} < 0 & (b \ge 4)
\end{dcases},
\end{align}
and

\begin{align}
H_{3}(b, q) &= \log C_{0} + 1 + \left(\log 2- \f{1}{2}\right)b + \left(-\f{bq}{2}+\f{3}{4}b+q+\f{3}{2}\right)\log b + \f{3}{2}\log q, \\
\pdv{H_{3}}{b} &= \left(\f{3}{4}-\f{1}{2}q\right)\log b + q\left(\f{1}{b}-\f{1}{2}\right) + \f{1}{4} + \log 2 + \f{3}{2b} \\
&\begin{dcases}
= \left(\f{3}{4}-\f{1}{2}\cdot 3\right)\log 3 + 3\left(\f{1}{3}-\f{1}{2}\right) + \f{1}{4} + \log 2 + \f{3}{2\cdot 3} > 0 & (b = q = 3) \\
\le \left(\f{3}{4}-\f{1}{2}\cdot 4\right)\log 3 + 4\left(\f{1}{3}-\f{1}{2}\right) + \f{1}{4} + \log 2 + \f{3}{2\cdot 3} < 0 & (b = 3,\ q > b) \\
\le \left(\f{3}{4}-\f{1}{2}\cdot 4\right)\log 4 + 4\left(\f{1}{4}-\f{1}{2}\right) + \f{1}{4} + \log 2 + \f{3}{2\cdot 4} < 0 & (b \ge 4)
\end{dcases}, \\
\pdv{H_{3}}{q} &= \left(1-\f{b}{2}\right)\log b + \f{3}{2q} \le \left(1-\f{3}{2}\right)\log 3 + \f{3}{2\cdot 3} < 0.
\end{align}
Therefore, except for $H_{2}$ and $H_{3}$ in the case $b = q = 3$, all of $H_{1}$, $H_{2}$, and $H_{3}$
are decreasing in both $b$ and $q$.

Taking into account that $q$ must be even whenever $b$ is even, direct computation gives

\begin{align}
&H_{1}(3, k)\ (3 \le k \le 13), H_{1}(4, 2), H_{1}(4, 4), H_{1}(4, 6), \\
&\hspace{1em}H_{1}(5, 2), H_{1}(5, 3), H_{1}(5, 4), H_{1}(6, 2), H_{1}(7, 2), H_{1}(8, 2)  > 0, \\
&H_{1}(3, 14), H_{1}(4, 8), H_{1}(5, 5), H_{1}(6, 4), H_1(7, 3), H_{1}(9, 2) < 0, \\
&H_{2}(3, k)\ (3 \le k \le 7), H_{2}(4, 2), H_{2}(4, 4), H_{2}(4, 6), \\
&\hspace{1em}H_{2}(5, 2), H_{2}(5, 3), H_{2}(5, 4), H_{2}(5, 5), H_{2}(6, 2), H_{2}(7, 2), H_{2}(8, 2) > 0, \\
&H_{2}(3, 8), H_{2}(4, 8), H_{2}(5, 6), H_{2}(6, 4), H_{2}(7, 3), H_{2}(9, 2) < 0, \\
&H_{3}(3, k)\ (3 \le k \le 20), H_{3}(4, 4), H_{3}(4, 6), H_{3}(4, 8), H_{3}(5, 5) > 0, \\
&H_{3}(3, 21), H_{3}(4, 10), H_{3}(5, 6) < 0.
\end{align}
Therefore, in the case $\lvert D_{r} \rvert \ge 2$, $\B \ge g/2$, and $m \mid q$, the inequality \eqref{eq:dest3} holds
except for the following pairs:

\begin{align}
(b, q) ={} &(3, k)\ (3 \le k \le 20),\ (4, k)\ (k = 2, 4, 6, 8),\ (5, k)\ (2 \le k \le 5), \\
\label{eq:exB21}
&(6, 2), (7, 2), (8, 2).
\end{align}

\HL

%%%%%%%%%%%%%%%%%%%%%%%%%%%%%%%%%%%%%%%%
\noindent
(B-2-2) The subcase $m \nmid q$.

\noindent
\eg $(b, q, m) = (4, 2, 4), (6, 2, 6), (4, 4, 8), (8, 2, 4), (8, 2, 8)$
\HL

In this case, since $\lvert D_{r} \rvert \ge 2$, there exists an integer $r \in D_{r}$ with $2 \le r \le g$ and $r \mid b'$.

Moreover, since $m \nmid q$, we have $g=\gcd(m,q)<m$, and since $g\mid m$, this implies $g\le m/2$.
Therefore $2 \le g \le m/2$, $m \ge 4$, and $d_{0} = m/g \ge 2$.
Furthermore, since $b = d_{0}b' \ge 2b'$ and $b' \ge 2$, we obtain $b \ge 4$.

Since $D_{r}$ contains such an $r \ge 2$, the sum in \eqref{eq:ABdef} contains at least one term with $r \ge 2$, and hence

\begin{align}
\B &= \sum_{r \in D_{r}, r\ge 2}(d_{0}r)^{q'r-1} \ge (2d_{0})^{2q'-1}.
\end{align}
It follows that, in \eqref{eq:B2lsm},

\begin{align}
\A\sqrt{\f{g}{2\B}} &= d_{0}^{q'-1}\sqrt{\f{g}{2\B}} \le d_{0}^{q'-1}\sqrt{\f{g}{2}}(2d_{0})^{-q'+\f{1}{2}}
= 2^{-q'+\f{1}{2}}d_{0}^{-\f{1}{2}}\sqrt{\f{g}{2}} \\
\label{eq:alpha}
&\le (2d_{0})^{-\f{1}{2}}\sqrt{\f{g}{2}} = \sqrt{\f{g}{2m}}\sqrt{\f{g}{2}} = \f{g}{2\sqrt{m}} \le \f{\f{m}{2}}{2\sqrt{m}} = \f{\sqrt{m}}{4}.
\end{align}

Since

\begin{align}
\Lambda &= 2^{2q'-1}d_{0}^{q'} \ge 2^{2\cdot 1-1}\cdot 2^{1} = 4,
\end{align}
by \eqref{eq:beta-lambda}, we obtain

\begin{align}
\label{eq:Bbound}
\B &<
\begin{dcases}
\f{4}{3}\sqrt{\f{bg}{m}}m^{q-1} & (m \le b) \\
\f{4}{3}\sqrt{\f{bg}{m}}b^{\f{bq}{m}-1} & (m \ge b)
\end{dcases}.
\end{align}
\HL

\noindent
(i) The subcase $m \le b$.

By \eqref{eq:B2lsm}, \eqref{eq:alpha}, and \eqref{eq:Bbound}, we obtain the following estimate for $m \le b$:

\begin{align}
\log S&(b, q, m) \le \f{\sqrt{m}}{4} + \f{g}{2}\left(1+\log 2 + \log \f{4}{3}\sqrt{\f{bg}{m}}m^{q-1} - \log g\right) \\
\label{eq:sbqmmleb}
&= \f{\sqrt{m}}{4} + \f{g}{2}\left(1+ 3\log 2 - \log 3 + \f{1}{2} (\log b - \log g) + \left(q-\f{3}{2}\right)\log m\right).
\end{align}

Let the expression in parentheses be denoted by $f_{1}(b, q, m)$. Since $2g \le m \le b$, it follows that

\begin{align}
f_{1}(b, q, m) &= 1+ 3\log 2 - \log 3 + \f{1}{2}\log \f{b}{g} + \left(q-\f{3}{2}\right)\log m \\
&\ge 1+ 3\log 2 - \log 3 + \f{1}{2} \log 2 + \left(2-\f{3}{2}\right)\log 2 > 0.
\end{align}
Therefore, applying the bound $2 \le g \le m/2$ to \eqref{eq:sbqmmleb}, we obtain

\begin{align}
\log S(b, q, m) &\le \f{\sqrt{m}}{4} + \f{m}{4}\left(1+ 3\log 2 - \log 3 + \f{1}{2} (\log b - \log 2) + \left(q-\f{3}{2}\right)\log m\right) \\
\label{eq:sbqmmleb2}
&= \f{\sqrt{m}}{4} + \f{m}{4}\left(1+ \f{5}{2}\log 2 - \log 3 + \f{1}{2}\log b + \left(q-\f{3}{2}\right)\log m\right).
\end{align}
Substituting this into the right-hand side of \eqref{eq:fbqm3} and denoting the resulting expression
by $G(b, q, m)$, we obtain

\begin{align}
G(b, q, m) ={} &\log C_{0} - m +\left(\f{m}{2}+1\right)\log b + \f{3}{2}\log q + \left(m-q(b-1)+\f{1}{2}\right)\log m \\
\label{eq:Gmleb}
&+ \f{\sqrt{m}}{4} + \f{m}{4}\left(1+ \f{5}{2}\log 2 - \log 3 + \f{1}{2}\log b + \left(q-\f{3}{2}\right)\log m\right),
\end{align}
and

\begin{align}
\pdv[2]{G}{m} &= \f{2q+5}{8m} + \f{q(b-1)-\f{1}{2}-\f{1}{16}\sqrt{m}}{m^{2}} \ge \f{2(b-1)-\f{1}{2}-\f{1}{16}\sqrt{b}}{m^{2}} \\
&= \f{\sqrt{b}(2\sqrt{b}-\f{1}{16}) - \f{5}{2}}{m^{2}} \ge \f{\sqrt{4}(2\sqrt{4}-\f{1}{16}) - \f{5}{2}}{m^{2}} > 0.
\end{align}
Thus $G$ is convex in $m$.
Since $4 \le m \le b$, the maximum is attained at one of the boundary points.

Setting $H_{1}(b, q) \ceq G(b, q, 4)$ and $H_{2}(b, q) = G(b, q, b)$, we obtain

\begin{align}
G(b, q, m) \le \max\{ H_{1}(b, q), H_{2}(b, q) \}.
\end{align}

Substituting $m = 4$ into \eqref{eq:Gmleb}, we obtain

\begin{align}
H_{1}(b, q) &= \log C_{0} + \left(-2bq+4q+\f{17}{2}\right)\log 2 - \log 3 - \f{5}{2} + \f{7}{2}\log b + \f{3}{2}\log q, \\
\pdv{H_{1}}{b} &= - 2q\log 2 + \f{7}{2b} \le - 2 \cdot 2 \log 2 + \f{7}{2\cdot 4} < 0, \\
\pdv{H_{1}}{q} &= (-2b+4)\log 2 + \f{3}{2q} \le (-2\cdot 4+4)\log 2 + \f{3}{2\cdot 2} < 0.
\end{align}

Substituting $m = b$ into \eqref{eq:Gmleb}, we obtain

\begin{align}
H_{2}(b, q) ={} &\log C_{0} + \f{\sqrt{b}}{4} - \f{3}{4}b + \f{5}{8}b\log 2 - \f{b}{4}\log 3 + \left(-\f{3}{4}bq+\f{5}{4}b+q+\f{3}{2}\right)\log b \\
&+ \f{3}{2}\log q, \\
\pdv{H_{2}}{b} ={} &\f{-3q+5}{4}\log b + \left(-\f{3}{4}+\f{1}{b}\right)q + \f{5}{8}\log 2 - \f{1}{4}\log 3 + \f{1}{2} + \f{3}{2b} + \f{1}{8\sqrt{b}}\\
\le{} &\f{-3\cdot 2+5}{4}\log 4 + \left(-\f{3}{4}+\f{1}{4}\right)\cdot 2 + \f{5}{8}\log 2 - \f{1}{4}\log 3 + \f{1}{2} + \f{3}{2\cdot 4} + \f{1}{8\sqrt{4}} < 0, \\
\pdv{H_{2}}{q} ={} &\left(1-\f{3}{4}b\right)\log b+\f{3}{2q}
\le \left(1-\f{3}{4}\cdot 4\right)\log 4+\f{3}{2\cdot 2} < 0.
\end{align}
Thus, both $H_{1}$ and $H_{2}$ are decreasing with respect to both $b$ and $q$.

Noting that $q$ must be even whenever $b$ is even, we compute the values as follows:

\begin{align}
&H_{1}(4, 2), H_{1}(5, 2) > 0, \\
&H_{1}(4, 4), H_{1}(5, 3), H_{1}(6, 2) < 0, \\
&H_{2}(k, 2)\ (4 \le k \le 9) > 0, \\
&H_{2}(4, 4), H_{2}(5, 3), H_{2}(10, 2) < 0.
\end{align}

Therefore, we have $G(b,q,m) < 0$ for all admissible $(b, q, m)$, except when

\begin{align}
(b, q) = (k, 2)\ (4 \le k \le 9).
\end{align}

\HL

\noindent
(ii) The subcase $m \ge b$.

For $m \ge b$, we have the bound

\begin{align}
\log S&(b, q, m) \le \f{\sqrt{m}}{4} + \f{g}{2}\left(1+\log 2 + \log \f{4}{3}\sqrt{\f{bg}{m}}b^{\f{bq}{m}-1} - \log g\right) \\
\label{eq:sbqmmgeb}
&= \f{\sqrt{m}}{4} + \f{g}{2}\left(1+ 3\log 2 - \log 3 + \left(\f{bq}{m}-\f{1}{2}\right)\log b - \f{1}{2}(\log m + \log g)\right).
\end{align}
Let the expression in parentheses be denoted by $f_{2}(b, q, m)$. 
Since $f_{2}$ is not necessarily non-negative, we first consider the subcase $f_{2}(b,q,m) < 0$.
\HL

\noindent
(ii-1) The subcase $m \ge b$ and $f_{2} < 0$.

When $f_{2}(b, q, m) < 0$, it follows from \eqref{eq:sbqmmgeb} that

\begin{align}
\log S(b, q, m) &< \f{\sqrt{m}}{4}.
\end{align}
Substituting this into the right-hand side of \eqref{eq:fbqm3} and denoting the resulting expression
by $G(b, q, m)$, we obtain

\begin{align}
G(b, q, m) &= \log C_{0} - m +\left(\f{m}{2}+1\right)\log b + \f{3}{2}\log q + \left(m-q(b-1)+\f{1}{2}\right)\log m + \f{\sqrt{m}}{4},
\end{align}
and

\begin{align}
\pdv[2]{G}{m} &= \f{1}{m} + \f{q(b-1)-\f{1}{2}-\f{1}{16}\sqrt{m}}{m^{2}} \ge \f{q(b-1)-\f{1}{2}-\f{1}{16}\sqrt{\f{bq}{2}}}{m^{2}} \\
&= \f{\sqrt{q}\left(\sqrt{q}(b-1)-\f{1}{16}\sqrt{\f{b}{2}}\right)-\f{1}{2}}{m^{2}}
\ge \f{\sqrt{2}\left(\sqrt{2}(b-1)-\f{1}{16}\sqrt{\f{b}{2}}\right)-\f{1}{2}}{m^{2}} \\
\label{eq:G2m2f2}
&= \f{\sqrt{2b}\left(\sqrt{2b}-\f{1}{16\sqrt{2}}\right)-\sqrt{2} - \f{1}{2}}{m^{2}}
\ge \f{\sqrt{8}\left(\sqrt{8}-\f{1}{16\sqrt{2}}\right)-\sqrt{2} - \f{1}{2}}{m^{2}} > 0.
\end{align}
Thus $G$ is convex in $m$.
Since $b \le m \le bq/2$, the maximum is attained at one of the boundary points.

Setting
$H_{1}(b, q) \ceq G(b, q, b)$ and $H_{2}(b, q) \ceq G(b, q, bq/2)$,
we obtain

\begin{align}
H_{1}(b, q) &= \log C_{0} + \f{\sqrt{b}}{4} - b + \left(-bq + \f{3}{2}b + q + \f{3}{2}\right)\log b + \f{3}{2}\log q, \\
\pdv{H_{1}}{b} &= \left(\f{3}{2}-q\right)\log b + \left(\f{1}{b}-1\right)q + \f{1}{2} + \f{3}{2b} + \f{1}{8\sqrt{b}} \\
&\le \left(\f{3}{2}-2\right)\log 4 + \left(\f{1}{4}-1\right)\cdot 2 + \f{1}{2} + \f{3}{2\cdot 4} + \f{1}{8\sqrt{4}} < 0, \\
\pdv{H_{1}}{q} &= (1-b)\log b + \f{3}{2q} \le (1-4)\log 4 + \f{3}{2\cdot 2} < 0,
\end{align}
and

\begin{align}
H_{2}(b, q) ={} &\log C_{0} + \left(\f{bq}{2}-q-\f{1}{2}\right)\log 2 + \f{\sqrt{2bq}}{8} - \f{bq}{2} + \left(-\f{bq}{4}+q+\f{3}{2}\right)\log b \\
&+ \left(-\f{bq}{2}+q+2\right)\log q, \\
\pdv{H_{2}}{b} ={} &\f{q}{4}\left(-\log b-2\log q-3+2\log 2 + \f{4}{b} + \f{1}{4}\sqrt{\f{2}{bq}}\right) + \f{3}{2b} \\
\le{} &\f{2}{4}\left(-\log 4-2\log 2-3+2\log 2 + \f{4}{4} + \f{1}{4}\sqrt{\f{2}{4\cdot 2}}\right) + \f{3}{2\cdot 4} < 0, \\
\pdv{H_{2}}{q} ={} &\f{b}{4}\left(-\log b - 4 + 2\log 2 + \f{1}{4}\sqrt{\f{2}{bq}} + \f{4\log b}{b}\right) + \left(1 - \f{b}{2}\right)\log q - \log 2 + 1 + \f{2}{q} \\
\le{} &\f{4}{4}\left(-\log 4 - 4 + 2\log 2 + \f{1}{4}\sqrt{\f{2}{4\cdot 2}} + \f{4\log 4}{4}\right) + \left(1 - \f{4}{2}\right)\log 2
- \log 2 + 1 + \f{2}{2} \\
<{} &0.
\end{align}
Thus, both $H_{1}$ and $H_{2}$ are decreasing with respect to both $b$ and $q$.

Noting that $q$ must be even whenever $b$ is even, we compute the values as follows:

\begin{align}
&H_{1}(4, 2) > 0, \\
&H_{1}(4, 4), H_{1}(5, 2) < 0, \\
&H_{2}(4, 2) > 0, \\
&H_{2}(4, 4), H_{2}(5, 2) < 0.
\end{align}

For $(b, q) = (4, 2)$, we have $m = 2, 4$, and for those combinations $f_{2}(b, q, m)$ is always positive.
Hence $(b, q) = (4, 2)$ does not fall into the present subcase, and we conclude that
$G(b, q, m) < 0$ for all $(b, q, m)$ with $m \ge b$ and $f_{2}(b, q, m) < 0$.

\HL

\noindent
(ii-2) The subcase $m \ge b$ and $f_{2} \ge 0$.

Since $f_{2}(b, q, m) \ge 0$, applying the bound $2 \le g \le m/2$ to \eqref{eq:sbqmmgeb}, we obtain

\begin{align}
\log S(b, q, m) &\le \f{\sqrt{m}}{4} + \f{m}{4}\left(1+ 3\log 2 - \log 3 + \left(\f{bq}{m}-\f{1}{2}\right)\log b - \f{1}{2}(\log m + \log 2)\right) \\
&= \f{\sqrt{m}}{4} + \f{m}{4}\left(1+ \f{5}{2}\log 2 - \log 3 + \left(\f{bq}{m}-\f{1}{2}\right)\log b - \f{1}{2}\log m\right).
\end{align}
Substituting this into the right-hand side of \eqref{eq:fbqm3} and denoting the resulting expression
by $G(b, q, m)$, we obtain

\begin{align}
G(b, q, m) ={} &\log C_{0} - m +\left(\f{m}{2}+1\right)\log b + \f{3}{2}\log q + \left(m-q(b-1)+\f{1}{2}\right)\log m \\
\label{eq:gbqmmgeb}
&+ \f{\sqrt{m}}{4} + \f{m}{4}\left(1+ \f{5}{2}\log 2 - \log 3 + \left(\f{bq}{m}-\f{1}{2}\right)\log b - \f{1}{2}\log m\right),
\end{align}
and by the same type of estimate as in \eqref{eq:G2m2f2},

\begin{align}
\pdv[2]{G}{m} =
\f{7}{8m} + \f{q(b-1)-\f{1}{2}-\f{1}{16}\sqrt{m}}{m^{2}} \ge \f{q(b-1)-\f{1}{2}-\f{1}{16}\sqrt{\f{bq}{2}}}{m^{2}} > 0.
\end{align}
Thus $G$ is convex in $m$.
Since $b \le m \le bq/2$, the maximum is attained at one of the boundary points.

Setting $H_{1}(b, q) \ceq G(b, q, b)$ and $H_{2}(b, q) \ceq G(b, q, bq/2)$, we obtain

\begin{align}
G(b, q, m) \le \max\{H_{1}(b, q), H_{2}(b, q) \}.
\end{align}

Substituting $m = b$ into \eqref{eq:gbqmmgeb}, we obtain

\begin{align}
H_{1}(b, q) ={} &\log C_{0} + \f{\sqrt{b}}{4} - \f{3}{4}b + \f{5}{8}b\log 2 - \f{b}{4}\log 3 +
\left(-\f{3}{4}bq+\f{5}{4}b+q+\f{3}{2}\right)\log b + \f{3}{2}\log q, \\
\pdv{H_{1}}{b} ={} &\f{-3q+5}{4}\log b + \left(-\f{3}{4}+\f{1}{b}\right)q + \f{5}{8}\log 2 - \f{1}{4}\log 3 + \f{1}{2} + \f{3}{2b} + \f{1}{8\sqrt{b}}\\
\le{} &\f{-3\cdot 2+5}{4}\log 4 + \left(-\f{3}{4}+\f{1}{4}\right)\cdot 2 + \f{5}{8}\log 2 - \f{1}{4}\log 3 + \f{1}{2} + \f{3}{2\cdot 4} + \f{1}{8\sqrt{4}} < 0, \\
\pdv{H_{1}}{q} ={} &\left(1-\f{3}{4}b\right)\log b+\f{3}{2q}
\le \left(1-\f{3}{4}\cdot 4\right)\log 4+\f{3}{2\cdot 2} < 0.
\end{align}

Substituting $m=bq/2$ into \eqref{eq:gbqmmgeb}, we obtain

\begin{align}
H_{2}(b, q) ={} &\log C_{0} + \left(\f{7}{8}bq-q-\f{1}{2}\right)\log 2 - \f{bq}{8}\log 3 - \f{3}{8}bq + \f{\sqrt{2bq}}{8} \\
& + \left(-\f{bq}{8}+q+\f{3}{2}\right)\log b + \left(-\f{9}{16}bq+q+2\right)\log q, \\
\pdv{H_{2}}{b} ={} &\f{q}{16}\left(-2\log b-9\log q-8+14\log 2-2\log 3+\f{16}{b}+\sqrt{\f{2}{bq}}\right)+\f{3}{2b} \\
&\le \f{2}{16}\left(-2\log 4-9\log 2-8+14\log 2-2\log 3+\f{16}{4}+\sqrt{\f{2}{4\cdot 2}}\right)+\f{3}{2\cdot 4} < 0, \\
\pdv{H_{2}}{q} ={} &\f{b}{16}\left(-2\log b -15 + 14\log 2-2\log 3\ + \sqrt{\f{2}{bq}} + \f{16\log b}{b}\right) \\
&+ 1 - \log 2+ \f{2}{q} + \left(-\f{9}{16}b+1\right)\log q \\
\le{} &\f{4}{16}\left(-2\log 4 -15 + 14\log 2-2\log 3\ + \sqrt{\f{2}{4\cdot 2}} + \f{16\log 4}{4}\right) \\
&+ 1 - \log 2+ \f{2}{2} + \left(-\f{9}{16}\cdot 4+1\right)\log 2 < 0.
\end{align}
Thus, both $H_{1}$ and $H_{2}$ are decreasing with respect to both $b$ and $q$.

Noting that $q$ must be even whenever $b$ is even, we compute the values as follows:

\begin{align}
&H_{1}(k, 2)\ (4 \le k \le 9) > 0, \\
&H_{1}(4, 4), H_{1}(5, 3), H_{1}(10, 2) < 0, \\
&H_{2}(k, 2)\ (4 \le k \le 9), H_{2}(4, 4), H_{2}(5, 3) > 0, \\
&H_{2}(4, 6), H_{2}(5, 4), H_{2}(7, 3), H_{2}(10, 2) < 0.
\end{align}

Therefore, we have $G(b, q, m) < 0$ except for

\begin{align}
(b, q) = (k, 2)\ (4 \le k \le 9), (4, 4), (5, 3).
\end{align}

Combining the results of (i), (ii-1), and (ii-2), when $\lvert D_{r} \rvert \ge 2$, $\B \ge g/2$, and $m \nmid q$,
the inequality \eqref{eq:dest3} holds for all pairs except the following:

\begin{align}
\label{eq:exB22}
(b, q) = (k, 2)\ (4 \le k \le 9), (4, 4), (5, 3).
\end{align}

\OL

By combining all the five cases \eqref{eq:exA1}, \eqref{eq:exA2}, \eqref{eq:exB1}, \eqref{eq:exB21}, and \eqref{eq:exB22},
we have shown that \eqref{eq:dest3} holds for all pairs $(b, q)$ except the following:

\begin{align}
&(3, k)\q (2 \le k \le 20), \\
&(4, 2), (4, 4), (4, 6), (4, 8), \\
&(5, 2), (5, 3), (5, 4), (5, 5), \\
&(6, 2), (7, 2), (8, 2), (9, 2).
\end{align}

For each of these exceptional pairs, direct computation shows that there exists
a permutation $y$ such that, with $x = \sigman$ fixed,
the automorphism group is trivial; see \tabref{tab:exceptionbge3-1} and \tabref{tab:exceptionbge3-2}.

When $(b, q) = (3, 2)$, the monodromy group $G = \langle x, y \rangle$ has order~120 and
$C_{S_n}(G) \cong \{1\}$.
For all remaining $(b, q)$ pairs, the group $\langle x, y \rangle$ is equal to $S_n$ or $A_n$.
Since $C_{S_n}(S_n) \cong C_{S_n}(A_n) \cong \{1\}$,
the automorphism group is trivial for each such pair $(x, y)$.

Thus, we conclude that every passport $[n, b^{q}, n]$ of genus~$\ge 2$ with
$b \ge 3$ and $q \ge 2$ admits at least one dessin with a trivial automorphism group.
\end{proof}
\HL

%%%%%%%%%%%%
%%% table-bge3-1.tex: begin
\begin{table}[htbp]
  \centering
  \scriptsize
  \begin{tabular}{|c|c|c|c|c|p{11.1cm}|}
    \hline
    \multirow{2}{*}{$b$} &
    \multirow{2}{*}{$q$} &
    \multicolumn{4}{|l|}{$y$} \\ \cline{3-6}
    & & $G$ & $w(x,y)$ & $p_{w}$ & $w$ \\ \hline
    \multirow{2}{*}{$3$} &
    \multirow{2}{*}{$2$} &
    \multicolumn{4}{|p{7.5cm}|}{$(1\ 2\ 4)(3\ 5\ 6)$} \\ \cline{3-6}
     & & $120$ & & & \\ \hline
    \multirow{2}{*}{$3$} &
    \multirow{2}{*}{$3$} &
    \multicolumn{4}{|p{7.5cm}|}{$(1\ 2\ 4)(3\ 6\ 8)(5\ 9\ 7)$} \\ \cline{3-6}
     & & $A_{9}$ & $x^{3}yxy^{2}x^{2}y^{2}$ & $5$ & $(1\ 4\ 9\ 5\ 2)$ \\ \hline
    \multirow{2}{*}{$3$} &
    \multirow{2}{*}{$4$} &
    \multicolumn{4}{|p{7.5cm}|}{$(1\ 10\ 11)(2\ 4\ 7)(3\ 5\ 9)(6\ 12\ 8)$} \\ \cline{3-6}
     & & $S_{12}$ & $yxyx^{2}yxyx^{2}$ & $7$ & $(1\ 9\ 10\ 11\ 12\ 4\ 3)$ \\ \hline
    \multirow{2}{*}{$3$} &
    \multirow{2}{*}{$5$} &
    \multicolumn{4}{|p{7.5cm}|}{$(1\ 3\ 6)(2\ 10\ 8)(4\ 7\ 9)(5\ 11\ 15)(12\ 13\ 14)$} \\ \cline{3-6}
     & & $A_{15}$ & $x^{2}yxy^{2}x^{2}y^{2}x$ & $11$ & $(1\ 8\ 4\ 6\ 9\ 10\ 13\ 15\ 2\ 12\ 14)$ \\ \hline
    \multirow{2}{*}{$3$} &
    \multirow{2}{*}{$6$} &
    \multicolumn{4}{|p{7.5cm}|}{$(1\ 13\ 10)(2\ 7\ 4)(3\ 16\ 9)(5\ 15\ 11)(6\ 14\ 17)(8\ 12\ 18)$} \\ \cline{3-6}
     & & $S_{18}$ & $x^{4}yxyx^{2}y^{2}x^{3}$ & $13$ & $(1\ 6\ 7\ 4\ 2\ 9\ 11\ 14\ 5\ 16\ 17\ 15\ 12)$ \\ \hline
    \multirow{2}{*}{$3$} &
    \multirow{2}{*}{$7$} &
    \multicolumn{4}{|p{7.5cm}|}{$(1\ 13\ 16)(2\ 6\ 7)(3\ 4\ 5)(8\ 20\ 18)(9\ 19\ 21)(10\ 12\ 17)(11\ 14\ 15)$} \\ \cline{3-6}
     & & $A_{21}$ & $xy^{2}x^{2}yx^{2}y$ & $17$ & $(1\ 2\ 17\ 10\ 13\ 18\ 12\ 8\ 15\ 21\ 14\ 11\ 6\ 20\ 9\ 16\ 3)$ \\ \hline
    \multirow{2}{*}{$3$} &
    \multirow{2}{*}{$8$} &
    \multicolumn{4}{|p{7.5cm}|}{$(1\ 17\ 22)(2\ 8\ 24)(3\ 15\ 23)(4\ 6\ 20)(5\ 12\ 18)(7\ 21\ 9)(10\ 14\ 16)(11\ 19\ 13)$} \\ \cline{3-6}
     & & $S_{24}$ & $xyx^{2}yx^{2}y^{2}xyxyx$ & $19$ & $(2\ 18\ 23\ 19\ 20\ 17\ 8\ 7\ 6\ 15\ 9\ 11\ 5\ 22\ 10\ 21\ 13\ 3\ 24)$ \\ \hline
    \multirow{2}{*}{$3$} &
    \multirow{2}{*}{$9$} &
    \multicolumn{4}{|p{7.5cm}|}{$(1\ 26\ 24)(2\ 9\ 22)(3\ 17\ 15)(4\ 12\ 27)(5\ 20\ 8)(6\ 21\ 19)(7\ 13\ 10)(11\ 18\ 14)(16\ 23\ 25)$} \\ \cline{3-6}
     & & $A_{27}$ & $x^{2}yx^{2}yx^{2}y$ & $17$ & $(2\ 10\ 3\ 7\ 22\ 13\ 11\ 9\ 19\ 15\ 4\ 12\ 20\ 24\ 8\ 5\ 14)$ \\ \hline
    \multirow{2}{*}{$3$} &
    \multirow{2}{*}{$10$} &
    \multicolumn{4}{|p{7.5cm}|}{$(1\ 13\ 30)(2\ 8\ 29)(3\ 5\ 7)(4\ 28\ 12)(6\ 21\ 14)(9\ 19\ 24)(10\ 22\ 23)(11\ 16\ 26)(15\ 27\ 20)(17\ 18\ 25)$} \\ \cline{3-6}
     & & $S_{30}$ & $yx^{2}y^{2}xyx^{2}y^{2}x$ & $23$ & $(1\ 10\ 5\ 18\ 20\ 27\ 4\ 28\ 9\ 24\ 30\ 11\ 13\ 12\ 26\ 16\ 29\ 21\ 7\ 22\ 17\ 15\ 8)$ \\ \hline
    \multirow{2}{*}{$3$} &
    \multirow{2}{*}{$11$} &
    \multicolumn{4}{|p{7.5cm}|}{$(1\ 24\ 12)(2\ 27\ 32)(3\ 20\ 31)(4\ 29\ 26)(5\ 19\ 14)(6\ 33\ 10)(7\ 15\ 11)(8\ 21\ 22)(9\ 18\ 28)(13\ 16\ 23)(17\ 25\ 30)$} \\ \cline{3-6}
     & & $A_{33}$ & $x^{2}y^{2}x^{2}yx^{3}y^{2}x$ & $29$ & $(1\ 6\ 13\ 2\ 31\ 7\ 29\ 17\ 33\ 25\ 28\ 32\ 5\ 11\ 14\ 24\ 8\ 16\ 26\ 10\ 27\ 3\ 20\ 30\ 9\ 23\ 22\ 15\ 19)$ \\ \hline
    \multirow{2}{*}{$3$} &
    \multirow{2}{*}{$12$} &
    \multicolumn{4}{|p{7.5cm}|}{$(1\ 15\ 12)(2\ 29\ 25)(3\ 17\ 18)(4\ 26\ 35)(5\ 16\ 13)(6\ 34\ 8)(7\ 30\ 27)(9\ 31\ 14)(10\ 24\ 33)(11\ 20\ 23)(19\ 21\ 36)(22\ 28\ 32)$} \\ \cline{3-6}
     & & $S_{36}$ & $xyxyx^{6}y^{2}$ & $31$ & $(1\ 27\ 24\ 10\ 4\ 19\ 5\ 29\ 31\ 6\ 25\ 17\ 23\ 20\ 22\ 28\ 11\ 36\ 7\ 21\ 18\ 2\ 13\ 26\ 3\ 9\ 34\ 30\ 35\ 12\ 16)$ \\ \hline
    \multirow{2}{*}{$3$} &
    \multirow{2}{*}{$13$} &
    \multicolumn{4}{|p{7.5cm}|}{$(1\ 26\ 18)(2\ 12\ 5)(3\ 11\ 15)(4\ 30\ 31)(6\ 21\ 13)(7\ 19\ 14)(8\ 32\ 34)(9\ 17\ 25)(10\ 37\ 29)(16\ 24\ 33)(20\ 27\ 35)(22\ 28\ 39)\allowbreak(23\ 36\ 38)$} \\ \cline{3-6}
     & & $A_{39}$ & $xy^{2}x^{2}y^{2}x^{2}y^{2}x$ & $31$ & $(1\ 34\ 29\ 12\ 19\ 3\ 2\ 4\ 7\ 18\ 21\ 15\ 38\ 14\ 39\ 11\ 25\ 10\ 20\ 24\ 26\ 6\ 28\ 5\ 22\ 36\ 32\ 16\ 13\ 35\ 8)$ \\ \hline
    \multirow{2}{*}{$3$} &
    \multirow{2}{*}{$14$} &
    \multicolumn{4}{|p{7.5cm}|}{$(1\ 41\ 22)(2\ 40\ 4)(3\ 16\ 33)(5\ 31\ 10)(6\ 35\ 21)(7\ 36\ 13)(8\ 30\ 12)(9\ 20\ 15)(11\ 24\ 17)(14\ 23\ 37)(18\ 25\ 42)(19\ 39\ 26)\allowbreak(27\ 28\ 29)(32\ 34\ 38)$} \\ \cline{3-6}
     & & $S_{42}$ & $xy^{2}xyx^{2}y^{2}x^{2}y$ & $37$ & $(2\ 29\ 37\ 39\ 28\ 24\ 9\ 26\ 21\ 12\ 41\ 13\ 31\ 7\ 38\ 20\ 10\ 32\ 36\ 5\ 40\ 35\ 33\ 16\ 6\ 23\ 14\ 4\ 27\ 22\allowbreak42\ 19\ 25\ 8\ 11\ 15\ 3)$ \\ \hline
    \multirow{2}{*}{$3$} &
    \multirow{2}{*}{$15$} &
    \multicolumn{4}{|p{7.5cm}|}{$(1\ 40\ 5)(2\ 29\ 6)(3\ 19\ 9)(4\ 21\ 27)(7\ 26\ 45)(8\ 28\ 34)(10\ 16\ 33)(11\ 23\ 41)(12\ 20\ 25)(13\ 37\ 24)(14\ 15\ 18)(17\ 38\ 42)\allowbreak(22\ 35\ 30)(31\ 32\ 36)(39\ 43\ 44)$} \\ \cline{3-6}
     & & $A_{45}$ & $y^{2}x^{2}y^{2}x^{2}$ & $41$ & $(1\ 41\ 23\ 30\ 16\ 42\ 26\ 33\ 37\ 5\ 6\ 32\ 35\ 14\ 25\ 11\ 7\ 4\ 36\ 3\ 38\ 9\ 39\ 20\ 31\ 15\ 43\ 8\ 22\ 44\ 45\allowbreak34\ 28\ 13\ 10\ 21\ 24\ 19\ 29\ 17\ 40)$ \\ \hline
    \multirow{2}{*}{$3$} &
    \multirow{2}{*}{$16$} &
    \multicolumn{4}{|p{7.5cm}|}{$(1\ 4\ 14)(2\ 33\ 20)(3\ 8\ 22)(5\ 24\ 46)(6\ 36\ 32)(7\ 47\ 16)(9\ 27\ 41)(10\ 34\ 30)(11\ 12\ 23)(13\ 45\ 37)(15\ 48\ 35)(17\ 43\ 26)\allowbreak(18\ 39\ 28)(19\ 44\ 42)(21\ 38\ 29)(25\ 40\ 31)$} \\ \cline{3-6}
     & & $S_{48}$ & $yxyxy^{2}xyxyxy^{2}x$ & $43$ & $(1\ 21\ 41\ 7\ 4\ 8\ 23\ 40\ 33\ 9\ 3\ 34\ 47\ 5\ 37\ 20\ 10\ 19\ 26\ 32\ 6\ 31\ 28\ 18\ 29\ 48\ 43\ 2\ 13\ 44\ 45\allowbreak17\ 30\ 24\ 22\ 38\ 14\ 12\ 16\ 25\ 15\ 27\ 42)$ \\ \hline
    \multirow{2}{*}{$3$} &
    \multirow{2}{*}{$17$} &
    \multicolumn{4}{|p{7.5cm}|}{$(1\ 40\ 23)(2\ 32\ 9)(3\ 24\ 22)(4\ 46\ 34)(5\ 31\ 11)(6\ 12\ 18)(7\ 51\ 39)(8\ 33\ 16)(10\ 21\ 29)(13\ 30\ 41)(14\ 25\ 19)(15\ 50\ 17)\allowbreak(20\ 43\ 47)(26\ 27\ 36)(28\ 42\ 48)(35\ 38\ 49)(37\ 44\ 45)$} \\ \cline{3-6}
     & & $A_{51}$ & $xy^{2}x^{3}yxyxyx$ & $47$ & $(1\ 13\ 45\ 30\ 51\ 49\ 23\ 50\ 15\ 27\ 6\ 34\ 28\ 18\ 29\ 36\ 24\ 43\ 31\ 19\ 46\ 38\ 40\ 14\ 37\ 33\ 7\ 47\ 17\allowbreak21\ 41\ 22\ 26\ 2\ 5\ 20\ 8\ 25\ 16\ 12\ 11\ 10\ 9\ 39\ 35\ 44\ 48)$ \\ \hline
    \multirow{2}{*}{$3$} &
    \multirow{2}{*}{$18$} &
    \multicolumn{4}{|p{7.5cm}|}{$(1\ 4\ 28)(2\ 20\ 23)(3\ 25\ 11)(5\ 42\ 54)(6\ 7\ 38)(8\ 45\ 47)(9\ 34\ 21)(10\ 52\ 40)(12\ 32\ 29)(13\ 37\ 35)(14\ 26\ 27)(15\ 44\ 17)\allowbreak(16\ 51\ 22)(18\ 19\ 43)(24\ 53\ 36)(30\ 41\ 50)(31\ 33\ 48)(39\ 46\ 49)$} \\ \cline{3-6}
     & & $S_{54}$ & $yxy^{2}xy^{2}xy^{2}xy$ & $47$ & $(1\ 48\ 5\ 32\ 47\ 3\ 17\ 13\ 21\ 39\ 36\ 53\ 42\ 29\ 18\ 41\ 34\ 27\ 51\ 6\ 52\ 8\ 11\ 35\ 4\ 37\ 23\ 22\ 16\ 44\allowbreak2\ 45\ 50\ 49\ 28\ 26\ 12\ 24\ 25\ 15\ 40\ 19\ 9\ 31\ 33\ 10\ 20)$ \\ \hline
    \multirow{2}{*}{$3$} &
    \multirow{2}{*}{$19$} &
    \multicolumn{4}{|p{7.5cm}|}{$(1\ 33\ 50)(2\ 20\ 34)(3\ 36\ 19)(4\ 56\ 25)(5\ 32\ 38)(6\ 22\ 30)(7\ 15\ 57)(8\ 53\ 41)(9\ 12\ 47)(10\ 39\ 14)(11\ 55\ 52)(13\ 26\ 45)\allowbreak(16\ 43\ 28)(17\ 44\ 24)(18\ 49\ 37)(21\ 27\ 29)(23\ 42\ 40)(31\ 35\ 46)(48\ 54\ 51)$} \\ \cline{3-6}
     & & $A_{57}$ & $x^{3}yx^{5}yxy^{2}$ & $53$ & $(1\ 44\ 54\ 43\ 40\ 53\ 47\ 38\ 55\ 14\ 19\ 45\ 5\ 6\ 26\ 3\ 17\ 15\ 36\ 2\ 51\ 10\ 35\ 31\ 13\ 22\ 37\ 25\ 50\ 18\allowbreak42\ 33\ 7\ 8\ 16\ 48\ 46\ 20\ 11\ 34\ 41\ 28\ 24\ 52\ 9\ 23\ 29\ 30\ 12\ 27\ 49\ 56\ 21)$ \\ \hline
    \multirow{2}{*}{$3$} &
    \multirow{2}{*}{$20$} &
    \multicolumn{4}{|p{7.5cm}|}{$(1\ 6\ 35)(2\ 8\ 17)(3\ 46\ 54)(4\ 14\ 29)(5\ 50\ 48)(7\ 51\ 52)(9\ 28\ 26)(10\ 55\ 56)(11\ 34\ 59)(12\ 24\ 36)(13\ 57\ 37)(15\ 42\ 19)\allowbreak(16\ 23\ 60)(18\ 21\ 31)(20\ 22\ 45)(25\ 53\ 33)(27\ 32\ 39)(30\ 49\ 40)(38\ 47\ 41)(43\ 44\ 58)$} \\ \cline{3-6}
     & & $S_{60}$ & $yxy^{2}xy^{2}xy^{2}xy$ & $53$ & $(1\ 4\ 14\ 52\ 34\ 47\ 12\ 13\ 2\ 48\ 39\ 29\ 45\ 11\ 9\ 54\ 5\ 56\ 57\ 6\ 24\ 60\ 30\ 21\ 37\ 18\ 44\ 33\ 22\ 38\allowbreak46\ 36\ 23\ 59\ 20\ 32\ 8\ 55\ 40\ 15\ 51\ 42\ 49\ 17\ 16\ 19\ 53\ 43\ 28\ 35\ 26\ 50\ 31)$ \\ \hline
    \multirow{2}{*}{$4$} &
    \multirow{2}{*}{$2$} &
    \multicolumn{4}{|p{7.5cm}|}{$(1\ 6\ 2\ 7)(3\ 4\ 8\ 5)$} \\ \cline{3-6}
     & & $S_{8}$ & $y^{2}xy^{2}x$ & $3$ & $(3\ 7\ 5)$ \\ \hline
    \multirow{2}{*}{$4$} &
    \multirow{2}{*}{$4$} &
    \multicolumn{4}{|p{7.5cm}|}{$(1\ 4\ 8\ 15)(2\ 3\ 10\ 16)(5\ 14\ 7\ 12)(6\ 13\ 11\ 9)$} \\ \cline{3-6}
     & & $S_{16}$ & $y^{3}xyx$ & $13$ & $(2\ 13\ 4\ 8\ 14\ 16\ 12\ 7\ 10\ 3\ 11\ 9\ 15)$ \\ \hline
    \multirow{2}{*}{$4$} &
    \multirow{2}{*}{$6$} &
    \multicolumn{4}{|p{7.5cm}|}{$(1\ 20\ 11\ 6)(2\ 24\ 15\ 13)(3\ 23\ 18\ 22)(4\ 17\ 7\ 8)(5\ 10\ 16\ 14)(9\ 19\ 12\ 21)$} \\ \cline{3-6}
     & & $S_{24}$ & $x^{2}y^{2}x^{4}y^{2}x^{2}$ & $19$ & $(2\ 3\ 8\ 5\ 19\ 24\ 23\ 4\ 15\ 11\ 9\ 18\ 16\ 6\ 21\ 17\ 13\ 22\ 10)$ \\ \hline
    \multirow{2}{*}{$4$} &
    \multirow{2}{*}{$8$} &
    \multicolumn{4}{|p{7.5cm}|}{$(1\ 29\ 8\ 28)(2\ 15\ 26\ 14)(3\ 32\ 9\ 13)(4\ 22\ 31\ 18)(5\ 27\ 6\ 24)(7\ 17\ 19\ 12)(10\ 25\ 30\ 20)(11\ 23\ 21\ 16)$} \\ \cline{3-6}
     & & $S_{32}$ & $x^{2}y^{2}x^{3}y^{2}x$ & $29$ & $(1\ 30\ 21\ 13\ 24\ 18\ 32\ 23\ 12\ 11\ 27\ 15\ 4\ 5\ 3\ 28\ 6\ 20\ 17\ 22\ 9\ 10\ 29\ 2\ 19\ 31\ 25\ 8\ 7)$ \\ \hline
  \end{tabular} \\
  \vspace{1\baselineskip}
  \caption{Elements $y$ such that $\AD \cong \{ 1 \}$ for $b = 3, 4$ ($x = (1\ 2\ \ldots\ n)$)}\label{tab:exceptionbge3-1}
\end{table}
%%% table-bge3-1.tex: end
%%%%%%%%%%%%

%%%%%%%%%%%%
%%% table-bge3-2.tex: begin
\begin{table}[htbp]
  \centering
  \scriptsize
  \begin{tabular}{|c|c|c|c|c|p{11.1cm}|}
    \hline
    \multirow{2}{*}{$b$} &
    \multirow{2}{*}{$q$} &
    \multicolumn{4}{|l|}{$y$} \\ \cline{3-6}
    & & $G$ & $w(x,y)$ & $p_{w}$ & $w$ \\ \hline
    \multirow{2}{*}{$5$} &
    \multirow{2}{*}{$2$} &
    \multicolumn{4}{|p{7.5cm}|}{$(1\ 2\ 10\ 8\ 5)(3\ 7\ 4\ 9\ 6)$} \\ \cline{3-6}
     & & $S_{10}$ & $y^{2}xy^{2}x$ & $7$ & $(1\ 3\ 9\ 7\ 8\ 6\ 5)$ \\ \hline
    \multirow{2}{*}{$5$} &
    \multirow{2}{*}{$3$} &
    \multicolumn{4}{|p{7.5cm}|}{$(1\ 11\ 12\ 5\ 8)(2\ 9\ 10\ 3\ 4)(6\ 14\ 7\ 13\ 15)$} \\ \cline{3-6}
     & & $A_{15}$ & $y^{2}x^{2}y^{3}x$ & $11$ & $(2\ 5\ 14\ 3\ 8\ 7\ 13\ 10\ 4\ 6\ 11)$ \\ \hline
    \multirow{2}{*}{$5$} &
    \multirow{2}{*}{$4$} &
    \multicolumn{4}{|p{7.5cm}|}{$(1\ 8\ 2\ 16\ 13)(3\ 19\ 17\ 10\ 20)(4\ 11\ 15\ 7\ 12)(5\ 6\ 9\ 14\ 18)$} \\ \cline{3-6}
     & & $S_{20}$ & $xyx^{2}y^{3}x$ & $17$ & $(1\ 20\ 6\ 2\ 5\ 4\ 14\ 10\ 19\ 18\ 17\ 16\ 7\ 8\ 13\ 3\ 15)$ \\ \hline
    \multirow{2}{*}{$5$} &
    \multirow{2}{*}{$5$} &
    \multicolumn{4}{|p{7.5cm}|}{$(1\ 8\ 19\ 24\ 21)(2\ 13\ 7\ 12\ 17)(3\ 6\ 18\ 4\ 9)(5\ 15\ 22\ 25\ 11)(10\ 16\ 23\ 20\ 14)$} \\ \cline{3-6}
     & & $A_{25}$ & $xy^{2}xy^{2}xy^{2}xy^{2}$ & $17$ & $(1\ 21\ 10\ 25\ 15\ 16\ 18\ 8\ 23\ 11\ 7\ 2\ 5\ 20\ 9\ 24\ 6)$ \\ \hline
    \multirow{2}{*}{$6$} &
    \multirow{2}{*}{$2$} &
    \multicolumn{4}{|p{7.5cm}|}{$(1\ 9\ 5\ 7\ 10\ 6)(2\ 3\ 4\ 12\ 8\ 11)$} \\ \cline{3-6}
     & & $S_{12}$ & $x^{3}yx^{4}yx$ & $3$ & $(4\ 5\ 10)$ \\ \hline
    \multirow{2}{*}{$7$} &
    \multirow{2}{*}{$2$} &
    \multicolumn{4}{|p{7.5cm}|}{$(1\ 4\ 14\ 11\ 12\ 6\ 9)(2\ 10\ 3\ 8\ 5\ 13\ 7)$} \\ \cline{3-6}
     & & $S_{14}$ & $y^{2}xy^{2}x$ & $11$ & $(1\ 11\ 8\ 7\ 12\ 5\ 3\ 9\ 4\ 13\ 2)$ \\ \hline
    \multirow{2}{*}{$8$} &
    \multirow{2}{*}{$2$} &
    \multicolumn{4}{|p{7.5cm}|}{$(1\ 14\ 4\ 9\ 11\ 15\ 8\ 13)(2\ 7\ 10\ 5\ 16\ 12\ 3\ 6)$} \\ \cline{3-6}
     & & $S_{16}$ & $yx^{2}y^{4}$ & $13$ & $(2\ 7\ 4\ 5\ 13\ 15\ 12\ 11\ 6\ 10\ 16\ 9\ 8)$ \\ \hline
    \multirow{2}{*}{$9$} &
    \multirow{2}{*}{$2$} &
    \multicolumn{4}{|p{7.5cm}|}{$(1\ 10\ 3\ 15\ 9\ 5\ 2\ 12\ 17)(4\ 11\ 13\ 18\ 16\ 14\ 7\ 8\ 6)$} \\ \cline{3-6}
     & & $S_{18}$ & $x^{3}yx^{2}y^{2}x^{2}yx$ & $13$ & $(1\ 6\ 4\ 12\ 13\ 14\ 15\ 8\ 9\ 10\ 11\ 5\ 3)$ \\ \hline
  \end{tabular} \\
  \vspace{1\baselineskip}
  \caption{Elements $y$ such that $\AD \cong \{ 1 \}$ for $b \ge 5$ ($x = (1\ 2\ \ldots\ n)$)}\label{tab:exceptionbge3-2}
\end{table}
%%% table-bge3-2.tex: end
%%%%%%%%%%%%

To find elements $y$ such that $\langle x, y \rangle = S_{n}$ or $A_{n}$ for $x = \sigman$,
we applied \thmref{thm:pn-3}, as in the case of $b = 2$.

For a permutation $y$ of cycle type $(b^{q})$, assume that $G = \langle x, y \rangle$ has no blocks
with respect to any residue class of a divisor $m$ of $n = bq$ with $2 \le m < n$, and that $G$ contains
an element $w$ of cycle type $(p_{w})$ for a prime $p_{w} \le n - 3$.
Then $G = S_{n}$ when $n$ is even, and $G = A_{n}$ when $n$ is odd.

A list of such examples of $y$ and $w$ obtained by computation is shown in \tabref{tab:exceptionbge3-1}
and \tabref{tab:exceptionbge3-2}.
As in \tabref{tab:exceptionbeq2}, we completed this table by generating random permutations
$y$ of cycle type $(b^{q})$ and searching for an element $w$ satisfying the required conditions
among the elements of $\langle x, y \rangle$.

When $(b, q) = (3, 2)$, the monodromy group never becomes $S_{6}$.
Instead, it can be a group of order~$120$, in which case the automorphism group is trivial.
A dessin realizing this case is shown in \figref{fig:g2-120}. This is the same dessin as the one
shown in the lower right of \figref{fig:genus2-n6} in Section~\ref{sec:uniform}.

\begin{figure}[htbp]
\centering
\includegraphics[width=80mm]{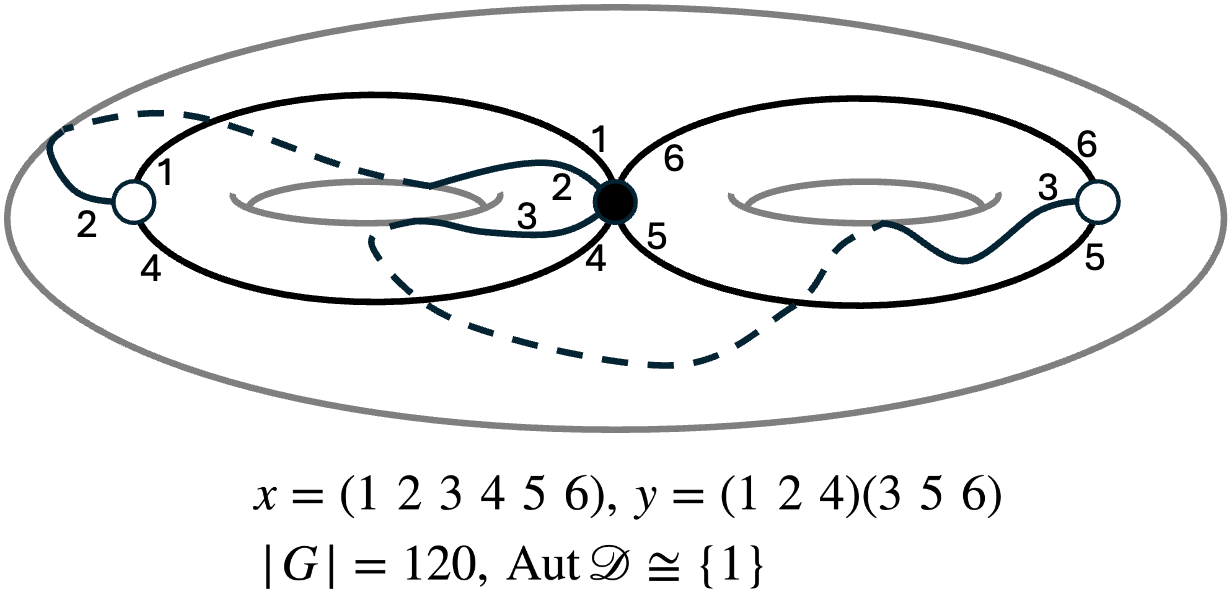} 
\caption{A dessin with trivial automorphism group having passport $[6,3^{2},6]$}
\label{fig:g2-120}
\end{figure}

%%%%%%%%%%%%%%%%%%%%%%%%%%%%%%%%%%%%
\subsection{Passports $[n, b^{q}, n]$}

Combining the results of this section, we obtain the following theorem.

\begin{thm}
\label{thm:gge2}
\thmggetwo
\end{thm}

\begin{proof}
Since the automorphism group is symmetric with respect to black vertices, white vertices, and faces,
it suffices to prove the statement for the passport $[n, b^{q}, n]$ ($n = bq$).

By \eqref{eq:genus}, the genus is

\begin{align}
\f{n - (1 + q + 1)}{2} + 1 = \f{n - q}{2} = \f{q(b-1)}{2} \ge 2.
\end{align}
Therefore, we have $b \ge 2$.

The case $b = n$ ($q = 1$) was proved in \propref{prop:nnn}.

The case $b = 2$ was proved in \propref{prop:beq2}.

The case $3 \le b < n \ (q \ge 2)$ was proved in \propref{prop:bge3}.
\end{proof}

%%%%%%%%%%%%%%%%%%%%%%%%%%%%%%%%%%%%%%%%%%%%%%%%
\HL
\FloatBarrier
\subsection*{Acknowledgements}

I would like to express my sincere gratitude to Associate Professor Yasuhiro Wakabayashi for his detailed guidance
and insightful advice throughout this research.

I am also deeply grateful to all the members of the Wakabayashi Laboratory for their valuable discussions and support.

\HL

\bibliographystyle{amsalpha}
\bibliography{References}

\end{document}